\documentclass[reqno,11pt,a4paper]{amsart}

\usepackage[english]{babel}
\usepackage{mathtools}    
\usepackage[dvipsnames]{xcolor}
\usepackage{newtxtext,newtxmath}  
\usepackage{enumitem}     
\usepackage{esint}        
\usepackage[mathscr]{euscript}  

\usepackage{geometry} 
\usepackage{etoolbox}  

\makeatletter  

\patchcmd{\@tocline}
  {\hfil}  
  {\leaders\hbox to 0.6em{\hss.\hss}\hfill}  
  {}{}  

\renewcommand{\@pnumwidth}{1.5em}  

\def\l@section      {\@tocline{1}{0.6em}{0em}{}{}}  

\def\l@subsection   {\@tocline{2}{0.3em}{2em}{}{}}  

\def\l@subsubsection{\@tocline{3}{0.2em}{4.7em}{}{}}  

\patchcmd{\@tocline}
  {\vskip #1}  
  {\vskip #1\relax}  
  {}{}

\makeatother  

\newtheorem{theorem}{Theorem}[section]
\newtheorem{proposition}[theorem]{Proposition}
\newtheorem{lemma}[theorem]{Lemma}
\newtheorem{definition}[theorem]{Definition}
\newtheorem{corollary}[theorem]{Corollary}

 \numberwithin{equation}{section}

  \newcommand{\dxu}{\textnormal{d}\mu}
  \newcommand{\dxe}{\textnormal{d}x}
  \newcommand{\dye}{\textnormal{d}y}

 \newcommand{\dt}{\textnormal{d} t}
\newcommand{\dr}{\textnormal{d}r}
\newcommand{\dhn}{\textnormal{d}\mathcal{H}^{n-1}} 
\newcommand{\tl}{T}

\newcommand{\pl}{{\bf q}} 
 \newcommand{\spt}{\operatorname{supp}}

\newcommand{\dist}{\operatorname{dist}}

\newcommand{\loc}{\operatorname{loc}}

\newcommand{\osc}{\operatorname{osc}}
\newcommand{\essinf}{{\operatorname{ess} \inf}}
\newcommand{\esssup}{{\operatorname{ess} \sup}}

\newcommand{\dx}{\textnormal{d}\mathtt{V} _{\tilde \beta}}

\newcommand{\dmva}{\textnormal{d}\mathtt{V} _1}
\newcommand{\dmvb}{\textnormal{d}\mathtt{V} _2}
\newcommand{\mv}{\mathtt{V}}

\newcommand{\mva}{\mathtt{V}_1}
\newcommand{\mvb}{\mathtt{V}_2}
\newcommand{\mvbt}{\mathtt{V}_{\tilde \beta}}

\newcommand{\mvbb}{\mathtt{V}_{\beta}}
\newcommand{\mvbd}{\mathtt{V}_{2\beta}}
\newcommand{\dmvbd}{\textnormal{d}\mathtt{V} _{2\beta}}

\newcommand{\bbt}{\tilde \beta}

\newcommand{\rep}{\bar{p}}

\newcommand{\dmvbb}{\textnormal{d} \mathtt{V}_{ \beta}}

\title{Local Hölder Regularity for Quasilinear Elliptic Equations with Mixed Local-Nonlocal Operators, Variable Exponents, and Weights}

\author{Juan Pablo Alcon Apaza}

\address{Departamento de Matem\'atica, Universidade Federal de São Carlos, 13565-905, São Carlos--SP, Brazil}

\email{juanpabloalconapaza@gmail.com}

\begin{document}

\maketitle

\begin{abstract}
We establish local boundedness and local Hölder continuity of weak solutions to the following prototype problem:
$$
- \operatorname{div}\left(|x|^{-2\beta}|\nabla u|^{\pl-2} \nabla u \right) + (-\Delta)_{\Omega , p (\cdot , \cdot), \beta}^{s(\cdot , \cdot)} u= 0 \quad \text { in } \quad \Omega,
$$
where $\Omega \subset \mathbb{R}^n$, $n \geq 2$, is a bounded domain. The nonlocal operator is defined by
$$
(-\Delta)_{\Omega , p(\cdot , \cdot) , \beta}^{s(\cdot , \cdot)}  u(x) :=\operatorname{P.V.} \int_{\Omega} \frac{|u(x)-u(y)|^{p(x,y)-2}(u(x)-u(y))}{|x-y|^{n+ s(x,y) p(x,y)}} \cdot \frac{1}{|x|^\beta|y|^\beta} \,  \dye.
$$ 
Here, $p : \Omega \times \Omega \rightarrow (1, \infty)$ and $s : \Omega \times \Omega  \rightarrow (0,1)$ are measurable functions, $\pl := \esssup _{\Omega \times \Omega} \ p$, and $0 \leq 2\beta < n$. Our approach is analytic and relies on an adaptation of the De Giorgi–Nash–Moser theory to a mixed local–nonlocal framework with variable exponents and weights.
\end{abstract}

\let\thefootnote\relax
\let\thefootnote\relax


\markright{\uppercase{Local Hölder Regularity for Quasilinear Elliptic Equations with Mixed Local and Nonlocal Operators}}

\tableofcontents

\section{Introduction}

In this article, we study the local boundedness and local Hölder continuity of weak solutions to the following problem:
\begin{equation} \label{4}
-\Delta_{\pl , 2\beta} u+\mathcal{L}_{s(\cdot , \cdot),p(\cdot , \cdot),\beta}(u)=0 \quad \text { in }  \quad \Omega ,
\end{equation}
where $\Omega$ is a bounded domain in $\mathbb{R}^n$ with $n \geq 2$. The functions $ p : \Omega \times \Omega \rightarrow (1, \infty)$ and $s : \Omega \times \Omega \rightarrow (0,1)$ are measurable. We denote $\mathbf{q} := \esssup _{\Omega \times \Omega} \, p$, and assume $0 \leq 2\beta < n$.

The {\it local weighted $\pl$-Laplace operator} is defined by
\begin{equation*}
-\Delta_{\pl ,2 \beta} u =-\operatorname{div}\left(|x|^{-2\beta}|\nabla u|^{\pl-2} \nabla u \right),
\end{equation*}
and  {\it the nonlocal weighted regional $s(x,y)$-fractional  $p(x,y)$-Laplace operator} is given by
\begin{equation*}
\mathcal{L}_{s (\cdot , \cdot),p(\cdot , \cdot),\beta}(u)(x)=\operatorname{P.V.} \int_{\Omega}|u(x)-u(y)|^{p(x,y)-2}(u(x)-u(y))  \frac{K(x, y)}{|x|^\beta|y|^\beta}   \, \dye, 
\end{equation*}
where $\operatorname{P.V.}$ denotes the principal value, and
\begin{enumerate}[label=$(H_{\arabic*})$]
\item \label{162} $K : \Omega \times \Omega \rightarrow \mathbb{R}$ is a measurable function satisfying $K(x,y)=K(y,x)$, and
\begin{equation*}
\frac{\Lambda^{-1}}{|x-y|^{n+s(x,y)p(x,y) }} \leq K(x, y) \leq \frac{\Lambda}{|x-y|^{n+s(x,y)p(x,y) }}, \quad \text { a.e. } (x,y) \in \Omega \times \Omega,
\end{equation*}
for some $\Lambda \geq 1$. 

\item \label{163} The functions $s$, $p$ satisfy $s(x,y)=s(y,x)$ and $p(x,y)=p(y,x)$ for every $(x,y)$, and 
$$
 0<s^- := \underset{\Omega \times \Omega}{\essinf} \, s  \leq \underset{\Omega \times \Omega}{\esssup} \, s =: s^+ <1
$$
and
$$
1<p^- := \underset{\Omega \times \Omega}{\essinf} \, p \leq \underset{\Omega   \times \Omega}{\esssup} \, p  =: \pl <\infty .
$$
\end{enumerate}

The definition of $\mathcal{L}_{s(\cdot, \cdot), p(\cdot, \cdot), \beta}$ is motivated by the {\it regional fractional $p$-Laplace operator}:
$$
(-\Delta)_{\Omega, p}^s u(x):=\lim _{\epsilon \downarrow 0}(-\Delta)_{\Omega, p, \epsilon}^s u(x), \quad x \in \Omega,
$$
where 
$$
(-\Delta)_{\Omega, p, \epsilon}^s u(x):=C_{n, p, s} \int_{\{y \in \Omega \:|\: |y-x|> \epsilon\}}|u(x)-u(y)|^{p-2} \frac{u(x)-u(y)}{|x-y|^{n+p s}} \, \dye ,
$$
and
$$
C_{n, p, s}=\frac{s 2^{2 s} \Gamma\left(\frac{p s+p+n-2}{2}\right)}{\pi^{\frac{n}{2}} \Gamma(1-s)}.
$$
For details on this operator, see \cite{zbMATH05159612, zbMATH06406292, zbMATH06582990}.

\subsection{Context and Related Work}

$ $

\medskip

\noindent {\it {On the Local or Nonlocal Problems.}} Over the past two decades, there has been a surge of interest in nonlocal equations of the type
\begin{equation} \label{158}
(-\Delta)_{p(\cdot , \cdot)}^{s(\cdot , \cdot)} u(x):= \operatorname{P.V.} \int_{\mathbb{R}^n} \frac{|u(x)-u(y)|^{p(x, y)-2}(u(x)-u(y))}{|x-y|^{n+s(x, y) p(x, y)}} \, \dye =0 ,
\end{equation}
and related variational problems. We refer to the survey papers \cite{nezza2012guidefract, zbMATH06983323} and the monographs \cite{zbMATH06559661, zbMATH06533015} for an overview of the history, developments, and applications of nonlocal problems. In the context of regularity theory, a foundational result was established by Caffarelli and Silvestre \cite{zbMATH05204436}, who proved a Harnack inequality for the fractional Laplace equation using an extension argument. This pioneering work spurred rapid developments in the regularity theory for nonlocal equations of fractional Laplacian type, as seen in \cite{zbMATH02111870, zbMATH02232666, zbMATH05913780, zbMATH05551486, zbMATH05952965, zbMATH05223048, zbMATH05373494, zbMATH06435282, zbMATH06568053, zbMATH05049012}. Notably, Caffarelli, Chan, and Vasseur \cite{zbMATH05913780} adapted De Giorgi's approach to nonlocal equations.

For the fractional $p$-Laplace equation, Di Castro, Kuusi, and Palatucci \cite{zbMATH06330975, di2016localfract} established Harnack's inequality and Hölder continuity of weak solutions using De Giorgi's method. A key innovation in their work was the introduction of the {\it nonlocal tail}, which plays a crucial role in regularity estimates. Further developments in this direction can be found in \cite{zbMATH07511756, zbMATH06710259,  zbMATH07309248, zbMATH07576867, zbMATH07134362, zbMATH06604339, zbMATH06610795, zbMATH06637333, zbMATH06798251, zbMATH06442824, zbMATH06699542, zbMATH07309168, zbMATH07665927}.

More recently, regularity theory has been extended to nonlocal equations with nonstandard order and exponent. Bass and Kassmann \cite{zbMATH02111870, zbMATH02232666} and Silvestre \cite{zbMATH05049012} proved Hölder continuity and Harnack's inequality for bounded solutions to linear nonlocal equations with kernels of the form $|x-y|^{- 2 s(x) }$ or $\varphi(|x-y|)^{-1}$. Related results can be found in \cite{zbMATH06535137, arXiv:1505.05498, zbMATH07058421}. Additionally, De Filippis and Palatucci \cite{zbMATH07046606} studied nonlocal equations of double-phase type and established Hölder continuity for bounded solutions.

In \cite{cheng2020variablefrac}, Cheng, Ge, and Agarwal established the global boundedness of weak solutions to nonhomogeneous $s(\cdot, \cdot)$-fractional $p(\cdot, \cdot)$-Laplace equations using De Giorgi's approach. More recently, Chaker and Kim \cite{chaker2023local} obtained local Hölder continuity for weak solutions to nonlocal problems with a variable exponent $p(\cdot, \cdot)$ and a constant order $s(\cdot, \cdot) \equiv s$. In \cite{ok2023local}, Ok proved the local boundedness and local Hölder continuity of weak solutions to nonlocal equations with both a variable exponent $p(\cdot, \cdot)$ and a variable order $s(\cdot, \cdot)$.

The local analog of this case is the $p(x)$-Laplacian problem:
$$
-\Delta_{p(x)} u := -\operatorname{div}\left(|\nabla u|^{p(x)-2} \nabla u\right) = 0,
$$
which serves as a model for equations with nonstandard growth conditions. The regularity theory for this class of problems is well established; see \cite{zbMATH01586360, zbMATH02191022, zbMATH06603041, zbMATH06591396, zbMATH01270231, zbMATH05146777, zbMATH01279360}. In particular, \cite{zbMATH01279360} shows that weak solutions are locally bounded if $p(\cdot)$ is continuous, and locally Hölder continuous if $p(\cdot)$ is log-Hölder continuous.

\medskip

\noindent {\it {On Mixed Local-Nonlocal Problems.}} For the mixed local–nonlocal case with $p=2$, i.e.,
\begin{equation} \label{159}
-\Delta u + (-\Delta)^s u = 0,
\end{equation}
Foondun \cite{zbMATH05636583} has proved the Harnack inequality and local Hölder continuity for nonnegative solutions. Barlow–Bass–Chen–Kassmann \cite{zbMATH05544580} have obtained a Harnack inequality for the parabolic problem related to \eqref{159}. Chen–Kumagai \cite{zbMATH05769037} has proved the Harnack inequality and local Hölder continuity for the parabolic version of \eqref{159}. This parabolic Harnack estimate was then used to prove the elliptic Harnack inequality for \eqref{159} by Chen–Kim–Song–Vondraček in \cite{zbMATH06191309}. For more regularity results related to \eqref{159}, we refer to Athreya–Ramachandran \cite{zbMATH06871997}, Chen–Kim–Song \cite{zbMATH05938378}, and Chen–Kim–Song–Vondraček \cite{zbMATH06122084}. The arguments in these articles combine probabilistic and analytic techniques. Moreover, the Harnack inequality is proved only for globally nonnegative solutions.

Recently, interior Sobolev regularity, the strong maximum principle, and symmetry properties—among many other qualitative features of solutions to \eqref{159}—have been studied by Biagi–Dipierro–Valdinoci–Vecchi \cite{zbMATH07502820, zbMATH07400634}, Dipierro–Proietti Lippi–Valdinoci \cite{zbMATH07544282, zbMATH07752594}, and Dipierro–Ros–Oton–Serra–Valdinoci \cite{zbMATH07514705}. There also exist regularity results for a nonhomogeneous analogue of \eqref{159}. More precisely, Athreya–Ramachandran \cite{zbMATH06871997} has proved a Harnack inequality using probabilistic and analytic methods, and the authors in \cite{zbMATH07502820, zbMATH07544282} have obtained boundedness as well as interior and boundary regularity results using analytic techniques.

Biagi–Dipierro–Valdinoci–Vecchi \cite{zbMATH07400634} have also obtained interior regularity results for a nonhomogeneous version of
\begin{equation} \label{160}
-\Delta _p u + (-\Delta _p)^s u = 0.
\end{equation}
In \cite{zbMATH08016974}, Antonini and Cozzi address some regularity issues for mixed local–nonlocal quasilinear operators modeled on the sum of a $p$-Laplacian and a fractional $(s, q)$-Laplacian, as in \eqref{160}. Under suitable assumptions on the right-hand sides and the outer data, including $\partial\Omega \in C^{1, \alpha}$ and $p > s q$, they show that weak solutions of the Dirichlet problem are $C^{1, \theta}$-regular up to the boundary and establish a Hopf-type lemma for positive supersolutions.

De Filippis and Mingione \cite{zbMATH07796271} studied minimizers of functionals of the type
$$
w \mapsto \int_{\Omega} \left(  |D w|^p - f w \right) \, \dxe + \int _{\mathbb{R}^n}\int_{ \mathbb{R}^n} \frac{|w(x) - w(y)|^\gamma}{|x - y|^{n + s \gamma}}\, \dxe \dye
$$
with $p, \gamma > 1 > s > 0$ and $p > s \gamma$. They proved that minimizers are locally $C^{1, \alpha}$-regular in $\Omega$ and globally Hölder continuous. For the case where $f$ is a given vector field, see \cite{zbMATH07794358}.

Mixed local and nonlocal problems have been comprehensively studied in recent years due to their wide range of real-world applications. For instance, such operators model diffusion processes at different scales and also appear in the theory of optimal search strategies, biomathematics, animal foraging, and more; see, for example, \cite{zbMATH07752594} and the references therein. Another interesting application arises from the superposition of two stochastic processes with different scales, such as a classical random walk and a Lévy process; see \cite{zbMATH07460689} for further details.


\medskip

\noindent {\it {On Nonlocal Problems Involving Weights.}}  Abdellaoui, Attar, and Bentifour \cite{abdellaoui2016fractional} studied a class of nonlocal problems involving weights. They proved the existence of a weak solution to the problem
$$
\left\{
\begin{aligned}
(-\Delta)_{p, \beta}^s u &= f(x, u) & & \text { in } \Omega, \\
u &= 0 & &  \text { in } \mathbb{R}^n \setminus \Omega,
\end{aligned}
\right.
$$
for the largest class of data $f$, including $f(x, u) = f(x) \in L^1(\Omega)$. Here, $\Omega \subset \mathbb{R}^n$ is a domain containing the origin, and the operator is defined as
\begin{equation}\label{161}
(-\Delta)_{p, \beta}^s u(x) := \operatorname{P.V.} \int_{\mathbb{R}^n} \frac{|u(x) - u(y)|^{p-2}(u(x) - u(y))}{|x - y|^{n + ps}} \cdot  \frac{1}{|x|^\beta |y|^\beta} \,  \dye,
\end{equation}
where $1 < p < n$, $0 < s < 1$, $ps < n$, and $0 \leq \beta < \frac{n - ps}{2}$.

This operator \eqref{161} appears naturally in the context of improved Hardy inequalities. Namely, for all test functions $\phi \in C_0^\infty(\Omega)$ and for any $q < p$, we have
$$
G_{s, p}(\phi) \geq C \int_{\Omega} \int_{\Omega} \frac{|v(x) - v(y)|^p}{|x - y|^{n + qs}} w(x)^{\frac{p}{2}} w(y)^{\frac{p}{2}} \, \dxe  \dye,
$$
where
$$
G_{s, p}(\phi) := \int_{\mathbb{R}^n} \int_{\mathbb{R}^n} \frac{|\phi(x) - \phi(y)|^p}{|x - y|^{n + ps}} \, \dxe  \dye - \lambda_{n, p, s} \int_{\mathbb{R}^n} \frac{|\phi(x)|^p}{|x|^{ps}} \, \dxe,
$$
$\lambda_{n, p, s}$ is the optimal Hardy constant, $w(x) = |x|^{-(n- ps)/p}$, and $v(x) = \phi(x)/w(x)$. For more details, see \cite{zbMATH06766506, zbMATH06566683, zbMATH06293381, zbMATH05817574}.


\subsection{Main Results} We study the local boundedness of weak subsolutions to \eqref{4}. The argument is based on an energy estimate (Lemma \ref{73}), the Sobolev inequality, and an iteration technique (Lemma \ref{84}). The local Hölder continuity of weak solutions to \eqref{4} is also considered. We follow the approach of Garain and Kinnunen \cite{zbMATH07576867} and Di Castro, Kuusi, and Palatucci \cite{di2016localfract}, in which the local boundedness estimate and the logarithmic energy estimate (Lemma \ref{h51}) play an important role.

{One of the technical differences between the approaches to the mixed local–nonlocal problem and the purely nonlocal problem lies in the use of Poincaré-type Sobolev inequalities in forms adapted to the local problem (\cite[Corollary 1.4]{chani1985weightedpeano}; \cite[Corollary 1.57]{zbMATH01061233}) and to the nonlocal problem (\cite[Lemma 2.3]{ok2023local}; \cite[Lemma 4.9]{zbMATH06710259}; \cite[Formula (6.3)]{zbMATH02038439}). For example, this difference can be seen in the terms that appear in the lower bounds of the Caccioppoli-type estimates in \cite[Theorem 1.4]{di2016localfract}, \cite[Lemma 3.1]{zbMATH07576867}, and \cite[Lemma 4.2]{ok2023local}; see also Lemma \ref{73} below. These estimates are used to obtain logarithmic estimates (see \cite[Lemma 4.3]{ok2023local} and the estimates in \eqref{h25} below), where terms of the form
$$
r^{p^-_r - p^+_r}
$$
appear, with
$$
p^-_r = \inf_{B_r(x_0) \times B_r(x_0)} p \quad \text { and } \quad p^+_r = \sup_{B_r(x_0) \times B_r(x_0)} p.
$$}

{To handle such terms, it is useful to consider log-Hölder continuity conditions on $p$, for example,
$$
\sup_{0 < r \leq 1/2} \left( \omega_p(r) + \omega_s(r) \right) \ln \left( \frac{1}{r} \right) \leq c_{LH}, \quad \text { for some } \quad c_{LH} > 0,
$$
where
$$
\begin{aligned}
\omega_s(r) := & \displaystyle \sup_{B_r \subset \Omega} \sup_{x_1, x_2, y_1, y_2 \in B_r} \left| s(x_1, y_1) - s(x_2, y_2) \right|, \\[5pt]
\omega_p(r) := &\displaystyle \sup_{B_r \subset \Omega} \sup_{x_1, x_2, y_1, y_2 \in B_r} \left| p(x_1, y_1) - p(x_2, y_2) \right|.
\end{aligned} 
$$}

{With respect to $s$, terms of the form
$$
r^{s^-_r - s^+_r},
$$
where $ s^-_r = \inf_{B_r(x_0) \times B_r(x_0)} s $ and $ s^+_r = \sup_{B_r(x_0) \times B_r(x_0)} s $, appear in \cite[Lemma 4.4]{ok2023local} as a result of applying the logarithmic inequality together with the Poincaré-type inequality for the nonlocal problem (an inequality that involves terms related to $s_r^{-}, s_r^{+}, p_r^{-}$, and $p_r^{+}$). In our case, this issue does not arise (see Corollary \ref{h65}), since the Poincaré-type inequality we employ (Lemma \ref{58}) is associated with the space $W^{1,\pl}(B_r(x_0) ; \mv)$.}

{Although we do not assume a log-Hölder continuity condition on $s$, we impose a restriction on the oscillation of $p$ that depends on $s$, namely,
$$
\left(1 - s^{+}\right)\left(p^{-} - 1\right) \geq \frac{\pl}{p^{-}} \left(\pl - p^{-}\right),
$$
as shown in Theorem \ref{157}.}

{We say that a function $p : \Omega \times \Omega \rightarrow \mathbb{R}$ satisfies condition \ref{144} if:
\begin{enumerate}[label=($P_{\arabic*}$)]
\item  \label{144}  There exists a constant $\mathtt{C}>0$ such that 
\begin{gather}\label{145}
r^{p^-_r - p^+_r} \leq \mathtt{C} , \quad  \text { for every ball  } \quad B_r (x_0) \Subset \Omega .
\end{gather}
\end{enumerate}}

The condition \eqref{145},  taken from \cite{chaker2023local}, is in fact weaker than the log-Hölder continuity of $p$; for further details, see \cite[Remark 1.3]{chaker2023local}.

For a measurable function $v : \Omega \rightarrow \mathbb{R}$, we define
\begin{equation*}
\tl (v ; x_0, r, \rho):=\underset{x \in B_\rho (x_0)}{\esssup} \, \int_{\Omega \setminus B_r(x_0)} \frac{|v(y)|^{p(x, y)-1}}{|y-x_0|^{n+s(x, y) p(x, y)}} \cdot \frac{1}{|y|^\beta} \, \dye ,
\end{equation*}
where $B_{{r}} (x_0)\subset B_\rho (x_0) \subset \Omega$.


The main result of this work is the following:
\begin{theorem}\label{157} (Local Hölder continuity). Let $\Omega$ be a  bounded domain. Assume that conditions \ref{162}, \ref{163}, and {property \ref{144}} are satisfied. Also, suppose that
$$
\beta \geq 0, \quad n-2\beta - 1+     {s^-}>0, \quad n-\beta>s^+\pl, \quad \text { and }\quad (1-s^+)(p^--1)\geq \frac{\pl}{p^-}(\pl- p^-).
$$ 

Let $u\in W^{1,\pl } _{\loc} (\Omega ; |x|^{-2\beta}) \cap W^{s(x,y),p(x,y)} (\Omega ; |x|^{-\beta})$ be a weak solution of \eqref{4}, satisfying $\tl (u ; x_0, r, \rho) <\infty $ for all balls  $B_r (x_0) \subset B_\rho (x_0) \subset \Omega$.   Then $u$ is locally Hölder continuous in $\Omega$.

Moreover, if $B_R(x_0) \subset \Omega$   with $R \in (0, 1/2]$ is a ball such that either  $0 \notin \overline{B_{3R}(x_0)}$ or $x_0 = 0$, then there exist constants $\alpha \in ( 0, \frac{1-s^+}{\pl-1} )$ and $c>0$, depending on ${\sup _{y \in \Omega} |y|}$, $R$, $s^-$, $s^{+}$, $p^{-}$, $\pl$,  $\Lambda$, $\beta$, and $ n$, such that
\begin{equation*}
\begin{aligned}
\underset{B_{\rho}(x_{0})}{\osc} \, u := & \, \displaystyle \displaystyle \underset{B_{\rho}(x_{0})}{\esssup } \, u-\underset{B_{\rho}(x_{0})}{\essinf } \, u \\[5pt]
\leq & \, \displaystyle \displaystyle c\left(\frac{\rho}{r}\right)^{\alpha}   \left[ r^{\frac{\pl}{\pl - 1}}\tl\left( u ; x_{0}, \frac{r}{2} , r \right) ^{\frac{1}{\pl -1}}+\left(\fint _{B_{ r}(x_{0})}|u|^{\pl} \, \dmvbd \right)^{\frac{1}{\pl}} +1 \right],
\end{aligned}
\end{equation*}
where $B_{2r}(x_0) \subset B_R(x_0)$, $r \in (0, 1/4]$, and $\rho \in (0, r/2]$.

If $\beta = 0$, the constants $\alpha$ and $c$ depend only on  $s^-$, $s^+$, $p^-$, $\pl$, $\Lambda$, and $n$. In the case $x_0 = 0$, the constants $\alpha$ and $c$ depend only on  $s^-$, $s^+$, $p^-$, $\pl$, $\Lambda$, $\beta$, and $n$.
\end{theorem}

Our second main result is given in the next theorem. {Note that this result does not require the condition \ref{144}.}
\begin{theorem}\label{129}
(Local boundedness). Let $\Omega$ be a  bounded domain. Assume conditions \ref{162} and \ref{163} are satisfied. Also, assume that $\beta \geq 0$ and $n-\beta - 1+      {s^-}>0$.

Let $u\in W^{1,\pl } _{\loc} (\Omega ; |x|^{-2\beta}) \cap W^{s(x,y),p(x,y)} (\Omega ; |x|^{-\beta})$ be a weak solution of \eqref{4}, satisfying $\tl (u ; x_0, r, \rho) <\infty $ for all balls  $B_r (x_0) \subset B_\rho (x_0) \subset \Omega$. Then there exists a  constant $C=C(s^{+}, p^{-}, \pl , \beta, \Lambda, n)>0$, such that
\begin{equation*}
\begin{aligned}
\underset{B_{r/2}(x_{0})}{\esssup } \,  u  
 \leq  \delta (|x_0|+1)^{\frac{\beta}{\pl-1}} r^{\frac{\pl}{\pl-1}} \tl \left(u_+ ; x_{0}, \frac{r}{2} ,r \right)^{\frac{1}{\pl -1}} +C\delta^{-\frac{(\pl - 1)\kappa}{\pl(\kappa -1)}} \left(\fint_{B_{r}(x_{0})} u_+^{\pl} \, \dmvbd \right)^{\frac{1}{\pl}} +1,
\end{aligned}
\end{equation*}
whenever $B_{r}(x_{0}) \subset \Omega$ with $r \in(0,1/2]$ and $\delta \in(0,1]$. Here,      {\(\kappa = n/(n - \pl)\) if \(1 < \pl < n\), and \(\kappa = 2\) if \(\pl \geq n\).}
\end{theorem}

This article is organized as follows. Section \ref{184} introduces key definitions and preliminary results. Section \ref{185} is devoted to establishing the necessary energy estimates. In Sections \ref{186} and \ref{187}, we prove the local boundedness and local Hölder continuity of weak solutions, respectively.


\section{Preliminaries} \label{184}


We adopt the following notation:
\begin{enumerate} 
\item[$\bullet$]  $\pl := \esssup_{\Omega \times \Omega} \, p$.

\item[$\bullet$] For $x, y \in \mathbb{R}^n$, their Euclidean inner product is denoted by $\langle x, y \rangle$ or $x \cdot y$.  The Euclidean norm is denoted by $|x|$.

\item[$\bullet$] The Euclidean ball of radius $r$ centered at $x_0$ is denoted by
$B_r(x_0) = \{ x \in \mathbb{R}^n \: | \: |x - x_0| < r \}$.

\item[$\bullet$] The conjugate exponent of $q > 1$ is given by $q' := \frac{q}{q - 1}$.

\item[$\bullet$] We define 
$$
\mvbb (x) := |x|^{-\beta}, \quad  \mvbb (F) := \int _F  \dmvbb, \quad \text {  and  } \quad  \dmvbb := |x|^{-\beta}  \dxe ,
$$
where $F$ is a measurable set. The standard Lebesgue measure of $F$ is denoted by $|F|$.

\item[$\bullet$] For a measurable function $f : B_r(x_0) \rightarrow \mathbb{R}$, we write
$$
\fint _{B_r(x_0)} f\, \dmvbb  = (f)_{B_r(x_0), \mvbb } = \frac{1}{\mvbb (B_r(x_0))} \int_{B_r(x_0)} f \, \dmvbb .
$$

\item[$\bullet$] For $a \in \mathbb{R}$, we  set $a_+ := \max\{a, 0\}$ and $ a_- := -\min\{a, 0\}$.

\item[$\bullet$] We write 
$$
\mathcal{A}(u,x, y) : = |u(x)-u(y)|^{p(x,y)-2}(u(x)-u(y)) \quad \text { and } \quad \dxu (x,y) : =\frac{1}{|x|^\beta |y|^\beta} K(x, y) \, \dxe \dye.
$$
\item[$\bullet$] The symbols $c$, $C$, or $C_i$ (for $i \in \mathbb{N}$) denote positive constants that may vary from line to line, or even within the same line. If a constant depends on parameters $r_1, r_2, \ldots$, we write it as $C(r_1, r_2, \ldots)$.
\end{enumerate}

\subsection{Weighted Variable Exponent Lebesgue Spaces}

In this section, we provide introductory definitions and results on weighted variable exponent Lebesgue spaces. For further details, see \cite{dienhar2011lebesgue, diening2008muckenhoupt, fanzhao2001spaces, heinonen2006nonlinpot, kovavcik1991spaces, zbMATH01308968, turesson2000potentweig}.

Let $\mv$ be a weight function on $\mathbb R^n$; that is,  $\mv$ is a nonnegative and locally integrable function with respect to the Lebesgue measure.
 

Let $\Omega \subset \mathbb{R}^n$ be an open set, and let $p : \Omega \rightarrow[1, \infty)$ (called a {\it variable exponent} on $\Omega$) be a measurable function. The weighted variable exponent Lebesgue space $L^{p(x)}(\Omega; \mv)$ is defined by
$$
L^{p(x)}(\Omega ; \mv )=\left\{u : \Omega \rightarrow \mathbb{R} \: | \: u \text { is measurable  and }  \int_{\Omega}|u|^{{p(x)}}   \mv \, \dxe <\infty\right\},
$$
with the norm
\begin{equation}\label{142}
\|u\|_{p(x), \Omega , \mv} := \inf \left\{\lambda > 0 \: |\:  \varrho_{p(x), \Omega , \mv}\left(\frac{u}{\lambda}\right) \leq 1\right\},
\end{equation}
where
$$
\varrho_{p(x), \Omega , \mv} (u)  :=\int_{\Omega}|u|^{{p(x)}}  \mv \, \dxe .
$$

The relationship between the modular $\varrho _{p(x) , \Omega , \mv} (\cdot)$ and  the norm $\|\cdot\|_{p(x) , \Omega , \mv}$ is given by:
\begin{equation*}
\begin{aligned}
&   \min \left\{\varrho _{p(x) , \Omega , \mv}  (f)^{\frac{1}{p^{-}}}, \varrho _{p(x) , \Omega , \mv}(f)^{\frac{1}{p^+ }}\right\} \\[5pt]
&  \qquad \leq \|f\|_{p(x) , \Omega , \mv} \leq \max \left\{\varrho _{p(x) , \Omega , \mv}(f)^{\frac{1}{p^{-}}}, \varrho _{p(x) , \Omega , \mv}(f)^{\frac{1}{p^+ }}\right\},
\end{aligned}
\end{equation*}
where $p^{-}=\underset{\Omega}{\essinf} \,  p$ and $p^+ =\underset{ \Omega}{\esssup} \, p$.

Hölder's inequality can be written in the form:
\begin{equation}\label{154}
\|f g\|_{s(x) , \Omega , \mv} \leq 2\|f\|_{p(x) , \Omega , \mv} \|g\|_{q(x) , \Omega , \mv},
\end{equation}
where $\frac{1}{s}=\frac{1}{p}+\frac{1}{q}$; see  \cite[Lemma 3.2.20]{dienhar2011lebesgue}.

For an introduction to weighted Lebesgue and Sobolev spaces, we refer the reader to \cite{heinonen2006nonlinpot, turesson2000potentweig}. Let $q\in (1,\infty)$. The {\it weighted Lebesgue space} $L^q( \Omega ; \mv)$ is defined by
$$
L^q(\Omega ; \mv):=\left\{u : \Omega \rightarrow \mathbb{R} \:|\: u \text { is measurable  and } \int_\Omega |u|^q  \mv \, \dxe <\infty\right\},
$$
equipped with the norm
$$
\|u\|_{q, \Omega , \mv}:=\left(\int_\Omega |u|^q  \mv \, \dxe \right)^{\frac{1}{q}}.
$$

The {\it weighted Sobolev space} is defined by
$$
W^{1, q}(\Omega ; \mv):=\left\{u \in L^q(\Omega ; \mv )\:|\:  | \nabla u | \in L^q( \Omega; \mv)\right\}
$$
equipped with the norm
$$
\|u\|_{1, q, \Omega , \mv }:=\|u\|_{q, \Omega, \mv}+\|\nabla u\|_{q, \Omega, \mv}.
$$


\subsection{Weighted Variable Exponent Fractional Sobolev Spaces} In this subsection, we  provide introductory definitions and results on weighted variable exponent fractional Sobolev spaces. For more details, see  \cite{zbMATH07837279, zbMATH07575925, ho2019fract, zbMATH07647941}.

Let $\Omega \subset \mathbb{R}^n$ be an open set, and let $\mv$ be a weight function on $\mathbb{R}^n$. Let $p : \Omega \times \Omega \rightarrow [1,\infty )$ and $s : \Omega \times \Omega \rightarrow (0,1)$ be measurable functions such that
\begin{gather}
1< \underset{\Omega \times \Omega}{\essinf} \, p \leq   \underset{\Omega \times \Omega}{\esssup} \, p <\infty, \quad  p(x,y)=p(y,x), \text { a.e. } (x,y)\in \Omega\times \Omega, \label{174}\\[5pt]
0< \underset{\Omega \times \Omega}{\essinf} \, s \leq   \underset{\Omega \times \Omega}{\esssup} \, s <1 , \quad s(x,y)=s(y,x), \text { a.e. } (x,y)\in \Omega\times \Omega. \label{175}
\end{gather}

The weighted variable exponent fractional Sobolev space is defined as
$$
W^{s(x,y),p(x,y)} (\Omega ; \mv):= \left\{u\in L^{\rep (x)} (\Omega ; \mv)  \:|\: \tilde \varrho _{s,p, \Omega, \mv}(u) < \infty \right\},
$$
where
$$
\tilde \varrho_{s,p, \Omega, \mv}(u):=\int_{\Omega} \int_{\Omega} \frac{|u(x) - u(y)|^{p(x,y)} }{|x-y|^{n+ s(x,y)p(x,y)}}  \mv (x) \mv (y) \, \dxe \dye
$$
and
$$
\rep (x) := p(x,x).
$$

We define the seminorm
$$
[u]_{s , p ,\Omega,\mv} :=\inf \left\{ \lambda >0 \:|\: \tilde \varrho _{s,p,\Omega,\mv} \left(\frac{u}{\lambda}\right)\leq 1 \right\}.
$$

Then, $W^{s(x,y),p(x,y)} (\Omega ; \mv)$, endowed with the norm
$$
\|u\|_{s,p,\Omega,\mv} := \|u\|_{\rep (x) , \Omega , \mv} + [u]_{s , p ,\Omega,\mv} ,
$$
is a Banach space; see \cite{bahrouni2018compfrac, bahrouni2018fract, cheng2020variablefrac}.

We also define
$$
\hat \varrho _{s,p,\Omega,\mv} (u):= \varrho _{\rep (x) , \Omega , \mv} (u) +\tilde \varrho_{s,p, \Omega, \mv}(u),
$$
and the corresponding norm
$$
|u|_{s,p,\Omega,\mv} := \inf \left\{\lambda >0\:|\: \hat \varrho _{s,p,\Omega,\mv} \left(\frac{u}{\lambda}\right) \leq 1\right\}.
$$

Then, the norms $\|\cdot\|_{s,p,\Omega,\mv}$ and $|\cdot | _{s,p,\Omega,\mv}$ are comparable; see \cite{ho2019fract}. The following lemma is taken from \cite{ho2019fract}.
\begin{lemma}
 Let $u \in W^{s(x,y), p(x,y)}(\Omega ; \mv)$. Then:
 \begin{enumerate}[label=(\roman*)]
\item If $[u]_{s,p,\Omega,\mv} \geq 1$, then $[u]_{s,p,\Omega,\mv} ^{p^-} \leq \tilde \varrho_{s,p,\Omega,\mv} (u) \leq[u]_{s,p,\Omega,\mv} ^{p^+ }$.

\item If $[u]_{s,p,\Omega,\mv}  \leq 1$, then $[u]_{s,p,\Omega,\mv} ^{p^+ } \leq \tilde \varrho _{s,p,\Omega,\mv} (u) \leq [u]_{s,p,\Omega,\mv} ^{p^{-}}$.

\item If $|u|_{s,p,\Omega,\mv}  \geq 1$, then $|u|_{s,p,\Omega,\mv} ^{p^{-}} \leq \hat{\varrho} _{s,p,\Omega,\mv} (u) \leq |u|_{s,p,\Omega,\mv} ^{p^+ }$.

\item If $|u|_{s,p,\Omega,\mv}  \leq 1$, then $|u|_{s,p,\Omega,\mv} ^{p^+ } \leq \hat{\varrho} _{s,p,\Omega,\mv}  (u) \leq |u|_{s,p,\Omega,\mv} ^{p^{-}}$.
\end{enumerate}
\end{lemma}

The following lemma will be used several times throughout this article.
\begin{lemma} \label{27}
\begin{enumerate}[label=(\roman*)]
\item \label{23} If $R>0$, $0<|y|\leq R$, $0<\alpha <1$, $\bbt\geq 0$, and $n-\bbt -\alpha>0$, then
\begin{equation}\label{21}
\int_{B_R(0)} \frac{1}{|x-y|^{n-\alpha}|x|^{\bbt}} \, \dxe \leq  C \frac{R^{2\alpha}}{|y|^{\bbt +\alpha}},
\end{equation}
where
 $$
C=   \frac{ 2^{1-\alpha}  \pi^{\frac{n}{2}}   \Gamma\left(\frac{1-\alpha}{2}\right)}{ \Gamma\left(\frac{n-1}{2}\right) \Gamma\left(1-\frac{\alpha}{2}\right)}    \left[  \frac{\Gamma(n-\bbt-\alpha) \Gamma\left(\frac{\alpha}{2}\right)}{\Gamma\left(n-\bbt-\frac{\alpha}{2}\right)} + \frac{1}{\alpha}  \right].
$$

\item \label{28} If $0<R\leq |y|/2$, $\bbt \geq 0$, and $n- \bbt  >\gamma >0$, then
\begin{equation} \label{34}
\int_{B_R(0)} \frac{1}{|x-y|^{n+\gamma}|x|^{\bbt}} \, \dxe \leq  \frac{2^{4+\gamma}\pi^{\frac{n+1}{2}}}{3\Gamma\left(\frac{n-1}{2}\right) (n- \bbt -\gamma )} \cdot  \frac{R^{n- \bbt -\gamma }}{|y| ^{n}} .
\end{equation}

\item \label{167} Assume $n-\bbt - (1-a)>0$ and $\bbt \geq 0$. Also, let $0<a<1$ and $1\leq b$. Let $K\subset \mathbb{R}^n$ be a bounded measurable set. If $x \in K$, $x\neq 0$, and $K\subset B_r(0)$, then
\begin{equation}\label{168}
\int _{K} \frac{1}{|x-y|^{n-(1-a)b}|y|^{\bbt}} \, \dye \leq C(r,a,b) \frac{1}{|x|^{\bbt}}.
\end{equation}

\end{enumerate}
\end{lemma}
\begin{proof}
\ref{23} We set $r=|x|$ and $\rho=|y|$. Then $x=r x^{\prime}$, $y=\rho y^{\prime}$, where $|x^{\prime}|=|y^{\prime}|=1$. Therefore,
$$
 \int_{B_R(0)} \frac{1}{|x-y|^{n-\alpha}|x|^{\bbt}} \, \dxe = \int_0^R \frac{r^{n-1}}{r^{\bbt} \rho ^{n-\alpha}}\left(\int_{|x^{\prime}|=1} \frac{1}{\left|\frac{r}{\rho} x^{\prime} - y^{\prime}\right|^{n-\alpha}}  \, \dhn (x^{\prime}) \right) \, \dr .
$$

We now set $t=\frac{r}{\rho}$; hence,
\begin{equation} \label{20}
\int_{B_R(0)} \frac{1}{|x-y|^{n-\alpha}|x|^{\bbt}} \, \dxe =  \frac{1}{\rho ^{\bbt-\alpha}}  \int_0^{\frac{R}{\rho}} t^{n-\bbt -1}\left(\int_{|x^{\prime}|=1} \frac{1}{\left|tx^{\prime}- y^{\prime}\right|^{n-\alpha }} \, \dhn (x^\prime)\right) \, \dt .
\end{equation}

We define
$$
K_\alpha (t):=\int_{|x^{\prime}|=1} \frac{1}{\left|tx^{\prime}- y^{\prime}\right|^{n-\alpha}} \, \dhn (x^\prime).
$$

As in \cite[pp. 248 - 251]{ferrari2012radial},  we find
\begin{equation}\label{94}
K_\alpha(t)=2 \frac{\pi^{\frac{n-1}{2}}}{\Gamma\left(\frac{n-1}{2}\right)} \int_0^\pi \frac{\sin ^{n-2} \theta }{\left(1-2 t\cos  \theta +t^2\right)^{\frac{n-\alpha}{2}}} \, \textnormal{d} \theta .
\end{equation}

For $t>1$, we have
\begin{equation*}
K_\alpha (t) = 2 \frac{\pi^{\frac{n-1}{2}}}{\Gamma\left(\frac{n-1}{2}\right)}  t ^{-n+2}\left(t^2-1\right)^{-1+\alpha}    \int_0^\pi \sin ^{n-2}  \theta \frac{\left(\sqrt{t^2-\sin ^2 \theta }+\cos \theta \right)^{1-\alpha}}{\sqrt{t ^2-\sin ^2 \theta }} \, \textnormal{d} \theta.
\end{equation*}

Since $\sqrt{t^2-\sin ^2 \theta}\geq |\cos \theta|$,
\begin{equation}\label{16}
\begin{aligned}
K_\alpha (t) \leq  & \, \displaystyle 2 \frac{\pi^{\frac{n-1}{2}}}{\Gamma\left(\frac{n-1}{2}\right)}  t ^{-n+2}\left(t^2-1\right)^{-1+\alpha}  2^{1-\alpha}   \int_0^\pi \frac{1}{|\cos \theta| ^\alpha} \, \textnormal{d} \theta \\
   = & \, \displaystyle 2 ^{2-\alpha}\frac{\pi^{\frac{n}{2}}}{\Gamma\left(\frac{n-1}{2}\right)}  t ^{-n+2}\left(t^2-1\right)^{-1+\alpha}  \frac{\Gamma\left(\frac{1-\alpha}{2}\right)}{\Gamma\left(1-\frac{\alpha}{2}\right)}.
\end{aligned}
\end{equation}

Similarly, if $s<1 $,
\begin{equation}\label{17}
\begin{aligned}
K_\alpha (s)= & \, \displaystyle s^{-n+\alpha} K_\alpha \left(\frac{1}{s}\right) \\
 \leq & \, \displaystyle 2 ^{2-\alpha}\frac{\pi^{\frac{n}{2}}}{\Gamma\left(\frac{n-1}{2}\right)}  s ^{-\alpha}\left(1-s^2\right)^{-1+\alpha}  \frac{\Gamma\left(\frac{1-\alpha}{2}\right)}{\Gamma\left(1-\frac{\alpha}{2}\right)}.
\end{aligned}
\end{equation}

From, \cite{gradshteyn2007table} or \cite{ abramowitz1965formulas, magnus1966formulas},
\begin{equation} \label{18}
\int _0^1 t^{n-\bbt -1} K_\alpha (t) \, \dt \leq  
2 ^{1-\alpha}\frac{\pi^{\frac{n}{2}}}{\Gamma\left(\frac{n-1}{2}\right)}  \cdot \frac{\Gamma\left(\frac{1-\alpha}{2}\right)}{\Gamma\left(1-\frac{\alpha}{2}\right)} \cdot \frac{\Gamma (n-\bbt -\alpha)\Gamma \left(\frac{\alpha}{2}\right)}{ \Gamma \left( n-\bbt -\frac{\alpha}{2} \right)}
\end{equation}
and, since $t^{-\bbt + 1}\leq t$ if $t\geq 1$:
\begin{equation} \label{19}
\int _1^{\frac{R}{\rho}} t^{n-\bbt -1} K_\alpha (t) \, \dt \leq  \frac{2 ^{1-\alpha}}{\alpha} \cdot \frac{\pi^{\frac{n}{2}}}{\Gamma\left(\frac{n-1}{2}\right)}   \cdot \frac{\Gamma\left(\frac{1-\alpha}{2}\right)}{\Gamma\left(1-\frac{\alpha}{2}\right)}\left(\frac{R}{\rho}\right)^{2\alpha}.
\end{equation}

From \eqref{20}, \eqref{18}, and \eqref{19}, we obtain \eqref{21}.

\medskip

\ref{28} Similarly to \ref{23}, we set $r=|x|$ and $\rho=|y|$. Then $x=r x^{\prime}$ and $y=\rho y^{\prime}$, where $|x^{\prime}|=|y^{\prime}|=1$. Then
\begin{equation} \label{29}
\int_{B_R(0)} \frac{1}{|x-y|^{n+\gamma}|x|^{\bbt}} \, \dxe =  \frac{1}{\rho ^{\bbt + \gamma}}  \int_0^{\frac{R}{\rho}} t^{n-\bbt -1}\hat K_\gamma (t) \, \dt .
\end{equation}
where $t=\frac{r}{\rho}$, and  for $t>1$,
\begin{align*}
 \hat K_\gamma (t):= & \, \displaystyle \int_{|x^{\prime}|=1} \frac{1}{\left|tx^{\prime}- y^{\prime}\right|^{n+\gamma}} \, \dhn (x^\prime)\\
= & \, \displaystyle 2 \frac{\pi^{\frac{n-1}{2}}}{\Gamma\left(\frac{n-1}{2}\right)} t^{-n+2}\left(t^2-1\right)^{-1 - \gamma} \int_0^\pi \sin ^{n-2} \theta \frac{\left(\sqrt{t^2-\sin ^2 \theta}+\cos \theta\right)^{1+\gamma}}{\sqrt{t^2-\sin ^2 \theta}} \, \textnormal{d} \theta.
\end{align*}

Then, since $t^2  -1\geq \sqrt{t^2 -\sin ^2\theta}$ if $t\geq 2$,
\begin{equation*}\label{30}
\hat K_\gamma (t) \leq  2^{2+\gamma} \frac{\pi^{\frac{n+1}{2}}}{\Gamma\left(\frac{n-1}{2}\right)} t^{-n+2}(t^2 -1)^{-1} .
\end{equation*}

If $0<s\leq 1/2$, this implies $(1-s^2)^{-1}\leq 3/4$, and
\begin{equation}\label{31}
\begin{aligned}
\hat K_\gamma (s)= & \, \displaystyle s^{-n-\gamma} \hat K_\gamma \left(\frac{1}{s}\right)\\
\leq & \, \displaystyle  \frac{ 2^{4+\gamma} \pi^{\frac{n+1}{2}}}{ 3 \Gamma\left(\frac{n-1}{2}\right)} s^{-\gamma} .
\end{aligned}
\end{equation}

Hence,
\begin{equation} \label{32}
\int_{B_R(0)} \frac{1}{|x-y|^{n+\gamma}|x|^{\bbt}} \, \dxe \leq  \frac{2^{4+\gamma}\pi^{\frac{n+1}{2}}}{3\Gamma\left(\frac{n-1}{2}\right) (n- \bbt -\gamma )} \cdot  \frac{1}{|y| ^{\bbt + \gamma}} \left( \frac{R}{|y|}\right) ^{n- \bbt -\gamma },  
\end{equation}
which proves \eqref{34}.

\medskip

\ref{167} Let $x\in K\subset B_r (0)$ with $x\neq 0$, we have
\begin{align*}
\int_{K} \frac{1}{|x-y|^{n-(1-a) b}|y|^{\bbt}} \, \dye = & \, \displaystyle \displaystyle \int_{K \setminus B_{|x|} (0)} \frac{1}{|x-y|^{n-(1-a) b}|y|^{\bbt}} \, \dye + \int_{ K\cap B_{|x|} (0)} \frac{1}{|x-y|^{n-(1-a) b}|y|^{\bbt}} \, \dye\\[5pt]
\leq & \, \displaystyle \frac{1}{|x|^{\bbt}} \int _{B_{2r}( {x} ) }  \frac{1}{|x-y|^{n-(1-a) b}} \, \dye + \int _{B_{|x|} (0)} \frac{|x-y|^{(1-a) (b-1)}}{|x-y|^{n-(1-a)}|y|^{\bbt}} \, \dye\\[7pt]
\leq & \, \displaystyle \frac{|\partial B_1(0)|(2r)^{(1-a)b}}{(1-a)b|x|^{\bbt}} + (2|x|)^{(1-a) (b-1)} \int _{B_{|x|} (0)} \frac{1}{|x-y|^{n-(1-a)}|y|^{\bbt}} \, \dye.
\end{align*}
Applying \eqref{21} in \ref{23}, we conclude \eqref{168}. \end{proof}

\begin{lemma} \label{176}
Let $\Omega \subset \mathbb{R}^n$ be a {bounded} open set. Assume that conditions \eqref{174} and \eqref{175} on $s$ and $p$ hold. Furthermore, for all measurable  sets $K\Subset \Omega$, suppose there exist constants $c_0, c_1>0$ such that
\begin{gather}
c_0 \leq \mv (x), \quad \text { a.e. } x\in K, \label{172}\\[5pt]
\int _{K } \frac{\mv (y)}{|x-y|^{n-(1-s(x,y))p(x,y)}} \, \dye \leq c_1\mv (x), \quad \text { a.e. } x\in K \setminus \{0\}, \label{169}\\[5pt]
 {\int _{\Omega} \mv \, \dxe < \infty}.\label{173}
\end{gather}

If $u\in W^{s(x,y),p(x,y)} (\Omega ; \mv) \cap L^{{\pl}} _{\loc}(\Omega ; \mv ^2)$ and $\zeta : \Omega \rightarrow \mathbb{R}$ is a Lipschitz function with $\spt \zeta {\Subset} \Omega$, then  
$$
\zeta u\in W^{s(x,y),p(x,y)} (\Omega ; \mv) \cap L^{{\pl}} (\Omega ; \mv ^2).
$$
\end{lemma}
\begin{proof}
 
 We have
\begin{equation}\label{170}
|\zeta (x) - \zeta (y) |\leq C_2 |x-y|, \quad \forall  (x,y) \in {\Omega \times \Omega} ,
\end{equation}
for some constant $C_2 > 1$.

{Since $\spt \zeta \Subset \Omega$, there exist open sets $U_1, U_2$ such that
$$
\operatorname{supp} \zeta \subset U_1 \Subset U_2 \Subset \Omega .
$$}

 Employing the identity $ab-bd = a(b-d) + d(a-b)$ for all $a,b,c,d \in \mathbb{R}$,  {and using the decomposition
$$
\Omega \times \Omega = \left(\mathcal S\times \mathcal S \right)\cup  \left[ \mathcal S \times \left( \Omega \setminus \mathcal S \right) \right] \cup \left[   \left( \Omega \setminus \mathcal S \right) \times \mathcal S \right] \cup \left[   \left( \Omega \setminus \mathcal S \right) \times \left( \Omega \setminus \mathcal S \right) \right], \quad \mathcal{S} \subset \Omega, 
$$
we obtain}
\begin{align*}
\int _{\Omega} \int _{\Omega} & \frac{|\zeta (x)u (x) -\zeta (y) u (y)|^{p(x,y)}}{|x-y|^{n+s(x,y)p(x,y)}}  \mv (x) \mv (y) \, \dxe \dye \\[5pt]
= & \, \displaystyle \int _{  {U_2}  } \int _{ {U_2} } \frac{|\zeta (x)u (x) -\zeta (y) u (y)|^{p(x,y)}}{|x-y|^{n+s(x,y)p(x,y)}}  \mv (x) \mv (y) \, \dxe \dye \\[5pt]
& \, \displaystyle + 2\int _{ {U_2} } \int _{ {\Omega \setminus  U_2 } } \frac{|\zeta (y) u (y)|^{p(x,y)}}{|x-y|^{n+s(x,y)p(x,y)}}  \mv (x) \mv (y) \, \dxe \dye\\[7pt]
 \leq & \, \displaystyle 2^{{\pl} -1}\int _{{U_2} } \int _{{U_2} } \frac{ \left|u (x)(\zeta (x) -\zeta (y))\right|^{p(x,y)}   }{|x-y|^{n+s(x,y)p(x,y)}}  \mv (x) \mv (y)\,  \dxe \dye \\[5pt]
 & \, \displaystyle + 2^{{\pl} -1}\int _{{U_2}} \int _{ {U_2} } \frac{  \left| \zeta (y)(u (x) -u (y))\right|^{p(x,y)}  }{|x-y|^{n+s(x,y)p(x,y)}}  \mv (x) \mv (y)\, \dxe \dye \\[5pt]
& \, \displaystyle + 2\int _{ {U_1}} \left( | u (y)|^{{\pl}} +1\right)\mv (y)\int _{{\Omega \setminus  U_2}} \frac{1}{|x-y|^{n+s(x,y)p(x,y)}}  \mv (x) \, \dxe \dye \\[7pt]
=: & \, \displaystyle I_1 + I_2 + I_3. 
 \end{align*}

By \eqref{170} and \eqref{169}, we have
\begin{equation}\label{171}
I_1 \leq c_1 C_2^{{\pl}}2^{{\pl} -1}    \int _{{U_2}} \left( |u|^{{\pl}} +1 \right) \mv ^2 \, \dxe <\infty,
\end{equation}
since $u\in L_{\loc}^{{\pl}}(\Omega ; \mv ^2)$.

Moreover, $I_2$ is finite because $u \in W^{s(x, y), p(x, y)}(\Omega ; \mv)$. We also have $ I_3< \infty$ by \eqref{172}, \eqref{173}, and the fact that for all $(x,y)\in {(\Omega \setminus U_2)\times U_1}$, 
{$$
\frac{1}{|x-y|^{n+s(x, y) p(x, y)}} \mv(x)\leq \left[ \esssup _{(x,y)\in(\Omega \setminus U_2)\times U_1} \delta ^{-n- s(x,y)p(x,y)} \right] \mv (x) , 
$$
where $\delta = \dist (\Omega  \setminus U_2 , U_1 )>0 $.}

Therefore, we conclude the proof of the lemma.\end{proof}

The following result will be useful when using test functions.
\begin{corollary}
Let $\Omega \subset \mathbb{R}^n$ be a {bounded} open set. Assume that conditions \eqref{174} and \eqref{175} on $s$ and $p$ hold; also, $n-\beta -(1-s^-) >0$. If $u\in W^{s(x,y),p(x,y)} (\Omega ; |x|^{-\beta}) \cap L^{{\pl}} _{\loc}(\Omega ; |x| ^{-2\beta})$ and $\zeta : \Omega \rightarrow \mathbb{R}$ is a Lipschitz function with $\spt \zeta {\Subset} \Omega$, then  
$$
\zeta u\in W^{s(x,y),p(x,y)} (\Omega ; |x|^{-\beta}) \cap L^{{\pl}} (\Omega ; |x| ^{-2\beta}).
$$
\end{corollary}
\begin{proof}
By Lemma \ref{27} \ref{167}, we have that $\mv (x) = |x|^{-\beta}$ satisfies condition \eqref{169} in Lemma \ref{176}. {Furthermore, $\mv(x)$ also satisfies \eqref{173} since $n>\beta$}. Hence, employing this last lemma, we conclude the proof of the corollary.\end{proof}

Another useful result is the following:
\begin{lemma} \label{177} Let $\Omega \subset \mathbb{R}^n$ be an open  set. Assume that conditions \eqref{174} and \eqref{175} on $s$ and $p$ hold.
\begin{enumerate}[label=(\roman*)]
\item \label{178} Let $\mv_0$ a weight satisfying $\mv _0^{-1} \in L_{\loc}^{\frac{1}{q-1}} (\mathbb{R} ^n)$, where  $q\in (1,\infty)$. If  $u,v\in W^{1,q} _{\loc} (\Omega ; \mv _0 ) \cap L^{\infty}_{\loc} (\Omega)$, then $u v\in W^{1,q} _{\loc} (\Omega ; \mv _0) \cap L^{\infty}_{\loc} (\Omega)$.

\item \label{179} If  $u,v\in  W^{s(x,y),p(x,y)} (\Omega ; \mv) $, then $\max \{u,v \}\in  W^{s(x,y),p(x,y)} (\Omega ; \mv)$.
\end{enumerate}
\end{lemma}
\begin{proof}
The proof of \ref{178} follows from the observation that \( W^{1,q}_{\loc}(\Omega; \mv _0) \subset W^{1,1}_{\loc}(\Omega) \). The proof of \ref{179} is derived from the definition of the weighted variable exponent fractional Sobolev space and the identity \( \max \{u,v\} = \frac{u+v + |u-v|}{2} \).\end{proof}



\subsection{Sobolev Inequalities}

We say $\mv$ satisfies the {\it doubling condition} (with respect to Lebesgue measure) if
$$
\mv \left(B_{2 r}(x)\right) \leq c \mv \left(B_r(x)\right)
$$
with $c$ independent of $r$ and $x$. We write $\mv \in D_{\infty}$ for such $\mv$.

 {A weight $\mv$ is said to belong to the {\it Muckenhoupt class} $A_q$, for $1 < q<\infty$,  if there exists a positive constant $A$ such that, for every ball $B \subset \mathbb{R}^n$,
\begin{equation*}
\left(\fint _B \mv \,  \dxe \right)\left(\fint _B \mv ^{-\frac{1}{q-1}} \, \dxe \right)^{q-1} \leq A.
\end{equation*}}

{The following lemma will be useful for applying subsequent results--concerning Lipschitz functions--to our problems.
\begin{lemma} 
Assume $U $ is a bounded domain with Lipschitz boundary, and let $ \mv \in A_q $. Suppose $ u \in W^{1, q}(U;\mv) $ for some $ 1 < q < \infty $. Then there exists a sequence of functions $ u_m \in C^{\infty}(\bar{U}) $ such that
$$
u_m \rightarrow u \quad \text { in } \quad  W^{1, q}(U;\mv).
$$
\end{lemma}
\begin{proof}
By \cite[Theorem 2.1.13]{turesson2000potentweig}, there exists an extension $ \hat{u} \in W^{1, q}(\mathbb{R}^n; \mv) $ of $ u$. Then, by \cite[Corollary 2.1.6]{turesson2000potentweig}, there is a sequence $ \hat{u}_m \in C^{\infty}(\mathbb{R}^n) $ such that
$$
\hat{u}_m \rightarrow \hat{u} \quad \text{in } \quad  W^{1, q}(\mathbb{R}^n; \mv).
$$
Defining $ u_m := \hat{u}_m|_U $, we conclude the proof of the lemma.\end{proof}}

For $0<\alpha, p_0, q<\infty$, we consider pairs $\mva, \mvb$ of weights in $D_{\infty}$ that satisfy
\begin{equation}\label{52}
\left(\frac{s}{h}\right)^\alpha \left(\frac{\mva (B_s(y))}{\mva (B_h(y))}\right)^{\frac{1}{q}} \leq c\left(\frac{\mvb (B_s(y))}{\mvb (B_h(y))}\right)^{\frac{1}{p_0}}, \quad 0<s \leq h,
\end{equation}
with $c$ independent of $x$, $s$, and $h$. For the next proposition, we only need to assume that \eqref{52} holds when $\alpha=1$ and $y \in B_{2 h}(x)$ for a certain $x$.

For a measurable function $f : B_r(x_0) \rightarrow \mathbb{R}$, we define
$$
(f)_{B_r(x_0),\mv}:=\frac{1}{\mv(B_r(x_0))} \int_{B_r(x_0)} f \mv \dxe.
$$
\begin{proposition} \label{58} ({\cite[Corollary 1.4]{chani1985weightedpeano}}). Let $u$  be Lipschitz continuous on a ball $B_r(x) \subset \mathbb R^n$, and assume that $1<p_0<q$, $\mva \in D_{\infty}$, the pair $\mva$, $\mvb$ satisfies \eqref{52} with $\alpha=1$ for $y \in B_{2 r}(x)$, and $\mvb \in A_{p_0}$. Then
$$
\left(\fint_{B_r(x)} \left|u-(u)_{B_r(x),\mva}  \right|^q  \dmva \right)^{\frac{1}{q}}  \leq   c r \left( \fint_{B_r(x)}|\nabla u|^{p_0} \dmvb \right)^{\frac{1}{p_0}}
$$
with $c$ independent of $u, r$, and $x$.
\end{proposition}

The following technical result will be used in this and subsequent sections.
\begin{lemma}\label{180}
Assume that $n>\bbt\geq 0$ and $q\in (1,\infty)$. Then:
\begin{enumerate}[label=(\roman*)]
\item \label{181} The weight $\mvbt (x) = |x|^{-\bbt}$ is doubling and belongs to the Muckenhoupt class $A_q$.

\item \label{66} If $x_0 \in \mathbb{R}^n$ and $h\geq s>0$, we have the estimate
\begin{equation}\label{68}
\frac{\mvbt (B_s(x_0))}{\mvbt (B_h(x_0))}\geq C(\bbt ,n) \left(\frac{s}{h}\right)^n, 
\end{equation}
where 
$$
\mvbt (F) := \int _{F} \dx \quad \text { and } \quad  \dx :=  \mvbt \dxe,
$$
for a measurable set $F \subset \mathbb{R}^n$.
\end{enumerate}
\end{lemma}
\begin{proof}
\ref{181}  The proof can be found in one of the following references: \cite[Chapter 7, Weighted Inequalities]{grafakos2014fourier} and \cite[Chapter 5, Weighted Inequalities]{stein1993harmonic}, or \cite[Chapter IX, $A_p$ Weights]{torchinsky1986harmonic}.

\medskip

\ref{66}   Let $h\geq s>0$. For all $x\in B_s (x_0)$, we have
\begin{equation*} 
|x_0|+h \geq |x_0| +s > |x|.
\end{equation*}
Hence,
\begin{equation}\label{54}
\begin{aligned}
I:= & \, \displaystyle \frac{\int _{B_s (x_0)} |x|^{-\bbt} \, \dxe }{\int _{B_h (x_0)} |x|^{-\bbt} \, \dxe} \\[4pt]
\geq & \, \displaystyle \frac{|B_s(x_0)|}{\int _{B_h (x_0)} |x|^{-\bbt} \, \dxe} \cdot \frac{1}{(|x_0| + h)^{\bbt}}.
\end{aligned}
\end{equation}

If $3h\leq |x_0|$, and noting that $|x_0|-h<|x|$ for all $x\in B_h(x_0)$, we have
\begin{equation}\label{55}
\begin{aligned}
I\geq & \, \displaystyle \frac{|B_s(x_0)|}{|B_h(x_0)|}\left( \frac{|x_0| - h}{|x_0| + h} \right)^{\bbt}\\[4pt]
 \geq & \, \displaystyle \frac{|B_s(x_0)|}{|B_h(x_0)|}\left( \frac{1}{2} \right)^{\bbt}.
 \end{aligned}
\end{equation}

If $3h> |x_0|$ and $|x|\geq h$ for all $ x\in B_h(x_0)$, then from \eqref{54},
\begin{equation}\label{56}
\begin{aligned}
I\geq & \, \displaystyle \frac{|B_s(x_0)|}{|B_h(x_0)|}\left( \frac{h}{|x_0| + h} \right)^{\bbt} \\[4pt]
\geq & \, \displaystyle \frac{|B_s(x_0)|}{|B_h(x_0)|}\left( \frac{1}{4} \right)^{\bbt}.
\end{aligned}
\end{equation}

If $3h> |x_0|$ and $|x_1|< h$ for some $ x_1\in B_h(x_0)$, then using \eqref{54}, we obtain
\begin{equation}\label{57}
\begin{aligned}
I\geq  & \, \displaystyle \frac{|B_s(x_0)|}{\int _{B_{3h} (0)} |x|^{-\bbt} \, \dxe} \cdot \frac{1}{(|x_0| + h)^{\bbt}}\\[5pt]
 \geq & \, \displaystyle \frac{(n-\bbt)|B_s(x_0)|}{|\partial B_{1} (0)| (3h)^n}\left(  \frac{3h}{|x_0| + h} \right)^{\bbt} \\[7pt]
 \geq & \, \displaystyle \frac{(n- \bbt)|B_s(x_0)|}{|\partial B_{1} (0)| (3h)^n}\left(  \frac{3}{4} \right)^{\bbt}.
 \end{aligned}
\end{equation}

From \eqref{54}-\eqref{57}, we obtain
$$
\frac{\mvbt (B_s(x_0))}{\mvbt (B_h(x_0))}\geq C(\bbt ,n) \left(\frac{s}{h}\right)^n. 
$$

\end{proof}

As a consequence of Proposition \ref{58} and Lemma \ref{180},  we obtain the following result.
\begin{corollary} \label{59}  Let $u$  be Lipschitz continuous on a ball $B_r(x) \subset \mathbb R^n$, and assume that $1<p_0<q$, $1-n(1/p_0 - 1/q)\geq 0$, and $n>\bbt \geq 0$. Then
$$
\left(\fint_{B_r(x)} \left|u-(u)_{B_r(x),\mvbt}  \right|^q  \, \dx \right)^{\frac{1}{q}}  \leq   c r \left( \fint_{B_r(x)}|\nabla u|^{p_0} \,  \dx \right)^{\frac{1}{p_0}}
$$
with $c$ independent of $u$, $r$, and $x$.

\end{corollary}

\begin{proposition} (\cite[Theorem 1.5]{chani1985weightedpeano}). Let $u$ be Lipschitz continuous on a ball $B_r(x) \subset \mathbb R^n$, and assume that $1<p_0<q$, $\mva \in D_{\infty}$, the pair $\mva$, $\mvb$ satisfies \eqref{52} with $\alpha=1$ for $y \in B_{2 r}(x)$, and $\mvb \in A_{p_0}$. Suppose that $\spt u \subset B_h (x)$. Then
$$
\left(\fint_{B_r(x)} |u|^q  \, \dmva \right)^{\frac{1}{q}}  \leq   c r \left( \fint_{B_r(x)}|\nabla u|^{p_0} \, \dmvb \right)^{\frac{1}{p_0}},
$$
with $c$ independent of $u$, $r$, and $x$.
\end{proposition}

\begin{corollary} \label{70}
Let $u$ be Lipschitz continuous on a ball $B_r(x) \subset \mathbb R^n$ with $\spt u \subset B_r(x)$. Let
\begin{equation}\label{69}
\kappa = \left\{
 \begin{aligned}
 & \, \displaystyle \frac{n}{n-p_0} & \, \displaystyle  & \, \displaystyle  \text { if } 1<p_0<n,  \\
 &2 & \, \displaystyle  & \, \displaystyle  \text { if } p_0 \geq n.
\end{aligned} \right.
\end{equation}

There exists a  constant $c>0$, independent of $u$, $r$, and $x$,  such that
\begin{equation}\label{2}
\left(\fint_{B_r(x)} |u|^{\kappa p_0}  \, \dx \right)^{\frac{1}{\kappa p_0}}  \leq   c r \left( \fint_{B_r(x)}|\nabla u|^{p_0} \, \dx \right)^{\frac{1}{p_0}}.
\end{equation}

\end{corollary}




\subsection{Definition and Auxiliary Results for Weak Solutions}
Next, we define the notion of a weak solution to \eqref{4}.
\begin{definition} \label{166} Let $\Omega$ be a  bounded domain, and assume conditions \ref{162} and \ref{163} are satisfied, and that $n-(1-s^-)>\beta \geq 0$. A function   $u \in W_{\loc}^{1, \pl }(\Omega ; \mvbd) \cap W^{s(x,y),p(x,y)} (\Omega ; \mvbb)$ is a \textnormal{weak subsolution} of \eqref{4} if, for every  nonnegative test function $\phi \in W_{\loc} ^{1, \pl}(\Omega ; \mvbd ) \cap W^{s(x,y),p(x,y)} (\Omega ; \mvbb)$ with $\spt \phi \subset \Omega $, we have
\begin{equation} \label{3}
\int_{\Omega}|\nabla u|^{\pl-2} \nabla u       {\cdot} \nabla \phi \, \dmvbd+\int_{\Omega} \int_{\Omega} \mathcal{A}(u, x, y)(\phi(x)-\phi(y)) \, \dxu  (x,y)\leq 0,
\end{equation}
where 
$$
\mathcal{A}(u,x, y)=|u(x)-u(y)|^{p(x,y)-2}(u(x)-u(y)) \quad \text { and } \quad  \dxu (x,y) =\frac{1}{|x|^\beta |y|^\beta} K(x, y) \, \dxe \dye .
$$
 
Analogously,  $u\in W_{\loc}^{1, \pl }(\Omega ; \mvbd) \cap W^{s(x,y),p(x,y)} (\Omega ; \mvbb)$ is a \textnormal{weak supersolution} of \eqref{4} if the integral in \eqref{3} is nonnegative for every nonnegative test function $\phi \in W_{\loc} ^{1, \pl}(\Omega ; \mvbd ) \cap W^{s(x,y),p(x,y)} (\Omega ; \mvbb)$ with $\spt \phi \subset \Omega$.

 A function $u\in W_{\loc}^{1, \pl }(\Omega ; \mvbd) \cap W^{s(x,y),p(x,y)} (\Omega ; \mvbb)$ is a \textnormal{weak solution} of \eqref{4} if equality holds in \eqref{3} for every $\phi \in W_{\loc} ^{1, \pl}(\Omega ; \mvbd ) \cap W^{s(x,y),p(x,y)} (\Omega ; \mvbb)$ with $\spt \phi \subset \Omega$, and without any sign restriction on $\phi$.
\end{definition}


It follows directly from Definition \ref{166} that $u$ is a weak subsolution of \eqref{4} if and only if $-u$ is a weak supersolution of \eqref{4}. Moreover, for any $c \in \mathbb{R}$, $u+c$ is a weak solution of \eqref{4} if and only if $u$ is a weak solution of \eqref{4}.

We now discuss further structural properties of weak solutions. We denote the positive and negative parts of $a \in \mathbb{R}$ by $a_{+}=\max \{a, 0\}$ and $a_{-}=\max \{-a, 0\}$, respectively. 

\begin{lemma} \label{164} A function $u\in W_{\loc}^{1, \pl }(\Omega ; \mvbd) \cap W^{s(x,y),p(x,y)} (\Omega ; \mvbb)$ is a weak solution of \eqref{4} if and only if $u$ is both a weak subsolution and a weak supersolution of \eqref{4}.
\end{lemma}

\begin{proof} 
The proof follows the same argument as in  \cite[Lemma 2.7]{zbMATH07576867}.\end{proof}

Next, we show that the property of being a weak subsolution is preserved under taking the positive part. As a direct consequence $u_{-}$ is a weak subsolution of \eqref{4}, whenever $u$ is a weak supersolution of \eqref{4}.
\begin{lemma} \label{165} Assume that $u\in  W_{\loc}^{1, \pl }(\Omega ; \mvbd) \cap W^{s(x,y),p(x,y)} (\Omega ; \mvbb) $ is a weak subsolution of \eqref{4}. Then $u_{+} \in  W_{\loc}^{1, \pl }(\Omega ; \mvbd) \cap W^{s(x,y),p(x,y)} (\Omega ; \mvbb)$ is a weak subsolution of \eqref{4}.
\end{lemma}

\begin{proof} The argument follows the same reasoning as in \cite[Lemma 2.8]{zbMATH07576867}, taking into account Lemma \ref{177}. Since {this lemma allows us to claim that} 
$$
u_{k}=\min \left\{k u_{+}, 1 \right\}\in W_{\loc}^{1, \pl }(\Omega ; \mvbd) \cap W^{s(x,y),p(x,y)} (\Omega ; \mvbb) \cap L^\infty (\Omega), \quad  \text { for } \quad  k=1,2, \ldots ,
$$
 and  for any $\ell >0$,
$$
T_{\ell} (\phi ) \in W_{\loc}^{1, \pl }(\Omega ; \mvbd) \cap W^{s(x,y),p(x,y)} (\Omega ; \mvbb) \cap L^\infty (\Omega),
$$
where $\phi \in W_{\loc}^{1, \pl }(\Omega ; \mvbd) \cap W^{s(x,y),p(x,y)} (\Omega ; \mvbb)$, with $\spt \phi \subset \Omega$. The truncation function  $T_\ell : \mathbb{R} \rightarrow \mathbb{R}$, $\ell>0$, is defined by
\begin{align*}
T_\ell (t)=\left\{
\begin{aligned}
&t & & \text { if } |t| \leq \ell,\\
&\operatorname{sign}(t) \ell & & \text { if } |t| > \ell.
\end{aligned}\right.
\end{align*}
Hence, by  Lemma \ref{177}, $u_k T_{\ell} (\phi ) \in W_{\loc}^{1, \pl }(\Omega ; \mvbd) \cap W^{s(x,y),p(x,y)} (\Omega ; \mvbb) \cap L^\infty (\Omega)$.

To complete the proof of Lemma \ref{165}, we refer to the proof of \cite[Lemma 2.8]{zbMATH07576867}.\end{proof}

\section{Energy estimates} \label{185}
The following energy estimate will be crucial for us.


\begin{lemma} \label{7} (\cite[Lemma 3.4]{chaker2023local}). Let $a, b \geq 0$, $\tau_1, \tau_2 \in[0,1]$, and $1<p_1 \leq p_2 \leq p_3 $. Then,
\begin{align*}
|a-b|^{p_2 -2}(a-b)&\left(a \tau_1^{p_3 }-b \tau_2^{p_3 }\right) \\[4pt]
\geq & \, \displaystyle \frac{1}{2}|a-b|^{p_2}\left(\max \left\{\tau_1, \tau_2\right\} \right)^{p_3 }-C\left(\max \{a, b\}\right)^{p_2}\left|\tau_1-\tau_2\right|^{p_2},
\end{align*}
where $C=\frac{p_3}{p_1} (2p_3)^{p_3 -1}$.
\end{lemma}

\begin{lemma} \label{73} Let $u$ be a weak subsolution of \eqref{4}, and define $w=(u-k)_{+}$ with $k \in \mathbb{R}$. Then
 \begin{equation} \label{10}
\begin{aligned}
   \int_{B_{r}(x_{0})} & \psi^{\pl}|\nabla w|^{\pl} \, \dmvbd + \int_{B_{r}(x_{0})} \int_{B_{r}(x_{0})}\left(\max \left\{ \psi(x) , \psi(y) \right\} \right) ^{\pl }|w(x) -w(y)|^{p(x,y)} \, \dxu (x,y)  \\[7pt]
 \leq & \, \displaystyle 2\left[2(\pl -1)\right] ^{\pl -1}\int_{B_{r}(x_{0})} w^{\pl}|\nabla \psi|^{\pl} \, \dmvbd\\[5pt]
& \, \displaystyle   + \frac{ \left(2\pl\right)^{\pl } }{p^-} \int_{B_{r}(x_{0})} \int_{B_{r}(x_{0})} \max \{w(x), w(y)\}^{p(x,y)}|\psi(x)-\psi(y)|^{p(x,y)} \, \dxu (x,y)\\[5pt]
& \, \displaystyle   + 4 \int_{\Omega \setminus B_{r}(x_{0}) } \int _{B_{r}(x_{0})} w(y)^{p(x, y)-1} w(x) \psi(x) ^{\pl } \, \dxu (x,y),
\end{aligned}
\end{equation}
whenever $B_{r}(x_{0}) \subset \Omega$ and $\psi \in C_{c}^{\infty}(B_{r}(x_{0}))$ is a nonnegative function.

 If $u$ is a weak supersolution of \eqref{4}, the estimate in \eqref{10} holds with $w=(u-k)_{-}$.
\end{lemma}

\begin{proof}
  Let $u$ be a weak subsolution of \eqref{4}. For $w=(u-k)_{+}$, choose $\phi=w \psi^{\pl}$ as a test function in \eqref{3}. We obtain
\begin{align*}
0 & \, \displaystyle  \geq \int_{B_{r}(x_{0})}  |\nabla u|^{\pl-2} \nabla u       {\cdot} \nabla(w \psi^{\pl}) \, \dmvbd+\int_{\Omega} \int_{\Omega}  \mathcal{A}(u,x, y)(w(x) \psi(x)^{\pl}-w(y) \psi(y)^{\pl}) \, \dxu  (x,y) \\[5pt]
& \, \displaystyle  =: I+J.
\end{align*}

\medskip

{\bf Step 1.}  We have
\begin{align*}
 I  = & \, \displaystyle \int_{B_{r}(x_{0})} |\nabla u|^{\pl-2} \nabla u      {\cdot} \left( \psi^{\pl } \nabla w  + \pl  \psi^{\pl  -1 }\nabla \psi \right) \, \dmvbd \\[5pt]
 \geq  & \, \displaystyle \int_{B_{r}(x_{0})}  \psi^{\pl }|\nabla w|^{\pl} \, \dmvbd+  \int_{B_{r}(x_{0})}  \pl  w\psi^{\pl  -1 }|\nabla u|^{\pl-2} \nabla u      {\cdot} \nabla \psi  \, \dmvbd .
\end{align*}

Using Young's inequality,
\begin{align*}
  -w\psi^{\pl  -1 }|\nabla u|^{\pl-2} \nabla u      {\cdot} \nabla \psi\leq & \, \displaystyle \psi ^{\pl } \left( \epsilon \frac{\pl-1}{\pl}  |\nabla w|^{\pl} + \epsilon ^{-\pl +1} \frac{1}{\pl} \psi ^{-\pl} w^{\pl}|\nabla \psi|^{\pl} \right)\\[5pt]
 \leq & \, \displaystyle \epsilon \frac{\pl -1}{\pl } \psi ^{\pl } |\nabla w|^{\pl} +  \epsilon ^{-\pl +1} \frac{1}{\pl }  w^{\pl }|\nabla \psi|^{\pl} .
\end{align*}
Choosing $\epsilon = \frac{1}{2(\pl -1)}$, we obtain
\begin{equation}\label{6}
\begin{aligned}
I\geq & \, \displaystyle \frac{1}{2}\int_{B_{r}(x_{0})}  \psi^{\pl }|\nabla w|^{\pl} \, \dmvbd \\[5pt]
  & \, \displaystyle -   \left[2 (\pl  -1)\right] ^{\pl - 1}   \int_{B_{r}(x_{0})} w^{\pl}| \nabla \psi |^{\pl} \, \dmvbd .
\end{aligned}
\end{equation}

\medskip

{\bf Step 2.} Now we estimate $J$. We have 
$$
\begin{aligned}
  |u(x)&-u(y)|^{p(x,y)-2}(u(x)-u(y))\big(w(x) \psi(x) ^{\pl }-w(y) \psi (y) ^{\pl } \big) \\[5pt]
 = & \, \displaystyle |u(x)-u(y)|^{p(x,y) -2}(u(x)-u(y))\left[ (u(x)-k)_{+} \psi (x) ^{\pl } -(u(y)-k)_{+} \psi (y)^{\pl } \right] \\[5pt]
 = & \, \displaystyle \left\{ \begin{aligned}
&   |w(x)-w(y)|^{p(x,y)-2} (w(x)-w(y)) \left[ w(x) \psi (x) ^{\pl } -w(y) \psi (y)^{\pl } \right]   &   &   \text { if } u(x), u(y)>k , \\
&   (u(x)-u(y))^{p(x,y) -1} w(x) \psi (x)^{\pl }  &   &   \text { if } u(x)\geq k, u(y) \leq k, \\
&   (u(y)-u(x))^{p(x,y)-1} w(y) \psi (y)^{\pl }  &   &   \text { if } u(x)\leq k, u(y) \geq k, \\
&   0  &   &    \text { otherwise}.
\end{aligned}\right. \\[7pt]
 \geq & \, \displaystyle |w(x)-w(y)|^{p(x,y)-2} (w(x)-w(y))\big(w(x) \psi (x) ^{\pl }-w(y) \psi (y) ^{\pl } \big) .
\end{aligned}
$$

From Lemma \ref{7}, for $(x,y) \in B_r(x_0) \times B_r(x_0)$,
\begin{equation}\label{8}  
\begin{aligned}
 & |w(x)-w(y)|^{p(x,y)-2} (w(x)-w(y))  \big(w(x) \psi (x) ^{\pl } -w(y) \psi (y) ^{\pl } \big)\\[5pt]
& \qquad \begin{aligned}
 \geq & \, \displaystyle \frac{1}{2} |w(x)-w(y)|^{p(x,y)} \left(\max \left\{ \psi  (x) , \psi  (y)\right\}\right)^{\pl } \\[5pt]
 & \, \displaystyle -   \frac{\pl}{p^-} (2\pl)^{\pl -1} \left( \max \left\{ w(x), w(y)\right\} \right)^{p(x,y)} |\psi (x) - \psi (y)| ^{p(x,y)}      {.}
 \end{aligned}
\end{aligned}
\end{equation}

If $(x,y) \in (\Omega \setminus B_r(x_0) ) \times  B_r(x_0) $, then
\begin{equation} \label{12}
\begin{aligned}
  |w(x)-w(y)&|^{p(x,y)-2} (w(x)-w(y)) \bigl(w(x) \psi (x)^{\pl } -w(y) \psi (y) ^{\pl } \bigr)\\[5pt]
 = & \, \displaystyle -|w(x)-w(y)|^{p(x,y)-2} \bigl(w(x)-w(y)\bigr)w(y) \psi (y) ^{\pl }\\[7pt]
   \geq & \, \displaystyle - w (x) ^{p(x,y) -1}  w(y)\psi (y) ^{\pl }.
\end{aligned}
\end{equation}
Similarly, if $(x,y) \in   B_r(x_0) \times (\Omega \setminus B_r(x_0) ) $, then
\begin{equation} \label{13}
\begin{aligned}
 |w(x)-w(y)&|^{p(x,y)-2} (w(x)-w(y)) \bigl(  w(x) \psi (x)^{\pl }  -w(y) \psi (y) ^{\pl } \bigr)\\[5pt]
 = & \, \displaystyle  |w(x)-w(y)|^{p(x,y)-2} \bigl(w(x)-w(y)\bigr)w(x) \psi (x) ^{\pl } \\[7pt]
  \geq & \, \displaystyle - w(y) ^{p(x,y) -1} w(x)\psi (x) ^{\pl }.
\end{aligned}
\end{equation}

Hence,
\begin{equation}\label{11}
\begin{aligned}
J \geq & \, \displaystyle \frac{1}{2}  \int _{{B_r(x_0)}} \int _{{B_r(x_0)}} |w(x)-w(y)|^{p(x,y)} \bigl(\max \left\{ \psi  (x) , \psi  (y)\right\}\bigr)^{\pl } \, \dxu (x,y) \\[5pt]
& \, \displaystyle  - \frac{\pl \left(2\pl\right)^{\pl -1} }{p^-}\int _{{B_r(x_0)}} \int _{{B_r(x_0)}} \bigl( \max \left\{ w(x), w(y)\right\} \bigr)^{p(x,y)} |\psi (x) - \psi (y)| ^{p(x,y)} \, \dxu (x,y)\\[5pt]
& \, \displaystyle   - 2 \int_{\Omega  \setminus B_r (x_0)} \int_{B_r (x_0)} w(y)^{p(x, y)-1} w(x) \psi(x)^{\pl } \, \dxu (x,y) \textcolor{blue}{.}
\end{aligned}
\end{equation}

By \eqref{6} - \eqref{11}, we obtain \eqref{10}. In the case of a weak supersolution, the estimate in \eqref{10}  follows by applying the obtained result to $-u$.\end{proof}

Recall that for a measurable function $v : \Omega \rightarrow \mathbb{R}$, we define
\begin{equation*}
\tl (v ; x_0, r, \rho):=\underset{x \in B_\rho (x_0)}{\esssup} \, \int_{\Omega \setminus B_r(x_0)} \frac{|v(y)|^{p(x, y)-1}}{|y-x_0|^{n+s(x, y) p(x, y)}} \cdot \frac{1}{|y|^\beta} \, \dye ,
\end{equation*}
where $B_r (x_0)  \subset B_\rho (x_0) \subset \Omega$.

{For the following result, we need the auxiliary 
\begin{lemma} \label{143}
Assume that  $1<q _0 \leq q_1 <\infty$. Let $a,b\in \mathbb{R}$ and $\ell>1$. Then
$$
|a|^{q_0} -\ell ^{q_1-1} |b|^{q_0} \leq \left( \frac{\ell}{\ell -1}\right)^{q_0-1} |a-b|^{q_0}.
$$ 
\end{lemma}
\begin{proof}
By the convexity of $t \mapsto t^{q_0}$ on $[0, \infty)$, we have
\begin{align*}
|a|^{q_0} \leq & \, \left(|b| +|a-b|  \right)^{q_0}\\[4pt]
= & \,  \ell ^{q_0} \left(\frac{1}{\ell}|b| +\frac{\ell-1}{\ell} \cdot \frac{|a-b|}{\ell -1}  \right)^{q_0}\\[4pt]
\leq & \,  \ell ^{q_1-1} |b|^q + \left( \frac{\ell}{\ell -1}\right)^{q_0-1} |a-b|^{q_0}.
\end{align*}
This concludes the proof of the lemma.\end{proof}}

Next, we obtain a logarithmic energy estimate.
\begin{lemma} \label{h51} Suppose that {property \ref{144}} holds, and that
 $$
 n-2\beta -1 +s^->0,  \quad \beta\geq 0, \quad \text { and } \quad n-\beta>s^{+} \pl.
 $$ 
 Furthermore, let $d>0$ and $0<2r\leq R \leq 1/2$. Assume that $u$ is a weak supersolution of \eqref{4} such that $u \geq 0$ in $B_{R}(x_{0}) \subset \Omega$. 

If $0\notin \overline{B_{3R}(x_{0})}$ and  $0<R_0\leq|x_0|/2$, then there exists a constant $c_1 = c_1 (s^{-},  s^{+},  p^{-}, \pl , \beta, \Lambda, n)>0$ such that
\begin{equation}\label{h48}
\begin{aligned}
  \frac{1}{\mvbd \left(B_r(x_0)\right)}&\int_{B_{r}(x_{0})}|\nabla \ln (u+d)|^{\pl} \, \dmvbd  \\[5pt]
 \leq & \, \displaystyle c_1 \left(\frac{|x_0|+1}{R_0}\right)^\beta \max  \left\{ 1 , d^{p^{-}- \pl }\right\}  \left(r^{-s^+  \pl}+\frac{R_0^{n-s^+  \pl}}{|x_0|^n}\right) \\[5pt]
& \, \displaystyle  +c_1 r^{ -\pl} +c_1(|x_0|+1)^\beta d^{1-\pl}\tl (u_- ; x_0 , R , 3r/2).
\end{aligned}
\end{equation}

If $x_0=0$, then there exists a constant $c_2 = c_2 (s^{-}, p^{-}, s^{+}, \pl , \beta, \Lambda, n)>0$ such that
\begin{equation}\label{h49}
\begin{aligned}
&   \int_{B_{r}(x_{0})}|\nabla \ln (u+d)|^{\pl} \, \dmvbd  \\[5pt]
& \qquad \leq  c_2 \max \left\{1, d^{p^{-}-\pl}\right\} r^{n-s^+  \pl -2 \beta} + c_2 r^{n-\pl -2\beta}+c_2 d^{1-\pl } r^{n-\beta} \tl (u_{-} ; x_0,  R , 3  r / 2 ).
\end{aligned}
\end{equation}
\end{lemma}

\begin{proof} Let $\psi \in C_{c}^{\infty}(B_{3r/2}(x_{0}))$ be such that 
$$
0 \leq \psi \leq 1 \text  { in } B_{3 r/2}(x_{0}), \quad  \psi=1 \text {  in } B_{r}(x_{0}), \quad \text { and } \quad |\nabla \psi| \leq \frac{{C_0}  }{ r} \text { in } B_{3 r/2}(x_{0}).
$$
 By choosing $\phi=(u+d)^{1-\pl} \psi^{\pl}$ as a test function in \eqref{3}, we obtain
\begin{align*}
0 \leq & \, \displaystyle  \int_{B_{2 r}(x_{0})} \int_{B_{2 r}(x_{0})} \mathcal{A}(u,x, y) \left[  (u(x)+d)^{1-\pl} \psi(x)^{\pl}-(u(y)+d)^{1-\pl} \psi(y)^{\pl}  \right] \, \dxu  (x,y) \\[5pt]
& \, \displaystyle   +2 \int_{\Omega \setminus B_{2 r}(x_{0})} \int_{B_{2 r}(x_{0})} \mathcal{A}(u,x, y)(u(x)+d)^{1-\pl} \psi(x)^{\pl} \, \dxu (x,y)  \\[5pt]
& \, \displaystyle   + \int_{B_{2 r}(x_{0})}|\nabla u|^{\pl-2} \nabla u       {\cdot} \nabla \left[ (u+d)^{1-\pl} \psi^{\pl} \right]  \, \dmvbd \\[7pt]
 = & \, \displaystyle  I_{1}+I_{2}+I_{3}.
\end{align*}

\medskip

{\bf Estimate of $I_{1}$:}  If $u(x)>u(y)$, {take
$$
q_0 =p(x,y), \quad q_1= \pl, \quad  a= \psi(x) ^{\frac{\pl}{p(x,y)}}, \quad b= \psi(y) ^{\frac{\pl}{p(x,y)}}, \quad  \ell = \frac{u(x)+d}{u(y)+d}
$$
in Lemma \ref{143}. Then we have}
\begin{align}
  |u(x)-& u(y)|^{p(x,y)-2}(u(x)-u(y))  \left[  \frac{\psi(x)^{\pl}}{(u(x)+d)^{\pl-1}}-\frac{\psi(y) ^{\pl }}{(u(y)+d)^{\pl-1}} \right] \nonumber \\[5pt]
   {=} & \, \displaystyle   {\frac{(u(x)-u(y))^{p(x,y) -1}}{(u(x)+d)^{\pl -1}}   \left[  \psi(x)^{\pl}  -\left( \frac{u(x)+d }{u(y)+d} \right)^{\pl-1} \psi(y) ^{\pl } \right] } \nonumber \\[5pt] 
   {\leq} & \, \displaystyle   {\frac{(u(x)-u(y))^{p(x,y) -1}}{(u(x)+d)^{\pl -1}} \left( \frac{u(x)+d }{u(x)-u(y)} \right)^{p(x,y)-1}   \left|  \psi(x)^{\frac{\pl}{p(x,y)}}  - \psi(y) ^{\frac{\pl}{p(x,y)}} \right|^{p(x,y)} } \nonumber \\[5pt] 
   {\leq} & \, \displaystyle  {\frac{1}{(u(x)+d)^{\pl -p(x,y) }}   \left|  \psi(x)^{\frac{\pl}{p(x,y)}}  - \psi(y) ^{\frac{\pl}{p(x,y)}} \right|^{p(x,y)}} \nonumber\\[7pt]
    {\leq}  & \, \displaystyle {\max \{ 1  , d^{p^-  - \pl} \}  \left|  \psi(x)^{\frac{\pl}{p(x,y)}}  - \psi(y) ^{\frac{\pl}{p(x,y)}} \right|^{p(x,y)}}. \label{h24}
\end{align}
A similar estimate is obtained when $u(x)<u(y)$.

Write
$$
p^-_{2r} = \inf _{B_{2 r}(x_0) \times B_{2 r}(x_0) } p \quad \text { and } \quad p^+_{2r} = \sup _{B_{2 r}(x_0) \times B_{2 r}(x_0) } p.
$$

     {Since
\begin{equation*}\label{146}
\left|\psi (x) ^{q_0}-\psi (y)^{q_0} \right| \leq q_0 \left(\sup _{B_{2r} (x_0)} \psi ^{q_0 -1} \left|\nabla \psi \right|\right)|x-y|
\end{equation*}
for all $x, y \in B_{2r}(x_0)$ and $q_0 \geq 1$, there exists a constant $C_1 = C_1(C_0, \pl) > 1$} such that
\begin{equation}\label{h22}
\begin{aligned}
\left|  \psi(x)^{\frac{\pl}{p(x,y)}}  - \psi(y) ^{\frac{\pl}{p(x,y)}} \right|^{p(x,y)} \leq  &  C_1 ^{p(x,y)}r^{-p(x,y)}|x-y|^{p(x,y)}\\[5pt]
\leq & C_1 ^{\pl}r^{-p^+_{2r}}|x-y|^{p(x,y)}
\end{aligned}
\end{equation} 
for all $x,y \in B_{2r} (x_0)$.

If $0\not\in \overline{B_{3R}(x_0)}$, then from \eqref{h24}, \eqref{h22}, and condition  \ref{144}, we have
\begin{equation} \label{h25}
\begin{aligned}
 I_1 \leq & \, \displaystyle \frac{C_1^{\pl}\Lambda \max \{ 1 , d^{p^-  - \pl} \}}{(|x_0| -2r)^{\beta}} r^{-{p^+_{2r}}}\int _{B_{2r} (x_0)} \frac{1}{|y|^\beta} \int_{B_{4r}(y)}  |x-y|^{-n+(1 - s^+ ){p^-_{2r}}}  \, \dxe \dye\\[5pt]
= & \, \displaystyle  \frac{2 ^{ 2(1-s^+ ){p^-_{2r}}} |\partial B_1 (0)|  C_1^{\pl}\Lambda  }{  (1-s^+ ){p^-_{2r}}  } \left(\frac{ |x_0| +2r }{|x_0| -2r} \right)^{\beta} \\[5pt]
& \, \displaystyle   \cdot \max \{ 1 , d^{p^-  - \pl} \}  {r^{p^-_{2r} - p^+_{2r}}} r^{-s^+ \pl } \mvbd \left(B_{2r} (x_0) \right)\\[7pt]
\leq & \, \displaystyle   c(\beta ,n) \frac{2 ^{ 2(1-s^+ )\pl}C_1^{\pl}\Lambda (2)^{\beta}}{  (1-s^+ ){p^-} } {\mathtt{C} 2^{\pl - p^-}}\\[5pt]
& \, \displaystyle  \cdot \max \{ 1 , d^{p^-  - \pl} \}  r^{-s^+ \pl } \mvbd \left(B_{3r/2} (x_0) \right),
\end{aligned}
\end{equation}
because $6r\leq |x_0|$, and by Lemma \ref{180} \ref{66}.

If $x_0=0$, then employing Lemma \ref{27} \ref{23}, we get
\begin{equation} \label{h26}
\begin{aligned}
I_1 \leq & \, \displaystyle C_2  r^{-{p^+_{2r}}}\int _{B_{2r} (0)}   \frac{1}{|y|^{\beta}} \int_{B_{2r}(0)} \frac{1}{|x-y|^{n- (1-s^+ )} }\frac{1}{|x|^{\beta}} |x-y|^{(1-s^+ ) ({p^-_{2r}}-1)}  \, \dxe \dye \\[5pt]
 \leq & \, \displaystyle C_2 C_3 r^{-{p^+_{2r}}+(1-s^+ ) ({p^-_{2r}}-1)} \int_{B_{2r} (0)} \frac{(2r)^{2(1-s^+ )}}{|y|^{2\beta + (1-s^+ )}} \, \dye\\[7pt]
      {\leq} & \, \displaystyle C_2 C_3 \frac{2^{n-2\beta +1-s^+ } |\partial B_1 (0)|}{\left[ n-2\beta -(1-s^+ )\right]} {r^{ p_{2r}^{-} - p_{2 r}^{+}} }r^{n-2\beta-s^+  \pl}.
\end{aligned}
\end{equation}
Here $C_2= C_1^{\pl}\Lambda \max \{1 , d^{p^-  - \pl} \}$ and $C_3 = C_3 (s^+,\pl,\beta,n)>0$. {Observe that, by condition \eqref{145} in \ref{144}, $r^{ p_{2r}^{-} - p_{2 r}^{+}}$ is bounded by  $\mathtt{C} 2^{\pl - p^-}$.}

\medskip

{\bf Estimate of $I_{2}$:} We follow the lines of the proof of \cite[p 1290, Lemma 1.3]{di2016localfract}.

Notice that when $y \in B_R(x_0)$, we have $u(y) \geq 0$, and so
\begin{equation}\label{95}
\begin{aligned}
 \frac{(u(x)-u(y))_{+}^{p(x,y)-1}}{(d+u(x))^{\pl-1}} = & \, \displaystyle \left[\frac{(u(x)-u(y))_{+}}{d+u(x)}\right]^{p(x,y)-1}  \cdot (d+u(x))^{p(x,y)-\pl }   \\[5pt]
 \leq & \, \displaystyle  \max \left\{ 1  , d^{ p^-  -\pl }  \right\}, 
\end{aligned}
\end{equation}
for all $x \in B_{2 r} (x_0)$ and $y \in B_R(x_0)$. 

Moreover, when $y \in \Omega \setminus B_R(x_0)$,
\begin{equation}\label{96}
\begin{aligned}
& \, \displaystyle \frac{(u(x)-u(y))_{+}^{p (x,y) -1} }{(d+u(x))^{\pl-1}}\\[5pt]
& \, \displaystyle \quad \ \leq 2^{p(x,y)-1}\left[  \left( \frac{u(x)}{d+u(x)} \right)^{p(x,y)-1}  \cdot (d+u(x))^{p(x,y)-\pl}+\frac{(u_{-}(y))^{p(x,y)-1}}{(d+u(x))^{\pl-1}}\right]\\[7pt]
& \, \displaystyle  \quad \ \leq 2^{\pl-1}\left[   \max \left\{ 1 , d^{ p^-  -\pl }  \right\} +d^{1-\pl} (u_{-}(y))^{p(x,y)-1} \right],
\end{aligned}
\end{equation}
for all $x \in B_{2 r} (x_0) $.

We have $\spt \psi \subset B_{3r/2}(x_0) $. From \eqref{95} and \eqref{96}, we see that
$$
\begin{aligned}
 I_2 \leq & \, \displaystyle   2  \int_{B_{2 r} (x_0)} \int_{B_R(x_0) \setminus B_{2 r} (x_0)} K(x, y)(u(x)-u(y))_{+}^{p(x,y)-1}(d+u(x))^{1-\pl} \psi(x)^{\pl} \frac{1}{|x|^\beta|y|^\beta}\,  \dye \dxe \\[4pt]
& \, \displaystyle   +2  \int_{B_{2 r} (x_0)} \int_{\Omega \setminus B_R (x_0)} K(x, y)(u(x)-u(y))_{+}^{p(x,y)-1}(d+u(x))^{1-\pl} \psi(x)^{\pl} \frac{1}{|x|^\beta|y|^\beta}\, \dye \dxe   \\[4pt]
 \leq  & \, \displaystyle  2 \Lambda \max \left\{1 , d^{p^{-}-\pl}\right\}  \left[\underbrace{ \int_{B_{3 r/2} (x_0)} \int_{{B_{R} (x_0)} \setminus B_{2 r} (x_0)} \frac{ 1}{|x-y|^{n+s^+  \pl}}  \cdot \frac{1}{|x|^\beta |y|^\beta} \, \dye \dxe }_{I_{2,1}}\right.\\[4pt]
& \, \displaystyle   + \left. 2^{\pl -1} \underbrace{\int_{B_{3 r/2} (x_0)} \int_{\Omega \setminus B_{2 r} (x_0)} \left( \frac{ 1}{|x-y|^{n+s^- p^-}} + \frac{ 1}{|x-y|^{n+s^+ \pl}} \right)\frac{1}{|x|^\beta |y|^\beta} \, \dye \dxe  }_{I_{2,2}}\right]\\[4pt]
& \, \displaystyle   + 2^{{\pl}} \Lambda d^{1-\pl} \underbrace{ \int_{B_{3 r/2} (x_0)}  \int_{\Omega \setminus B_R (x_0)}  \frac{(u_{-}(y))^{p(x,y)-1} }{|x-y|^{n+s(x,y)p(x,y)}} \cdot \frac{1}{|x|^\beta |y|^\beta}  \, \dye \dxe}_{I_{2,3}}.
\end{aligned}
$$

If $0\not\in \overline{B_{3R}(x_0)}$ and $0<R_0 \leq |x_0|/2$,  then
\begin{equation}\label{106}
\begin{aligned}
I_{2,1}\leq & \, \displaystyle  \frac{1}{R_0^\beta}  \int_{B_{3 r / 2}(x_0)} \frac{1}{|x|^\beta} \int_{\mathbb{R}^n \setminus \left( B_{2 r}(x_0) \cup B_{R_0} (0)\right)} \frac{4^{n+s^+ \pl}}{|x_0-y|^{n+s^+  \pl  }} \,  \dye \dxe\\[5pt]
& \, \displaystyle  + \int_{B_{3 r / 2}(x_0)} \frac{1}{|x|^\beta} \int_{B_{R_0} (0)} \frac{4^{n+s^+ \pl}}{|x_0-y|^{n+s^+  \pl  }}  \cdot \frac{1}{|y|^\beta} \,  \dye \dxe .
\end{aligned}
\end{equation}
In the estimate above, we used the fact that for $(x, y) \in B_{3 r / 2} (x_0) \times \mathbb{R}^n \setminus B_{2r} (x_0)$, it holds that $|x-y|\geq r/2$, and moreover,
\begin{equation} \label{h37}
\frac{|y-x_0|}{|y-x|} \leq 1+ \frac{ |x-x_0|}{|x-y|} \leq 1+ \frac{3/2}{1/2}=4.
\end{equation}
By  Lemma \ref{27} \ref{28}, we deduce
\begin{equation}\label{105}
\begin{aligned}
   \frac{1}{R_0^\beta}  &\int _{\mathbb R ^n \setminus B_{2r} (x_0)} \frac{1}{|x_0-y|^{n+s^+  \pl  }} \, \dye  + \int _{B_{R_0} (0)} \frac{1}{|x_0-y|^{n+s^+  \pl  }} \cdot \frac{1}{|y|^\beta} \, \dye \\[5pt]
 \leq & \, \displaystyle    C_4(s^+,\pl,\beta,n) \frac{(2r)^{-s^+  \pl}}{R_0 ^\beta}\cdot \frac{\left(|x_0| + \frac{3r}{2} \right)^\beta }{|x|^\beta} 
   + C_4  \frac{R_0^{n  -s^+  \pl   -\beta}}{|x_0|^{n}} \cdot \frac{\left(|x_0| + \frac{3r}{2} \right)^\beta}{|x|^\beta} ,
\end{aligned}
\end{equation}
because 
\begin{equation}\label{182}
\frac{|x_0|+ \frac{3r}{2}}{|x|}> 1 \quad \text { in } \quad B_{3r/2} (x_0).
\end{equation}

From \eqref{106} and \eqref{105}, it follows that
\begin{equation}\label{h36}
\begin{aligned}
I_{2,1}\leq & \, \displaystyle C_5(s^+,\pl,\beta,n)   \frac{\left(|x_0| + \frac{ 3r}{2}\right) ^{\beta}}{R_0 ^\beta} \left(2r\right)^{-s^+  \pl  } \mvbd \left(B_{3r/2} (x_0)\right)\\[5pt]
& \, \displaystyle  + C_5 \frac{\left(|x_0| + \frac{ 3r}{2}\right) ^{\beta} R_0^{n -s^+  \pl   -\beta} }{|x_0|^{n}}\mvbd \left(B_{3r/2} (x_0)\right).
\end{aligned}
\end{equation}

Now we estimate $I_{2,1}$ when $x_0=0$. From \eqref{h37},
\begin{equation}\label{h38} 
\begin{aligned}
I_{2,1} \leq & \, \displaystyle 2^{2(n+s^+  \pl  )} \int_{B_{3 r / 2}(0)}  \frac{1}{|x|^\beta} \int_{\mathbb{R}^n \setminus B_{2 r}(0)} \frac{1}{|y|^{n+s^+  \pl   + \beta}}  \, \dye \dxe\\[5pt]
 \leq & \, \displaystyle C_6(s^+,\pl,\beta, n) r^{n - s^+  \pl -2\beta}.
\end{aligned}
\end{equation}

Similarly to what we did for $I_{2,1}$, for $I_{2,2}$ we have
\begin{equation}\label{97}
I_{2,2} \leq \left\{\begin{aligned} & \, \displaystyle  C_7 \mvbd \left(B_{3r/2} (x_0)\right)  \left(\frac{|x_0|+1}{R_0}\right)^\beta \left( r^{-s^+  \pl  } + \frac{R_0^{n-s^+  \pl}}{|x_0|^{n}} \right)& \, \displaystyle  & \, \displaystyle  \text { if } 0\not\in \overline{B_{3R}(x_0)},\\
&C_7 r^{n-s^+  \pl - 2\beta} & \, \displaystyle  & \, \displaystyle  \text { if } x_0 =0,
\end{aligned}\right.
\end{equation}
where $C_7 = C_7(s^{-}, s^{+}, p^{-}, \pl , \beta, n)>0$.

For $I_{2,3}$, using  \eqref{h37} and \eqref{182}, we obtain
\begin{equation} \label{h42}
I_{2,3} \leq \left(|x_0|+\frac{3r}{2}\right)^\beta \mvbd \left(B_{3r/2} (x_0)\right)\tl (u_- ; x_0  , R , 3r/2) .
\end{equation}

Therefore, from \eqref{h36},  \eqref{97}, and \eqref{h42},
\begin{equation}\label{h44}
\begin{aligned}
 I_{2} \leq & \, \displaystyle  C_8 \max \left\{1, d^{p^{-}-\pl}\right\} \mvbd \left(B_{3r/2} (x_0)\right)  \left(\frac{|x_0|+1}{R_0}  \right)^\beta \left( r^{-s^+  \pl  } + \frac{R_0^{n-s^+  \pl}}{|x_0|^n} \right) \\[5pt]
& \, \displaystyle   + C_8 d^{1-\pl}\left(|x_0|+1\right)^\beta \mvbd \left(B_{3r/2} (x_0)\right)\tl (u_- ; x_0  , R , 3r/2) \quad \text { if } \quad 0\notin \overline{B_{3R}(x_0)},
\end{aligned}
\end{equation}
where $C_8=C_8(s^-,s^+, p^-, \pl, \beta , \Lambda , n )>0$.

And, if $x_0=0$, combining \eqref{h38}, \eqref{97}, and \eqref{h42}, we obtain
\begin{equation}\label{h45}
I_2\leq C_9 \max \left\{1, d^{p^{-}-\pl}\right\} r^{n-2 \beta-s^+  \pl  } + C_9 d^{1-\pl } r^{n-\beta} \tl (u_{-} ; x_0 , R , 3 r / 2),
\end{equation}
where $C_9=C_9(s^-,s^+, p^-, \pl, \beta , \Lambda , n )>0$.

\medskip

{\bf Estimate of $I_{3}$:} We have
\begin{equation*}
\begin{aligned}
I_{3} =& \, \displaystyle  \int_{B_{2 r}(x_{0})}|\nabla u|^{\pl - 2} \nabla u  \nabla \left[ (u+d)^{1-\pl} \psi^{\pl}\right]  \, \dmvbd \\[5pt]
 \leq & \, \displaystyle  (1-\pl) \int_{B_{2r}(x_{0})}|\nabla \ln (u+d)|^{\pl} \psi ^{\pl} \, \dmvbd +C_1 \pl\int _{B_{2r}(x_0)} |\nabla u|^{\pl -1} (u+d)^{1-\pl} \psi ^{\pl -1}\frac{1}{r} \, \dmvbd .
\end{aligned}
\end{equation*}

Employing Young's inequality:
$$
|\nabla u|^{\pl -1} (u+d)^{1-\pl} \psi ^{\pl -1}\frac{1}{r} \leq \frac{\pl-1}{\pl} \epsilon  |\nabla u|^{\pl} (u+d)^{-\pl} \psi ^{\pl} + \frac{1}{\pl} \epsilon ^{-\pl +1}\frac{1}{r^{\pl}}.
$$
Choosing $\epsilon = \frac{1}{2C_1}$, we obtain
\begin{equation}\label{h46}
I_3\leq -\frac{\pl-1}{2} \int_{B_{r}(x_{0})}|\nabla \ln (u+d)|^{\pl}  \, \dmvbd + C_1 (2C_1)^{\pl -1} \frac{ 1}{ r^{\pl} } \mvbd \left( B_{3r/2} (x_0)\right).
\end{equation}

Hence, using \eqref{h25}, \eqref{h44},  \eqref{h46}, and Lemma \ref{180} \ref{66}, the estimate \eqref{h48} follows. Furthermore, employing \eqref{h26}, \eqref{h45}, and \eqref{h46}, we obtain \eqref{h49}.
\end{proof}

As a consequence of Lemma \ref{h51}, we have the following result.
\begin{corollary} \label{h65}
Suppose that {property \ref{144}} holds, and that $n-2\beta -1 +s^->0$, $\beta\geq 0$, and $n-\beta>s^{+} \pl$. Furthermore, let $d>0$ and $0<2r \leq R\leq 1/2$. Assume that $u$ is a weak supersolution of \eqref{4} such that $u \geq 0$ in $B_{R}(x_{0}) \subset \Omega$.

Let $a, d>0$ and  $b>1$, and define
$$
v=\min \left\{\left(\ln \left(\frac{a+d}{u+d} \right)\right)_{+}, \ln b\right\}.
$$

If $0\not\in \overline{B_{3R}(x_{0})}$ and  $0<R_0\leq |x_0|/2$, then there exists a constant $c_1 = c_1 (s^{-}, p^{-}, s^{+}, \pl , \beta, \Lambda, n)>0$ such that
\begin{equation}\label{h63}
\begin{aligned}
 \displaystyle  \fint _{B_{r}(x_{0})} & \left|v-(v)_{B_{r}(x_{0}), \mvbd }\right|^{\pl} \, \dmvbd \\[5pt] 
 \leq & \, \displaystyle  c_1 \max  \left\{ 1 , d^{p^{-}- \pl }\right\} \left( \frac{|x_0|+1}{R_0}\right)^{\beta}\left(r^{\pl-s^+  \pl }+\frac{R_0^{n-s^+  \pl  }}{|x_0|^{n}} r^{\pl}\right) \\[5pt]
& \, \displaystyle   +c_1 +c_1 d^{1-\pl} (|x_0|+1)^{\beta}r^{\pl}\tl (u_- ; x_0  , R , 3r/2).
\end{aligned}
\end{equation}

If $x_0=0$, then there exists a constant $c_2 = c_2 (s^{-}, p^{-}, s^{+}, \pl , \beta, \Lambda, n)>0$ such that
\begin{equation}\label{h64}
\begin{aligned}
 \displaystyle  \fint _{B_{r}(x_{0})} & \left|v-(v)_{B_{r}(x_{0}), \mvbd }\right|^{\pl} \, \dmvbd  \\[5pt]
  \leq & \, \displaystyle  c_2 \max \left\{1, d^{p^{-}-\pl}\right\} r^{\pl-s^+  \pl} + c_2 +c_2 d^{1-\pl } r^{\pl +\beta} \tl (u_{-} ; x_0 , R , 3 r / 2 ).
\end{aligned}
\end{equation}

\end{corollary}
\begin{proof} By Corollary \ref{59}, we have
\begin{equation}\label{h60}
\begin{aligned}
 \displaystyle \fint _{B_{r}(x_{0})} & \left|v-(v)_{B_{r}(x_{0}), \mvbd }\right|^{\pl} \, \dmvbd \\[5pt]
  \leq & \, \displaystyle C_1(\beta,\pl,n) \frac{1}{\mvbd \left( B_r (x_0)\right)} r^{\pl} \int_{B_{r}(x_{0})}|\nabla v|^{\pl} \, \dmvbd \quad \text { if }  \quad 0\notin \overline{B_{3R}(x_0)},
\end{aligned}
\end{equation}
and
\begin{equation}\label{h61}
\begin{aligned}
 \displaystyle \fint _{B_{r}(x_{0})} & \left|v-(v)_{B_{r}(x_{0}), \mvbd }\right|^{\pl} \, \dmvbd \\[5pt]
 \leq & \, \displaystyle  C_2(\beta,\pl,n)  r^{\pl-n+2\beta} \int_{B_{r}(x_{0})}|\nabla v|^{\pl} \, \dmvbd \quad \text { if } \quad x_0=0.
\end{aligned}
\end{equation}

Since $v$ is a truncation of the sum of a constant and $\ln (u+d)$, it holds that
\begin{equation}\label{h62}
\int_{B_{r}(x_{0})}|\nabla v|^{\pl} \, \dmvbd \leq \int_{B_{r}(x_{0})}|\nabla \ln (u+d)|^{\pl} \, \dmvbd .
\end{equation}

The estimates in \eqref{h63} and \eqref{h64} follow by applying Lemma \ref{h51} together with \eqref{h60} - \eqref{h62}.\end{proof}

\section{Local boundedness} \label{186}

We apply the following lemma from real analysis. For the proof of Lemma \ref{84}, see \cite[Lemma 4.1]{dibenedetto1993degenerate}.

\begin{lemma}\label{84}
 Let $(Y_{j})_{j=0}^{\infty}$ be a sequence of positive real numbers such that $Y_{0} \leq c_{0}^{-\frac{1}{\gamma}} b^{-\frac{1}{\gamma^{2}}}$ and $Y_{j+1} \leq c_{0} b^{j} Y_{j}^{1+\gamma}, j=0,1,2, \ldots$, for some constants $c_{0}, b>1$ and $\gamma>0$. Then $\lim _{j \rightarrow \infty} Y_{j}=0$.
\end{lemma}

We now show that weak subsolutions of \eqref{4} are locally bounded. The statement of this result, as given in the introduction, is as follows:


\medskip

\noindent {\bf Theorem \ref{129}.} {\it (Local boundedness). Let $\Omega$ be a  bounded domain. Assume conditions \ref{162} and \ref{163} are satisfied. Also, assume that $\beta \geq 0$ and $n-\beta - 1+      {s^-} >0$.}

{\it Let $u\in W^{1,\pl } _{\loc} (\Omega ; |x|^{-2\beta}) \cap W^{s(x,y),p(x,y)} (\Omega ; |x|^{-\beta})$ be a weak solution of \eqref{4}, satisfying $\tl (u ; x_0, r, \rho) <\infty $ for all balls  $B_r (x_0) \subset B_\rho (x_0) \subset \Omega$. Then there exists a  constant $C=C(s^{+}, p^{-}, \pl , \beta, \Lambda, n)>0$, such that}
\begin{equation}\label{85}
\underset{B_{r/2}(x_{0})}{\esssup } \, u \leq \delta (|x_0|+1)^{\frac{\beta}{\pl-1}} r^{\frac{\pl}{\pl-1}} \tl \left(u_+ ; x_{0}, \frac{r}{2} ,r \right)^{\frac{1}{\pl -1}} +C\delta^{-\frac{(\pl - 1)\kappa}{\pl(\kappa -1)}} \left(\fint_{B_{r}(x_{0})} u_+^{\pl} \, \dmvbd \right)^{\frac{1}{\pl}} +1,
\end{equation}
{\it whenever $B_{r}(x_{0}) \subset \Omega$ with $r \in(0,1/2]$ and $\delta \in(0,1]$. Here,        {\(\kappa = n/(n - \pl)\) if \(1 < \pl < n\), and \(\kappa = 2\) if \(\pl \geq n\).}}

\medskip




\begin{proof} Let $B_{r}(x_{0}) \subset \Omega$ with $r \in(0,1/2]$. For $j=0,1,2, \ldots$, we define
$$
r_j := \frac{r}{2}(1 + 2^{-j}), \quad \tilde{r}_j := \frac{r_j + r_{j+1}}{2}, \quad B_j := B_{r_j}(x_0), \quad \tilde{B}_j := B_{\tilde{r}_j}(x_0).
$$

Let $(\psi_{j})_{j=0}^{\infty} \subset C_{c}^{\infty}(\tilde{B}_{j})$ be a sequence of cutoff functions such that,  for all $j=0,1,2, \ldots$,
$$
0 \leq \psi_{j} \leq 1 \text { in } \tilde{B}_{j},  \quad \psi_{j}=1 \text { in } B_{j+1}, \quad \text {  and } \quad |\nabla \psi_{j}| \leq \frac{C_1(n)2^{j+3}}{r},  
$$
 where $C_1> 1 $.

 For $j=0,1,2, \ldots$ and $k, \tilde{k} \geq 0$, we define 
$$
k_j := k + (1 - 2^{-j}) \tilde{k}, \quad \tilde{k}_j := \frac{k_j + k_{j+1}}{2}, \quad w_j := (u - k_j)_{+}, \quad \tilde{w}_j := (u - \tilde{k}_j)_{+}.
$$

From Lemma \ref{180} \ref{66}, 
\begin{equation}\label{83} 
\begin{aligned}
 \left(\frac{\tilde{k}}{2^{j+2}} \right)^{\frac{\pl (\kappa-1)}{\kappa}} \left(  \fint _{B_{j+1}}  w_{j+1}^{\pl} \, \dmvbd \right)^{\frac{1}{\kappa}}  = & \, \displaystyle  (k_{j+1}-\tilde{k}_{j})^{\frac{\pl(\kappa -1)}{\kappa}} \left(\fint _{B_{j+1}}  w_{j+1}^{\pl} \, \dmvbd \right)^{\frac{1}{\kappa}} \\[5pt]
 \leq & \, \displaystyle  \left[\left(\frac{3}{2} \right)^n C_2(\beta , n)\right]^{\frac{1}{ \kappa}}\left(\fint _{\tilde{B}_{j}}|\tilde{w}_{j} \psi_{j}|^{\pl  \kappa} \, \dmvbd \right)^{\frac{1}{\kappa}} ,
\end{aligned}
\end{equation}
where $\kappa$ is given by \eqref{69}.

By the Sobolev inequality in Corollary \ref{70}, with $C_3=C_3(\beta , \pl , n)>0$, we obtain
\begin{equation} \label{71}
\begin{aligned}
\left(\fint_{ \tilde B_{j}} | \tilde{w}_{j} \psi_{j} | ^{\pl \kappa} \, \dmvbd \right)^{\frac{1}{\kappa}}  \leq & \, \displaystyle  C_3 \frac{\mvbd(B_j)}{\mvbd (\tilde B _j)} \tilde r_j^{\pl} \fint_{B_{j}} \left|  \nabla \left(   \tilde{w}_{j} \psi_{j}  \right) \right|^{\pl} \, \dmvbd  \\
 \leq & \, \displaystyle  C_4(\beta , \pl , n) \left( \underbrace{ r^{\pl} \fint_{B_{j}} \tilde{w}_{j}^{\pl}|\nabla \psi_{j}|^{\pl} \, \dmvbd}_{I_1}+\underbrace{r^{\pl} \fint_{B_{j}} \psi_{j}^{\pl}|\nabla \tilde{w}_{j}|^{\pl} \, \dmvbd}_{I_2} \right),
\end{aligned}
\end{equation}
since, by  Lemma \ref{180} \ref{66}, 
$$
\frac{r_j  }{\tilde r_j}= \frac{2^{j+2} +4}{2^{j+2}  +3}\leq \frac{4}{3}, \quad \tilde r_j \leq r, \quad {\text { and } \quad \tilde B_{j} \subset B_{j}.}
$$ 

\medskip

{\bf Estimate of $I_{1}$:} Using the properties of $\psi_{j}$, we have
\begin{equation}\label{72}
\begin{aligned}
I_{1}= &\, \displaystyle  r^{\pl}\fint_{B_{j}} \tilde{w}_{j}^{\pl}|\nabla \psi_{j}|^{\pl} \, \dmvbd \\[5pt]
\leq  &\, \displaystyle   2^{3\pl}C_1^{\pl}  2^{j \pl} \fint _{B_{j}} w_{j}^{\pl} \, \dmvbd .
\end{aligned}
\end{equation}

\medskip

{\bf Estimate of $I_{2}$:} By Lemma \ref{73}, we obtain
\begin{equation}\label{80}
\begin{aligned}
 I_{2}= & \, \displaystyle  r^{\pl} \fint_{B_{j}} \psi_{j}^{\pl}|\nabla \tilde{w}_{j}|^{\pl} \, \dmvbd \\[5pt]
 \leq & \, \displaystyle  C_{5}(\pl) r^{\pl}\left( \fint_{B_{j}} \tilde{w}_{j}^{\pl}|\nabla \psi_{j}|^{\pl} \, \dmvbd \right.\\[5pt]
& \, \displaystyle  +\frac{1}{\mvbd(B_j)}\int_{B_{j}} \int_{B_{j}} \max \{\tilde{w}_{j}(x), \tilde{w}_{j}(y)\}^{p(x,y)}|\psi_{j}(x)-\psi_{j}(y)|^{p(x,y)} \, \dxu  (x,y) \\[5pt]
& \, \displaystyle   \left. + \frac{1}{\mvbd(B_j)} \int_{\Omega \setminus B_{j} } \int _{B_{j}} \tilde w(y)^{p(x, y)-1} \tilde w(x) \psi _j(x) ^{\pl } \, \dxu (x,y) \right) \\[7pt]
 = & \, \displaystyle  J_{1}+J_{2}+J_{3}
\end{aligned}
\end{equation}

\medskip

{\bf Estimates of $J_{1}$ and $J_{2}$:}  To estimate $J_{2}$, proceeding similarly to estimate (4.5) in \cite[p. 1292]{di2016localfract}, we obtain
\begin{equation*}
\begin{aligned}
J_2\leq & \, \displaystyle  \frac{C_5 r^{\pl}\Lambda}{\mvbd(B_j)}\int_{B_{j}} \int_{B_{j}} (\tilde{w}_{j}(x) + \tilde{w}_{j}(y) )^{p(x,y)} \frac{|\psi_{j}(x)-\psi_{j}(y)|^{p(x,y)} }{|x-y|^{n+s(x,y)p(x,y)} |x|^\beta  |y|^\beta} \,  \dxe \dye\\[5pt]
 \leq & \, \displaystyle  \frac{ 2^{\pl} C_5 r^{\pl}\Lambda}{\mvbd (B_j)}  \int_{B_{j}} \int_{B_{j} \cap \{u>\tilde k_j\}} \tilde{w}_{j}(x) ^{p(x,y)} \frac{|\psi_{j}(x)-\psi_{j}(y)|^{p(x,y)} }{|x-y|^{n+s(x,y)p(x,y)} |x|^\beta |y|^\beta} \,  \dxe \dye\\[7pt]
 \leq & \, \displaystyle  \frac{ 2^{{4\pl}} C_5C_1^{\pl} 2^{j\pl}\Lambda}{\mvbd (B_j)}\int_{B_{j}} \int_{B_{j}} \left[ w _{j}(x) ^{\pl} +\left( 2^{j+2}\frac{w_j(x)}{\tilde k}\right)^{\pl}\right]\frac{1}{|x-y|^{n-(1-s(x,y))p(x,y)} |x|^\beta |y|^\beta} \, \dxe \dye .
\end{aligned}
\end{equation*}
{This follows because}
$$
\frac{w_j}{\tilde k_j - k_j}\geq 1 \quad \text { in } \quad B_j \cap \{u>\tilde k_j\},
$$
and we use the estimates 
$$
|\psi _j(x) - \psi _j (y)|\leq \left(\sup _{\tilde B_j} |\nabla \psi _j| \right)|x-y|, \quad \tilde k_j - k_j= \frac{\tilde k}{2^{j+2}},
$$
 as well as the inequality $a^{p(x,y)}\leq a^{\pl} +1$ for all $a\geq 0$ and $  (x,y)\in \Omega \times \Omega$.

If $x_0\notin B_{3r}(0)$, then  $0\notin \overline{B_r(x_0)}$, and
\begin{equation}\label{78}
\begin{aligned}
J_2 \leq  & \, \displaystyle   C_6 (\pl , \Lambda)\frac{2^{2j\pl}}{\min \{\tilde k ^{\pl} , 1\} (|x_0|-r)^{\beta}} \fint_{B_{j}} w _{j}(x) ^{\pl} |x|^\beta \left[ \int_{B_{2r_j} (x)} \frac{1}{|x-y|^{n-(1-s^+)p^-} } \,  \dye \right] \, \dmvbd (x)\\[5pt]
 \leq & \, \displaystyle  C_7 (s^+ , p^- ,\pl , \Lambda ,n )\frac{2^{2j\pl} }{\min \{\tilde k ^{\pl} , 1\} } \left(\frac{|x_0|+r}{|x_0|-r} \right)^{\beta} r _j ^{(1-s^+)p^-} \fint_{B_{j}} w _{j} ^{\pl} \, \dmvbd \\[7pt]
 \leq & \, \displaystyle  C_7 \frac{2^{2j\pl} 2^\beta}{\min \{\tilde k ^{\pl} , 1\} } r_j^{(1-s^+)p^-} \fint_{B_{j}} w _{j} ^{\pl} \, \dmvbd .
\end{aligned}
\end{equation}

If $x_0\in B_{3r}(0)$, then from $|x-y|<2r_j$ in $B_j$,  we have 
\begin{equation} \label{76}
J_2 \leq  C_6 (\pl , \Lambda)\frac{2^{2j\pl} (2r_j)^{(1-s^+)(p^- -1)}}{\min \{\tilde k ^{\pl} , 1\} } \fint_{B_{j}} w _{j}(x) ^{\pl} |x|^\beta  \left[ \int_{B_{2r_j} (x)} \frac{1}{|x-y|^{n-(1-s^+)} |y|^\beta} \,  \dye \right] \, \dmvbd (x).
\end{equation}
If $x \in B_j \setminus \{0\}$,  then by Lemma \ref{27} \ref{23},       {choosing
$$
R=2|x|>0, \qquad \alpha = 1- s^+, \qquad \tilde \beta = \beta,
$$
we obtain}
\begin{equation} \label{74}
\begin{aligned}
\int_{B_{2r_j} (x)} & \displaystyle \frac{1}{|x-y|^{n-(1-s^+)} |y|^\beta} \,  \dye  \\[5pt]
      {\leq} & \, \displaystyle  \int_{B_{2r_j} (x) \setminus B_{2|x|} (0)} \frac{1}{|x-y|^{n-(1-s^+)} |y|^\beta} \,  \dye + \int_{B_{2|x|} (0)} \frac{1}{|x-y|^{n-(1-s^+)} |y|^\beta} \,  \dye\\[5pt]
\leq & \, \displaystyle   C_8(s^+, \beta , n )\frac{1}{|x|^\beta} \left( r_j^{1-s^+} + |x|^{1-s^+}\right) \\[7pt]
\leq & \, \displaystyle  C_9(s^+, \beta , n ) \frac{1}{|x|^\beta}       {r} ^{1-s^+}.
\end{aligned}
\end{equation}

From \eqref{76} and      {\eqref{74}}, we obtain 
\begin{equation} \label{77}
J_2 \leq  C_{     {10}} \frac{2^{2j\pl} }{\min \{\tilde k ^{\pl}, 1\} } \fint_{B_{j}} w _{j} ^{\pl}  \, \dmvbd , \quad \text { if } \quad x_0\in B_{3r} (0),
\end{equation}
 where $C_{     {10}}:= C_{     {10}}(s^+ , p^- , \pl  , \beta , \Lambda , n)>0$.

To estimate $J_{1}$, we use the estimate for $I_{1}$ in \eqref{72}. Considering  \eqref{78} and \eqref{77}, we get
\begin{equation}\label{79}
J_i\leq C_{     {11}} \frac{2^{2j\pl} }{\min \{\tilde k ^{\pl}, 1\} } \fint_{B_{j}} w _{j} ^{\pl} \, \dmvbd  ,
\end{equation}
for $i=1,2$, where $C_{     {11}}:= C_{     {11}}(s^+ , p^- , \pl  , \beta , \Lambda , n)>0$.

\medskip


{\bf Estimate of $J_{3}$:} We observe that $w_{j}(x)^{\pl} \geq(\tilde{k}_{j}-k_{j})^{\pl-1} \tilde{w}_{j}(x)$, and
$$
\frac{|y-x_{0}|}{|x-y|} \leq \frac{|y-x|+|x-x_{0}|}{|x-y|} \leq 1+\frac{\tilde{r}_{j}}{r_{j}-\tilde{r}_{j}} \leq 2^{j+4},
$$
for $y \in \mathbb{R}^{n} \setminus B_{j}$ and $x \in \tilde{B}_{j} \supset \spt \psi _j$.

 Let $\delta \in (0,1]$. We have
\begin{align}
&J_{3} \leq  \frac{C_5 \Lambda r^{\pl}}{\mvbd(B_j)} \int_{\Omega \setminus B_j } \int _{B_j} \tilde w _j (y)^{p(x, y)-1} \tilde w _j(x) \psi _j(x) ^{\pl } \frac{1}{|x-y|^{n+s(x,y)p(x,y)} |x|^{\beta} |y|^\beta} \, \dxe \dye  \nonumber \\[5pt]
& \, \displaystyle \quad \ \leq C_{     {12}}  2^{j(n+\pl s^+)} r^{\pl} \fint_{B_{j}} |x|^\beta \tilde{w}_{j} \dmvbd  \cdot \underset{x \in B_{ r_j}(x_0)}{\operatorname{ess} \sup }      {\Bigg\{} \int_{\Omega \setminus B_{r_j}(x_0)} \frac{\tilde{w}_{j}(y)^{p(x,y)-1}}{|x_0-y|^{n+s(x,y)p(x,y)}  |y|^\beta} \, \dye      {\Bigg\}} \nonumber \\[5pt]
& \, \displaystyle \quad \  \leq C_{     {12}}(|x_0| +r_j)^\beta 2^{j(n+\pl s^+)} r^{\pl} \tl \left( w_{0} ; x_{0}, \frac{r}{2} , r\right) \fint_{B_{j}} \frac{w_{j}^{\pl}}{(\tilde{k}_{j}-k_{j})^{\pl-1}} \, \dmvbd  \nonumber\\[5pt]
& \, \displaystyle \quad \  \leq C_{     {13}}(|x_0| +1)^\beta\frac{2^{j(n+\pl s^+ + \pl-1)}}{\tilde{k}^{\pl-1}} r^{\pl} \tl \left( w_{0} ; x_{0}, \frac{r}{2} , r\right) \fint_{B_{j}} w_{j}^{\pl} \, \dmvbd \nonumber\\[7pt]
& \, \displaystyle  \quad \ \leq C_{     {13}} 2^{j(n+\pl s^+ +\pl -1)} \delta^{1-\pl} \fint _{B_{j}} w_{j}^{\pl} \dmvbd, \label{81}
\end{align}
where $C_{i}= C_i(s^{+}, \pl ,  \Lambda , n)>0$ for $i=     {12,13}$, and 
$$
\tilde{k} \geq \delta (|x_0|+1)^{\frac{\beta}{\pl-1}} r^{\frac{\pl}{\pl-1}} \tl \left(w_{0} ; x_{0}, \frac{r}{2} ,r \right)^{\frac{1}{\pl -1}} +1.
$$

By applying \eqref{79} and \eqref{81} to \eqref{80}, we obtain
\begin{equation}\label{82}
I_2\leq C_{     {14}}  2^{2j(n+\pl s^+ + \pl-1)}    \delta^{1-\pl} \fint_{B_{j}} w _{j} ^{\pl} \, \dmvbd 
\end{equation}
 where $C_{     {14}}:= C_{     {14}}(s^+ , p^- , \pl  , \beta , \Lambda , n)>1$ and $\delta \in (0,1]$.
 
 \medskip
 
 {\bf We now proceed to complete the proof of the theorem.}

\medskip
 
  By inserting \eqref{83}, \eqref{72}, and \eqref{82} into \eqref{71}, we have
\begin{equation*}
\frac{1}{\tilde k}\left(\fint_{B_{j+1}}w_{j+1}^{\pl} \, \dmvbd \right)^{\frac{1}{\pl}} \leq C_{     {15}}2^{j\frac{2(n+\pl s^++\pl-1)}{\pl}\kappa} \delta^{\frac{1-\pl}{\pl}\kappa} \frac{1}{\tilde k ^{\kappa}} \left(  \fint_{B_{j}} w_{j}^{\pl} \, \dmvbd \right)^{\frac{\kappa}{\pl}} ,
\end{equation*}
where $C_{     {15}}=C_{     {15}}(s^+ , p^- , \pl  , \beta , \Lambda , n) >1$.

Set
$$
Y_{j}:=\left(\fint_{B_{j}} w_{j}^{\pl} \, \dmvbd\right)^{\frac{1}{\pl}}
$$
and
$$
\tilde{k}=\delta (|x_0|+1)^{\frac{\beta}{\pl-1}} r^{\frac{\pl}{\pl-1}}\tl \left(w_{0} ; x_{0}, \frac{r}{2} ,r \right)^{\frac{1}{\pl -1}} +c_{0}^{\frac{1}{\gamma}} b^{\frac{1}{\gamma^{2}}}\left(\fint_{B_{r}(x_{0})} w_{0}^{\pl} \, \dmvbd \right)^{\frac{1}{\pl}} +1,
$$
where
$$
c_{0}=C_{     {15}}\delta^{\frac{1-\pl}{\pl} \kappa}, \quad b=2^{\left[\frac{2(n+\pl s^++\pl-1)}{\pl}+\frac{\kappa-1}{\kappa}\right] \kappa} \quad \text { and } \quad \gamma=\kappa-1 .
$$

Hence,
\begin{equation*}
\frac{Y_{j+1}}{\tilde{k}} \leq c_0 b^j \left(\frac{Y_{j}}{\tilde{k}} \right)^{\kappa} .
\end{equation*}

Moreover, from the definition of $\tilde{k}$,  we have
$$
\frac{Y_{0}}{\tilde{k}} \leq c_{0}^{-\frac{1}{\gamma}} b^{-\frac{1}{\gamma^{2}}} .
$$

Thus, by Lemma \ref{84}, we obtain $Y_{j} \rightarrow 0$ as $j \rightarrow \infty$, which implies
$$
\underset{B_{r/2}(x_{0})}{\esssup } \, u \leq k + \tilde k,
$$
which gives \eqref{85} by choosing $k=0$.\end{proof}

\section{Local Hölder continuity} \label{187}

The following local Hölder continuity result (as stated in the introduction) for weak solutions of \eqref{4} follows from Lemma \ref{183} below.


\medskip

\noindent {\bf Theorem \ref{157}.}  {\it (Local Hölder continuity). Let $\Omega$ be a  bounded domain. Assume that conditions \ref{162}, \ref{163}, and {property \ref{144}} are satisfied. Also, suppose that
$$
\beta \geq 0,  \quad n-2\beta - 1+      {s^-} >0,  \quad n-\beta>s^+\pl, \quad \text { and } \quad (1-s^+)(p^--1)\geq \frac{\pl}{p^-}(\pl- p^-).
$$}

{\it Let $u\in W^{1,\pl } _{\loc} (\Omega ; |x|^{-2\beta}) \cap W^{s(x,y),p(x,y)} (\Omega ; |x|^{-\beta})$ be a weak solution of \eqref{4}, satisfying $\tl (u ; x_0, r, \rho) <\infty $ for all balls  $B_r (x_0) \subset B_\rho (x_0) \subset \Omega$.   Then $u$ is locally Hölder continuous in $\Omega$.}

{\it Moreover, if $B_R(x_0) \subset \Omega$    with $R \in (0, 1/2]$ is a ball such that either  $0 \notin \overline{B_{3R}(x_0)}$ or $x_0 = 0$, then there exist constants $\alpha \in ( 0, \frac{1-s^+}{\pl-1} )$ and $c>0$, depending on ${\sup _{y \in \Omega} |y|}$, $R$, $s^-$, $s^{+}$, $p^{-}$, $\pl$,  $\Lambda$, $\beta$, and $ n$, such that}
\begin{equation*}
\begin{aligned}
\underset{B_{\rho}(x_{0})}{\osc} \, u := & \, \displaystyle \underset{B_{\rho}(x_{0})}{\esssup } \, u-\underset{B_{\rho}(x_{0})}{\essinf } \, u \\[5pt]
\leq & \, \displaystyle c\left(\frac{\rho}{r}\right)^{\alpha}   \left[ r^{\frac{\pl}{\pl - 1}}\tl\left( u ; x_{0}, \frac{r}{2} , r \right) ^{\frac{1}{\pl -1}}+\left(\fint _{B_{ r}(x_{0})}|u|^{\pl} \, \dmvbd \right)^{\frac{1}{\pl}} +1 \right],
\end{aligned}
\end{equation*}
{\it where $B_{2r}(x_0) \subset B_R(x_0)$, $r \in (0, 1/4]$, and $\rho \in (0, r/2]$.}

{\it If $\beta = 0$, the constants $\alpha$ and $c$ depend only on  $s^-$, $s^+$, $p^-$, $\pl$, $\Lambda$, and $n$. In the case $x_0 = 0$, the constants $\alpha$ and $c$ depend only on  $s^-$, $s^+$, $p^-$, $\pl$, $\Lambda$, $\beta$, and $n$.}

\medskip


We prove the next result by arguing similarly to the proof of  \cite[Theorem 5.1]{zbMATH07576867} and \cite[Lemma 5.1]{di2016localfract}.

Choosing $\delta (|x_0|+1)^{\frac{\beta}{\pl -1}} =1$ in Theorem \ref{129}, we obtain
\begin{equation}\label{111}
\underset{B_{r/2}(x_0)}{\esssup} \ |u| \leq r^{\frac{\pl}{\pl -1}}\tl \left(u ; x_{0}, \frac{r}{2} , r\right)^{\frac{1}{\pl -1}}+
C\left(\fint _{B_{r}(x_{0})}|u|^{\pl} \, \dmvbd \right)^{\frac{1}{\pl}} +1,
\end{equation}
and
\begin{equation}\label{112} 
\underset{B_{r/2}(x_{0})}{\osc} \ u  \leq 2r^{\frac{\pl}{\pl -1}}\tl \left(u ; x_{0}, \frac{r}{2} , r\right)^{\frac{1}{\pl -1}}+
2C\left(\fint _{B_{r}(x_{0})}|u|^{\pl} \, \dmvbd \right)^{\frac{1}{\pl}} +2,
\end{equation}
where $C=C(s^+ , p^- , \pl , \beta , \Lambda ,    {\sup _{y \in \Omega} |y|}, n)>0$.

\begin{lemma} \label{183}
Let $\Omega$ be a  bounded domain. Assume that conditions \ref{162}, \ref{163}, and {property \ref{144}} are satisfied. Also, suppose that
$$
\beta \geq 0, \quad n-2\beta - 1+       {s^-} >0,  \quad n-\beta>s^+\pl,  \quad \text { and } \quad (1-s^+)(p^--1)\geq \frac{\pl}{p^-}(\pl- p^-).
$$ 

Let $u\in W^{1,\pl } _{\loc} (\Omega ; |x|^{-2\beta}) \cap W^{s(x,y),p(x,y)} (\Omega ; |x|^{-\beta})$ be a weak solution of \eqref{4} satisfying $\tl (u ; x_0, r, \rho) <\infty $ for all balls  $B_r (x_0) \subset B_\rho (x_0) \subset \Omega$.  Suppose that   $B_R(x_0) \subset \Omega$   is a ball such that either $0 \notin \overline{B_{3R}(x_0)}$ or $x_0 =0$.

Then there exist  constants $\eta \in (0, 1/4]$ and $\alpha \in(0, \frac{1-s^+}{\pl-1})$,  depending only on ${\sup _{y \in \Omega} |y|}$, $R$, $s^-$, $s^+$, $p^-$, $\pl$, $\Lambda$, $\beta$, and $n$, such that the following holds:

Assume that $0 < 2r \leq R \leq 1/2$. Define
$$
r_j := \eta^j \frac{r}{2} \quad \text{and} \quad B_j := B_{r_j}(x_0),  \quad j = 0, 1, 2, \ldots 
$$

Set
\begin{equation}\label{136}
\frac{1}{2} \omega(r_{0}):=r^{\frac{\pl}{\pl -1}}\tl \left(u ; x_{0}, \frac{r}{2} , r\right)^{\frac{1}{\pl -1}}+C\left(\fint _{B_{r}(x_{0})}|u|^{\pl} \, \dmvbd \right)^{\frac{1}{\pl}} +1 ,
\end{equation}
where  $C$ is the constant from \eqref{112}.

For each $j = 1, 2, \ldots$, let
\begin{equation}\label{99}
\omega(r_{j}) :=\left(\frac{r_{j}}{r_{0}}\right)^{\alpha} \omega(r_{0}).
\end{equation}

Then, for all $j = 0, 1, 2, \ldots$,
\begin{equation} \label{98}
\underset{B_{j}}{\osc} \, u \leq \omega(r_{j}).
\end{equation}

If $\beta = 0$, the constants $\eta$ and $\alpha$ depend only on  $s^-$, $s^+$, $p^-$, $\pl$, $\Lambda$, and $n$. In the case $x_0 = 0$, the constants $\eta$ and $\alpha$ depend only on  $s^-$, $s^+$, $p^-$, $\pl$, $\Lambda$, $\beta$, and $n$.
\end{lemma}

\begin{proof}  

 By Lemmas \ref{164}  and  \ref{165}, $u_+$ and $(-u)_+$ are weak subsolutions of \eqref{4}.

 From \eqref{112}, it follows that \eqref{98} holds true for $j=0$. Suppose \eqref{98} holds for all $i=0, \ldots, j$ for some $j \in\{0,1,2, \ldots\}$. To prove \eqref{98} for $j+1$, we proceed by induction.


We observe that either
\begin{equation}\label{100}
\frac{ \mvbd \left( B_{2 r_{j+1}}(x_{0}) \cap \left\{u \geq \underset{B_{j}}{\essinf} \,  u  + \frac{\omega(r_{j})}{2}\right\} \right)}{ \mvbd (  B_{2 r_{j+1}}(x_{0})  ) } \geq \frac{1}{2} ,
\end{equation}
or
\begin{equation}\label{101}
\frac{ \mvbd \left( B_{2 r_{j+1}}(x_{0}) \cap \left\{u \leq \underset{B_{j}}{\essinf} \,  u +\frac{\omega(r_{j})}{2}\right\} \right)}{ \mvbd (  B_{2 r_{j+1}}(x_{0})  ) } \geq \frac{1}{2} 
\end{equation}
holds. Define
\begin{equation}\label{140}
u_{j} :=\left\{\begin{aligned}
&   u-\underset{B_{j}}{\essinf} \, u &   &    \text { if \eqref{100} holds},  \\
&   \omega(r_{j})-\left(u-\underset{B_{j}}{\essinf} \,  u \right) &   &   \text { if \eqref{101} holds}.
\end{aligned} \right.
\end{equation}

Then $u_{j}$ is a weak solution of \eqref{4}. Also, by the induction hypothesis in both cases, $u_{j} \geq 0$ in $B_{j}$, and
\begin{equation}\label{107}
\underset{B_{j}}{\esssup} \, |u_j| \leq 2\omega (r_i), \quad \forall i\in \{0, \ldots, j\}.
\end{equation}
Furthermore, from \eqref{111}, we obtain
\begin{equation}\label{102}
\begin{aligned}
 |u_j|\leq  & \, \displaystyle |u| + \omega (r_j) + \underset{B_{j}}{\esssup} \, |u| \\[5pt]
 \leq & \, \displaystyle |u| + \omega (r_j) + \frac{1}{2}\omega (r_0) \quad \text { in  } \quad \Omega .
\end{aligned} 
\end{equation}
We also have
\begin{equation}\label{103}
\frac{\mvbd \left( B_{2 r_{j+1}}(x_{0}) \cap \left\{u_{j} \geq \frac{\omega(r_{j})}{2}\right\} \right)}{ \mvbd ( B_{2 r_{j+1}}(x_{0}))} \geq \frac{1}{2} . 
\end{equation}

Set 
$$
\epsilon := \eta ^{\frac{\pl}{\pl -1} -\alpha}.
$$

Then the following estimate holds:
\begin{equation}\label{122}
\begin{aligned}
  \left[\epsilon \omega (r_j)\right]^{p^- - \pl} r_{j+1} ^{(1-s^+)\pl}= & \, \displaystyle \left[\eta ^{\frac{\pl}{\pl-1} -\alpha} \left(\frac{r_j}{r_0} \right)^\alpha \omega (r_0)\right]^{p^- - \pl} r_{j+1} ^{(1-s^+)\pl}\\[5pt]
 \leq  & \, \displaystyle \eta ^{j[-\alpha (\pl - p^-) +(1-s^+)\pl]} \eta^{-\frac{(\pl -p^-)\pl}{\pl -1} + \alpha (\pl -p^-) + (1-s^+)\pl} \left( \frac{r}{2}\right)^{(1-s^+)\pl}\\[5pt]
 \leq & \, \displaystyle 2^{-(1-s^+)\pl} \\[7pt]
 \leq & \, \displaystyle 1,
\end{aligned}
\end{equation}
provided that  $(1-s^+)(\pl -1)\geq \pl - p^-$ and $(1-s^+)\pl \geq \alpha (\pl - p^-)$.

In the case where $\frac{(1-s^+)\pl}{\pl -1}\geq \alpha $, we observe that
\begin{equation}\label{124}
\begin{aligned}
  \left[\epsilon \omega (r_j)\right]^{1 - \pl} & \displaystyle \left( \frac{r_{j+1}}{r_j} \right)^{\pl}r_j ^{(1-s^+)\pl}\\[5pt]
  = & \, \displaystyle \left[\eta ^{\frac{\pl}{\pl-1} -\alpha} \left(\frac{r_j}{r_0} \right)^\alpha \omega (r_0)\right]^{1 - \pl} \left( \frac{r_{j+1}}{r_j} \right)^{\pl}r_j ^{(1-s^+)\pl}\\[5pt]
  \leq  & \, \displaystyle \eta ^{j[-\alpha (\pl - 1) +(1-s^+)\pl]} \eta^{-\pl + \alpha (\pl -1) + \pl} \\[7pt]
  \leq & \, \displaystyle 1.
\end{aligned}
\end{equation}

If $(1-s^{+})(\pl-1) \geq\frac{\pl}{p^-} (\pl-p^{-})$ and $(1-s^{+}) p^- \geq \alpha(\pl -p^{-})$, then we also have:
\begin{equation}\label{138}
\begin{aligned}
 \left[\epsilon \omega (r_j)\right]^{p^- - \pl} r_{j+1} ^{(1-s^+)p^-}=& \, \displaystyle \left[\eta ^{\frac{\pl}{\pl-1} -\alpha} \left(\frac{r_j}{r_0} \right)^\alpha \omega (r_0)\right]^{p^- - \pl} r_{j+1} ^{(1-s^+)p^-}\\[5pt]
 \leq  & \, \displaystyle \eta ^{j[-\alpha (\pl - p^-) +(1-s^+)p^-]} \eta^{-\frac{(\pl -p^-)\pl}{\pl -1} + \alpha (\pl -p^-) + (1-s^+)p^-} \left( \frac{r}{2}\right)^{(1-s^+)p^-}\\
 \leq & \, \displaystyle 2^{-(1-s^+)\pl} \\[7pt]
 \leq & \, \displaystyle 1.
\end{aligned}
\end{equation}

\medskip


{\bf Step 1.} Let $\epsilon=\eta^{\frac{\pl}{\pl-1}-\alpha}$, {where}
$$
\frac{1-s^+}{\pl -1}\geq \alpha>0 , \quad  (1-s^+)(p^- -1 )\geq \frac{\pl}{p^-}(\pl - p^-), \quad \text { and } \quad {\eta \leq \frac{1}{36}}.
$$ 
We claim that
\begin{equation}\label{123}
\frac{\mvbd\left(B_{2 r_{j+1}}(x_{0}) \cap\left\{u_{j} \leq 2 \epsilon \omega(r_{j})\right\}\right)}{ \mvbd \left( B_{2 r_{j+1}}(x_{0})\right)} \leq \frac{\hat{C}}{\ln \frac{1}{\eta}}, 
\end{equation}
 for some constant $\hat{C}=\hat C ({\sup _{y \in \Omega} |y|},   R, s^{-}, s^{+}, p^{-}, \pl , \beta, \Lambda, n)>0$. Furthermore, if $\beta = 0$, the constant $\hat{C}$ depends only on  $s^-$, $s^+$, $p^-$, $\pl$, $\Lambda$, and $n$. In the case $x_0 = 0$, the constant $\hat{C}$ depends only on  $s^-$, $s^+$, $p^-$, $\pl$, $\Lambda$, $\beta$, and $n$.
 
 \medskip
 

We set
\begin{equation} \label{127}
\lambda :=\ln \left(\frac{\frac{\omega(r_{j})}{2}+\epsilon \omega(r_{j})}{3 \epsilon \omega(r_{j})} \right)=\ln \left(\frac{\frac{1}{2}+\epsilon}{3 \epsilon}\right) \geq \frac{1}{2}\ln \frac{1}{\epsilon}
\end{equation}
and define
\begin{equation*}
\Theta :=\min \left\{\left(\ln \left(\frac{\frac{\omega(r_{j})}{2}+\epsilon \omega(r_{j})}{u_{j}+\epsilon \omega(r_{j})}\right)\right)_{+}, \lambda \right\} .
\end{equation*}

By \eqref{103}, we have
\begin{equation}\label{104}
\begin{aligned}
\lambda  =& \, \displaystyle \frac{1}{ \mvbd \left( B_{2 r_{j+1}}(x_{0}) \cap \left\{u_{j} \geq \frac{\omega(r_{j})}{2}\right\}\right)} \int_{B_{2 r_{j+1}}(x_{0}) \cap\{u_{j} \geq \frac{\omega(r_{j})}{2}\}} \lambda \, \dmvbd \\[5pt]
 = & \, \displaystyle  \frac{1}{\mvbd \left( B_{2 r_{j+1}}(x_{0}) \cap \left\{u_{j} \geq \frac{\omega(r_{j})}{2}\right\}\right)} \int_{B_{2 r_{j+1}}(x_{0}) \cap\{\Theta=0\}} \lambda \, \dmvbd  \\[5pt]
 \leq & \, \displaystyle  \frac{2}{\mvbd \left( B_{2 r_{j+1}}(x_{0})\right)} \int_{B_{2 r_{j+1}}(x_{0})}(\lambda-\Theta) \, \dmvbd \\[7pt]
 = & \, \displaystyle 2 \left(  \lambda-(\Theta)_{B_{2 r_{j+1}}(x_{0}) , \mvbd} \right),
\end{aligned}
\end{equation}
where 
$$
(\Theta)_{B _{2 r_{j+1}}(x_{0}), \mvbd}=\fint_{B_{2 r_{j+1}(x_{0})}} \Theta \, \dmvbd.
$$ 

Integrating \eqref{104} over the set $B_{2 r_{j+1}}(x_{0}) \cap {\{ \Theta=\lambda\}}$, we get
\begin{equation}\label{128}
\frac{\mvbd \left( B_{2 r_{j+1}}(x_{0}) \cap\{\Theta=\lambda \}\right)}{\mvbd \left( B_{2 r_{j+1}}(x_{0})\right)} \lambda \leq \frac{2}{\mvbd \left( B_{2 r_{j+1}}(x_{0})\right)} \int_{B_{2 r_{j+1}(x_{0})}}\left|\Theta-(\Theta)_{B_{2 r_{j+1}}(x_{0}) , \mvbd }\right| \, \dmvbd . 
\end{equation}

Applying Corollary \ref{h65} with 
$$
r=2r_{j+1}, \quad  R=r_j, \quad a=\frac{\omega(r_{j})}{2}, \quad d=\epsilon \omega(r_{j}), \quad \text { and } \quad b=e^{\lambda},
$$ 
and using \eqref{122},   we obtain 
$$
d^{p^- - \pl} r_{j+1}^{(1-s^+)\pl} \leq 1,
$$
and
\begin{equation}\label{115}
\begin{aligned}
&   \fint _{B_{2 r_{j+1}}(x_{0})}\left|\Theta-(\Theta)_{B_{2 r_{j+1}}(x_{0}) , \mvbd }\right|^{\pl} \, \dmvbd \\[5pt]
&   \qquad  \leq C_1\left\{\left[ \epsilon \omega(r_{j}) \right]^{1-\pl}\left(\frac{r_{j+1}}{r_{j}}\right)^{\pl} r_j^{\pl}\tl(u_{j} ; x_{0}, r_{j} , 3r_{j+1})+1 \right\},
\end{aligned} 
\end{equation}
provided $0\notin \overline{R_{3R}(x_0)}$, where $C_1=C_1({\sup _{y \in \Omega} |y|} , R ,  s^- , s^+ , p^- , \pl, \beta , \Lambda , n)>0$. 

In the case $x_0=0$, we obtain
\begin{equation}\label{116}
\begin{aligned}
&   \fint _{B_{2 r_{j+1}}(x_{0})}\left|\Theta-(\Theta)_{B_{2 r_{j+1}}(x_{0}) , \mvbd }\right|^{\pl} \, \dmvbd \\[5pt]
&   \qquad  \leq C_2\left\{  \left[ \epsilon \omega(r_{j}) \right] ^{1-\pl}\left(\frac{r_{j+1}}{r_{j}}\right)^{\pl} r_j^{\pl+\beta}\tl(u_{j} ; x_{0}, r_{j} , 3r_{j+1})+1 \right\} ,
\end{aligned}
\end{equation}
where $C_2=C_2(s^- , s^+ , p^- , \pl, \beta , \Lambda , n)>0$.

Following the reasoning of estimate (5.6) in  \cite[pp. 1294 - 1295]{di2016localfract}, we claim that
\begin{gather}
r_j ^{\pl}\tl\left(u_{j} ; x_{0}, r_{j} , \ell \right) \leq C_3\left[  \eta^{-\alpha(\pl-1)} \omega(r_{j})^{\pl-1} + r_j^{(1-s^+)\pl} \right] \quad \text { if } \quad 0 \notin \overline{B_{3R} (x_0)}, \label{119} \\[5pt]
r_j ^{\pl + \beta}\tl\left(u_{j} ; x_{0}, r_{j} , \ell \right) \leq C_4  \left[ \eta^{-\alpha(\pl-1)} \omega(r_{j})^{\pl-1} + r_j^{(1-s^+)\pl} \right] \quad \text { if } \quad x_0=0, \label{120}
\end{gather}
with $ 3r_{j+1} \leq \ell \leq r/2$, where $C_3=C_3(   {\sup _{y \in \Omega} |y|} , R ,  s^{-}, s^{+}, p^{-}, \pl , \beta, n)>0$ and $C_4 = C_4 (s^{-}, s^{+}, p^{-}, \pl, \beta, n)>0$.

\medskip

{\bf Now we  prove statements \eqref{119} and \eqref{120}, which appear below between \eqref{113} and \eqref{121}:}

\medskip

\begin{equation}\label{113}
\begin{aligned}
 r_j^{\pl}\tl (u_j ; x_0, r_j , \ell) \leq & \, \displaystyle r_j^{\pl}  \sum_{i=1}^j \int_{B_{i-1} \setminus B_i} \frac{|u_j(x)|^{\pl -1} + 1 }{|x-x_0|^{n+s^+ \pl } } \cdot \frac{1}{|x|^\beta} \, \dxe \\[5pt]
& \, \displaystyle   + \underset{y\in B_{\ell} (x_0)}{\esssup}      {\Bigg\{}  r_j^{\pl} \int_{\Omega \setminus B_0}   
 \frac{|u_j(x)|^{p(x,y) -1}}{|x-x_0|^{n+s(x,y) p(x,y) }} \cdot \frac{1}{|x|^\beta}\, \dxe      {\Bigg\}}.
\end{aligned}
\end{equation}

If $0\notin \overline{B_{3R}(x_0)}$, then $2R<|x|$ in $B_R(x_0)$. By \eqref{107}, we have
\begin{equation}\label{108}
\begin{aligned}
 r_j^{\pl}\int_{B_{i-1} \setminus B_i} \frac{|u_j(x)|^{\pl -1} +1}{|x-x_0|^{n+s^{+} \pl}} \frac{1}{|x|^\beta} \, \dxe \leq & \, \displaystyle  \frac{|     {\partial B_1(0)}|2^{\pl -1}}{s^+ \pl(2R)^\beta} \omega (r_{i-1})^{\pl  - 1 } \left( \frac{r_j}{r_{i}}\right)^{\pl} r_i ^{\pl}r_{i}^{-s^+\pl}\\[5pt]
& \, \displaystyle   +\frac{|     {\partial B_1(0)}|}{s^+ \pl (2R)^\beta}  r_j ^{\pl}  \eta ^{-s^+ \pl} r_{i-1}^{-s^+\pl},
\end{aligned}
\end{equation}
due to 
$$
r_i ^{-s^+ \pl - \beta} - r_{i-1} ^{-s^+\pl - \beta} = \frac{1-\eta ^{s^+ \pl + \beta}}{\eta ^{s^+ \pl + \beta}}r_{i-1} ^{-s^+\pl - \beta}.
$$

Furthermore,  from \eqref{102} and \eqref{136}, we have
\begin{equation}\label{109}
\begin{aligned}
&r_j^{\pl } \int_{\Omega \setminus B_0} \frac{|u_j(x)|^{p(x,y)-1}}{|x-x_0|^{n+s(x,y)p(x,y) }} \cdot \frac{1}{|x|^\beta} \, \dxe \\[5pt]
& \, \displaystyle \qquad \begin{aligned}
 \leq & \, \displaystyle  3^{\pl-1}r_j^{\pl } \int_{\Omega \setminus B_0}  \frac{|u|^{p(x,y)-1}}{|x-x_0|^{n+s(x,y) p(x,y)} }  \cdot \frac{1}{|x|^\beta}  + \frac{\omega (r_j) ^{p(x,y)-1}+\left[\omega (r_0)/2 \right] ^{p(x,y)-1}}{|x-x_0|^{n+s(x,y) p(x,y) } } \cdot \frac{1}{|x|^\beta} \, \dxe\\[7pt]
 \leq & \, \displaystyle  3^{\pl-1}r_j^{\pl } \left(r^{-\frac{\pl}{\pl -1}}\frac{1}{2}\omega (r_0)\right)^{\pl -1}  + 3^{\pl-1}r_j^{\pl } \left[\omega(r_j)^{\pl -1}+ \left(\frac{1}{2}\omega (r_0)\right)^{\pl-1} +2\right]\\[5pt]
& \, \displaystyle  \cdot \int_{\mathbb{R}^n \setminus B_0}   \left( \frac{1}{|x-x_0|^{n+s^- p^-}}+\frac{1}{|x-x_0|^{n+s^+ \pl}} \right) \frac{1}{|x|^\beta} \, \dxe  \qquad \text { if } \qquad y\in B_{\ell} (x_0) .
\end{aligned}
\end{aligned}
\end{equation}

By employing the same argument as in \eqref{105}, we obtain, for \eqref{109},
\begin{equation}\label{110}
\begin{aligned}
 r_j^{\pl } \int_{\Omega \setminus B_0} \frac{|u_j(x)|^{p(x,y)-1}}{|x-x_0|^{n+s(x,y)p(x,y) }} \cdot \frac{1}{|x|^\beta} \, \dxe  \leq & \, \displaystyle  C_5  \left( \frac{r_j}{r_1}\right)^{\pl} r_1 ^{\pl}r_0^{ -\pl} \omega (r_0)^{\pl-1}\\[5pt]
& \, \displaystyle   + C_5\left( \frac{r_j}{r_1}\right)^{\pl} r_1 ^{\pl}r_0^{ -s^+\pl} \left[\omega(r_0)^{\pl -1} +1\right] 
\end{aligned}
\end{equation}
where $C_5= C_5 ({\sup _{y \in \Omega} |y|}  , R ,  s^- , s^+ , p^- , \pl , \beta , n )>0$.

Hence, from \eqref{113}, \eqref{108}, and \eqref{110}, we have
\begin{equation} \label{188}
r_j^{\pl } \tl (u_j ; x_0, r_j, \ell )\leq C_6 \sum _{i=1}^j \left[ \left(\frac{r_j}{r_i} \right)^{\pl} \omega (r_{i-1}) ^{\pl -1} + r_j^{\pl} \eta ^{-s^+ \pl} r_{i-1}^{-s^{+} \pl} \right] \quad \text { if } \quad 0\notin \overline{B_{3R}(x_0)},
\end{equation}
where $C_6 = C_6 (  {\sup _{y \in \Omega} |y|} , R , s^- , s^+ , p^- , \pl , \beta , n )>0$.

On the other hand, following a similar argument to the one used to obtain \eqref{188}, we get
\begin{equation}\label{114}
r_j^{\pl +\beta} \tl (u_j ; x_0, r_j, \ell)\leq C_7\sum _{i=1}^j \left[ \left(\frac{r_j}{r_i} \right)^{\pl} \omega (r_{i-1}) ^{\pl -1} + r_j^{\pl +\beta} \eta ^{-s^{+} \pl-\beta} r_{i-1}^{-s^{+} \pl-\beta} \right] \quad \text { if } \quad x_0 = 0,
\end{equation}
where $C_7= C_7 (s^- , s^+ , p^- , \pl , \beta , n )>0$.

We  now proceed to estimate the right-hand side of \eqref{114}:
\begin{equation}\label{117}
\begin{aligned}
  \sum _{i=1}^j  r_j^{\pl +\beta} r_{i-1}^{-s^{+} \pl-\beta} = & \, \displaystyle r_j^{\pl +\beta} r_{0}^{-(s^{+} \pl+\beta)} \frac{\eta ^{-j(s^{+} \pl+\beta)} -1}{\eta ^{-(s^{+} \pl+\beta)}-1}\\[5pt]
 = & \, \displaystyle  r_j^{\pl +\beta} r_{j-1}^{-(s^{+} \pl+\beta)} \frac{1-\eta ^{j(s^{+} \pl+\beta)} }{1- \eta ^{s^{+} \pl+\beta}} \\[7pt]
 \leq & \, \displaystyle \frac{4^{s^{+} \pl+\beta}}{(\ln 4) (s^{+} \pl+\beta)} \eta ^{s^+ \pl + \beta} r_{j}^{(1-s^+)\pl}.
\end{aligned}
\end{equation}
Here, we use the inequality $(1-\eta^t) \sigma ^t \geq \sigma ^t -1\geq  (\ln \sigma ) t$, which holds if $0<\eta \leq  1/\sigma \leq 1$ and $t\geq 0$.

We also have
\begin{equation}\label{118}
\begin{aligned}
 \displaystyle \sum _{i=1}^j  \left(\frac{r_j}{r_i} \right)^{\pl} & \omega (r_{i-1}) ^{\pl -1}  \\[5pt]
  = & \, \displaystyle  \omega(r_0)^{\pl-1} \left(\frac{r_j}{r_0} \right)^{\alpha(\pl-1)} \sum_{i=1}^j \left(\frac{r_{i-1}}{r_i} \right)^{\alpha(\pl-1)}\left(\frac{r_j}{r_i}\right)^{\pl -\alpha(\pl-1)} \\[5pt]
 = & \, \displaystyle  \omega(r_j) ^{\pl-1} \eta ^{-\alpha(\pl-1)} \sum_{i=0}^{j-1} \eta^{i[ \pl-\alpha(\pl-1)]} \\[5pt]
 \leq & \, \displaystyle  \omega(r_j)^{\pl-1} \frac{\eta ^{-\alpha(\pl-1)}}{1-\eta^{\pl-\alpha(\pl-1)}} \\[5pt]
 \leq & \, \displaystyle  \frac{4^{ \pl -\alpha(\pl -1)}}{(\ln 4)[ \pl -\alpha(\pl -1)]} \eta ^{-\alpha(\pl-1)} \omega(r_j)^{\pl-1}\\[5pt]
 \leq & \, \displaystyle  \frac{4^{ \pl }}{(\ln 4)( \pl -1)} \eta ^{-\alpha(\pl-1)} \omega(r_j)^{\pl-1}.
\end{aligned}
\end{equation}

From \eqref{114}, \eqref{117}, and \eqref{118}, we conclude the proof of \eqref{120}, namely,
\begin{equation} \label{121}
r_j ^{\pl + \beta}\tl\left(u_{j} ; x_{0}, r_{j} , \ell\right) \leq C_4 \left[\eta^{-\alpha(\pl -1)} \omega(r_{j})^{\pl-1} + r_j^{(1-s^+)\pl} \right] \quad \text { if } \quad x_0=0. 
\end{equation}
Similarly, we prove \eqref{119}.

\medskip

{\bf We now  proceed to complete Step 1:}

\medskip

 By \eqref{115} and \eqref{119}, if $0\notin \overline{B_{3R}(x_0)}$, then
\begin{equation}\label{126}
\begin{aligned}
&   \fint _{B_{2 r_{j+1}}(x_{0})}\left|\Theta-(\Theta)_{B_{2 r_{j+1}}(x_{0}) , \mvbd }\right|^{\pl} \, \dmvbd \\[5pt]
&   \qquad \begin{aligned} 
  \leq & \, \displaystyle  C_1 \left\{ C_3 \left[ \epsilon \omega(r_{j})  \right] ^{1-\pl}\left(\frac{r_{j+1}}{r_{j}}\right)^{\pl} \left[ \eta^{-\alpha(\pl-1)} \omega(r_{j})^{\pl-1} + r_j^{(1-s^+)\pl} \right]       {+1} \right\}\\[7pt]
 \leq & \, \displaystyle  C_1\left(2C_3 +      {1} \right),
 \end{aligned}
\end{aligned}
\end{equation}
due to  \eqref{124}.

Similarly, from \eqref{116} and \eqref{120}, we obtain
\begin{equation}\label{125}
\begin{aligned}
&   \fint _{B_{2 r_{j+1}}(x_{0})}\left|\Theta-(\Theta)_{B_{2 r_{j+1}}(x_{0}) , \mvbd }\right|^{\pl} \, \dmvbd \\[5pt]
&   \qquad \begin{aligned} 
 \leq & \, \displaystyle  C_2 \left\{ C_4 \left[ \epsilon \omega(r_{j}) \right] ^{1-\pl}\left(\frac{r_{j+1}}{r_{j}}\right)^{\pl} \left[ \eta^{-\alpha(\pl-1)} \omega(r_{j})^{\pl-1} + r_j^{(1-s^+)\pl} \right]       {+ 1} \right\}\\[7pt]
 \leq & \, \displaystyle  C_2\left(2C_4 +      {1} \right) \quad \text { if } \quad x_0=0.
\end{aligned}
\end{aligned}
\end{equation}

Using these two inequalities, together with \eqref{127}, \eqref{128}, \eqref{115}, and \eqref{116}, we obtain \eqref{123}.

\medskip


{\bf Step 2.} We now  use an iteration argument to obtain \eqref{98} for $i=j+1$. To this end, for every $i=0,1,2, \ldots$, let 
$$
\rho_i := \left(1 + 2^{-i}\right) r_{j+1}, \quad \hat{\rho}_i := \frac{\rho_i + \rho_{i+1}}{2}, \quad B^i := B_{\rho_i}(x_0), \quad \hat{B}^i := B_{\hat{\rho}_i}(x_0).
$$

 Recalling that $\epsilon=\eta^{\frac{\pl}{\pl-1}-\alpha}$, define 
 $$
 k_{i} :=(1+2^{-i}) \epsilon \omega(r_{j}) \quad \text { and } \quad
A^{i}:=B^{i} \cap\{u_{j} \leq k_{i}\}, \quad i=0,1,2, \ldots .
$$

Let 
$$
w_{i}:=(k_{i}-u_{j})_{+}.
$$ 

Consider a sequence $(\psi_{i})_{i=0}^{\infty} \subset C_{c}^{\infty}(\hat{B}^{i})$  such that 
$$
0 \leq \psi_{i} \leq 1 \text { in } \hat{B}^{i}, \quad \psi_{i}=1 \text {  in  } B^{i+1},  \quad \text {  and } \quad |\nabla \psi_{i}| \leq \frac{C_8 2^{i}}{\rho_{i}}\leq \frac{C_82^i}{r_{j+1}} \text {  in } \hat{B}^{i},
$$
 where $C_8=C_8(n)>1$. 

Applying Corollary \ref{70}, we obtain
\begin{equation} \label{130}
\begin{aligned}
 (k_{i}-k_{i+1})^{\pl} \left(\frac{\mvbd \left( A^{i+1}\right)}{\mvbd \left(  B^{i+1}\right)}\right)^{\frac{1}{\kappa}}  \leq & \, \displaystyle \left(\fint_{B^{i+1}} w_{i}^{\kappa \pl} \, \dmvbd \right)^{\frac{1}{\kappa}} \leq C_9 \left(\fint _{B^{i}} w_{i}^{\kappa \pl} \psi_{i}^{\kappa \pl} \, \dmvbd \right)^{\frac{1}{\kappa}}  \\[5pt]
\leq & \, \displaystyle C_{10} \rho_{i}^{\pl} \fint_{B^{i}}\left|\nabla(w_{i} \psi_{i})\right|^{\pl} \, \dmvbd \\[7pt]
\leq & \, \displaystyle 2^{2\pl -1} C _{10} (I+J),
\end{aligned}
\end{equation}
where $C_{h}=C_{h}(\beta , \pl , n)>0$ for  $h=9,10$, and 
$$
I:=r_{j+1}^{\pl} \fint _{B^{i}} w_{i}^{\pl}|\nabla \psi_{i}|^{\pl} \, \dmvbd \quad \text { and } \quad J:= r_{j+1}^{\pl} \fint_{B^{i}}|\nabla w_{i}|^{\pl} \psi_{i}^{\pl} \, \dmvbd .
$$

\medskip

{\bf Estimate of $I$:} Since $u_{j} \geq 0$ in $B_{j}$, we have
\begin{equation}\label{131}
w_{i} \leq k_{i} \leq 2 \epsilon \omega(r_{j}) \quad \text {  in  } \quad B_j,
\end{equation}
and  $ B^{i} \subset B_j$. Thus, 
\begin{equation}\label{132}
I=r_{j+1}^{\pl}\fint_{B^{i}} w_{i}^{\pl}|\nabla \psi_{i}|^{\pl} \, \dmvbd \leq 2^{\pl} C_8 ^{\pl} \left[ \epsilon \omega(r_{j}) \right]^{\pl} 2^{i \pl} \frac{\mvbd \left( A^{i}\right)}{\mvbd \left( B^{i}\right)} .
\end{equation}

\medskip

{\bf Estimate of $J$:} By Lemma \ref{73}, there exists a constant $C_{11}=C_{11}(p^-, \pl , \Lambda)>0$ such that
\begin{equation}\label{133}
\mvbd (B^i)J =r_{j+1}^{\pl}\int_{B^{i}}|\nabla w_{i}|^{\pl} \psi_{i}^{\pl} \, \dmvbd \leq C_{11}(J_{1}+J_{2}+J_{3}),
\end{equation}
where
\begin{gather*}
J_{1}:=r_{j+1}^{\pl} \int_{B^{i}} w_{i}^{\pl}|\nabla \psi_{i}|^{\pl} \, \dmvbd , \\
 J_{2}:=r_{j+1}^{\pl}\int_{B^{i}} \int_{B^{i}} \max \{w_{i}(x), w_{i}(y)\}^{p(x,y)} \frac{|\psi_{i}(x)-\psi_{i}(y)|^{p(x,y)}}{|x-y|^{n+s(x,y)p(x,y)}} \cdot \frac{1}{|x|^\beta |y|^\beta}\, \dxe \dye 
\end{gather*}
and
$$
J_{3}:=\underset{x \in \hat{B}^{i}}{\operatorname{ess} \sup }      {\Bigg\{} r_{j+1}^{\pl}\int_{\Omega \setminus B^{i}} \frac{w_{i}(y)^{p(x,y)-1}}{|x-y|^{n+s(x,y)p(x,y)}} \frac{1}{|y|^\beta}\, \dye \cdot \int_{B^{i}} |x|^\beta w_{i} \psi_{i}^{\pl} \, \dmvbd (x)     {\Bigg\}}.
$$
Here, note that 
$$
|x| \leq \rho_i + |x_0| \leq 2r_{j+1} + |x_0| \quad \text { in } \quad B^i.
$$

From \eqref{132}, we get
\begin{equation}\label{134}
J_{1} \leq C_{12} (\pl ,n)  \left[ \epsilon \omega(r_{j})\right]^{\pl} 2^{i \pl} \mvbd \left(  A^{i}\right).
\end{equation}

\medskip

{\it Next, we  estimate $J_2$:}

\medskip

Assume that $0\notin \overline{B_{3R}(x_0)}$. Using \eqref{131} and the inequalities
\begin{gather*}
     {\max \{  |a| ,   |b| \} \leq |a|+ |b| \quad \text { for  } \quad a,b\in \mathbb{R},}\\[5pt]
     {  |a|^{q_2}\leq  |a|^{q_1}+ |a|^{q_3} \quad \text { for  } \quad a\in \mathbb{R}, 0<q_1\leq q_2 \leq q_3,}\\[5pt]
\left|\psi _i (x)- \psi _i(y)\right|\leq \left(\sup _{\hat B^i} \left|\nabla \psi _i\right| \right)|x-y| \quad \text { in } \quad B^i,
\end{gather*}
 we obtain
\begin{align}
J_2\leq & \, \displaystyle  r_{j+1}^{\pl}\int_{B^{i}} \int_{B^{i}}  \bigl( w_{i}(x) +  w_{i}(y)\bigr)^{p(x,y)}\frac{|\psi_{i}(x)-\psi_{i}(y)|^{p(x,y)}}{|x-y|^{n+s(x,y)p(x,y)}}  \cdot \frac{1}{|x|^\beta |y|^\beta}\, \dxe \dye \nonumber\\[5pt]
\leq & \, \displaystyle      {2^{\pl} r_{j+1}^{\pl}\int_{B^{i}} \int_{B^{i}}   w_{i}(y)^{p(x,y)}\frac{|\psi_{i}(x)-\psi_{i}(y)|^{p(x,y)}}{|x-y|^{n+s(x,y)p(x,y)}}  \cdot \frac{1}{|x|^\beta |y|^\beta}\, \dxe \dye} \nonumber\\[5pt]
 \leq & \, \displaystyle  \frac{ 2 ^{     {\pl}}C_8 ^{\pl} 2^{i \pl} }{(2R)^{2\beta}}  \int_{B^i \cap\left\{u_j \leq k_i\right\}} \int_{B^i} \frac{ [2\epsilon \omega (r_j)]^{p^-} + [2\epsilon \omega (r_j)]^{\pl}}{|x-y|^{n-(1-s^+) p^-}} \, \dxe \dye\nonumber \\[7pt]
 \leq & \, \displaystyle  C_{13}2^{i \pl}  r_{j+1}^{(1-s^+)p^-} \Big\{ [2\epsilon \omega (r_j)]^{p^-} + [2\epsilon \omega (r_j)]^{\pl} \Big\}\mvbd \left(B^i \cap\left\{u_j \leq k_i\right\}\right), \label{135}
\end{align}
since $B^i \subset B_{2\rho _i}(y)$ for  all $y\in B^i$. Here, $C_{13} = C_{13} (R , s^+ , p^- ,\pl , \beta , n  ) >0$.

In the case \(x_0 = 0\), we obtain a similar estimate to that in \eqref{135}, using the argument from \eqref{74}.


\medskip

{\it Now we estimate $J_3$.}

\medskip

Observe that
$$
\underset{x \in \hat{B}^{i}}{\esssup}       {\Bigg\{} r_{j+1}^{\pl}\int_{\Omega \setminus B^{i}} \frac{w_{i}(y)^{p(x,y)-1}}{|x-y|^{n+s(x,y)p(x,y)}} \cdot \frac{1}{|y|^\beta}\, \dye       {\Bigg\}} \leq  2^{(i+4)(n+s^+ \pl )}r_{j+1}^{\pl} \tl\left(w_i ; x_0, r_{j+1} ,3r_{j+1}\right)
$$
which follows from the inequality
\begin{align*}
\frac{1}{|y-x|}  =& \, \displaystyle \frac{1}{|y-x_{0}|} \cdot \frac{|y-x_{0}|}{|y-x|} \leq \frac{1}{|y-x_{0}|}\left(1+\frac{|x-x_{0}|}{|y-x|}\right)  \\[5pt]
 \leq & \, \displaystyle \frac{1}{|y-x_{0}|} \left(1+\frac{\hat{\rho}_{i}}{\rho_{i}-\hat{\rho}_{i}} \right) \leq \frac{2^{i+4}}{|y-x_{0}|},
\end{align*}
where $x \in \hat{B}^{i}$ and $y \in \mathbb{R}^{n} \setminus B^{i}$, and from the facts that $B_{r_{j+1}} (x_0) = B_{j+1} \subset B^i $ and $\hat B ^i \subset B_{3r_{j+1}} (x_0) $.

Recalling  \eqref{120} and \eqref{131}, together with the fact that 
$$
B_{3r_{j+1}} (x_0)\subset B_{r_j} (x_0) \quad
 \text { and } \quad w_i \leq |u_j| + 2\epsilon \omega(r_j) \text {  in } \Omega,
$$
 we further obtain, if $x_0=0$:
$$
\begin{aligned}
r_{j+1}^{\pl} & \tl\left(w_i ; x_0, r_{j+1} , 3r_{j+1} \right) \\[5pt]
  \leq & \, \displaystyle   \underset{y\in B_{r_{j}} (x_0)}{\esssup}      {\Bigg\{} r_{j+1}^{\pl} \int_{B_j \setminus B_{j+1}}   \frac{w_i(x) ^{p(x,y)-1}}{|x-x_0|^{n+s(x,y) p(x,y)}}  \cdot \frac{1}{|x|^\beta} \, \dxe      {\Bigg\}} \\[5pt]
& \, \displaystyle  + \underset{y\in B_{r_{j}} (x_0)}{\esssup}      {\Bigg\{} 2^{\pl -1} r_{j+1}^{\pl} \int_{\Omega \setminus B_{j}}   \frac{[2\epsilon\omega (r_j)] ^{p(x,y)-1}}{|x-x_0|^{n+s(x,y) p(x,y)}} \cdot \frac{1}{|x|^\beta} \, \dxe      {\Bigg\}}  + 2^{\pl -1} r_{j+1}^{\pl} \tl \left(u_i ; x_0, r_j , 3r_{j+1} \right)  \\[7pt]
 \leq & \, \displaystyle  C_{14}(s^-, s^+ , p^-, \pl , \beta ,n)  r_{j+1}^{\pl}\left\{ [ \epsilon \omega(r_j)]^{\pl-1}+  [ \epsilon \omega(r_j)]^{p^--1} \right\}\\[5pt]
& \, \displaystyle   \cdot \left[r_{j+1} ^{-s^+\pl - \beta} + \int_{\Omega \setminus B_j} \left(\frac{1}{|x-x_0|^{n+s^- p^- }} + \frac{1}{|x-x_0|^{n+s^+ \pl  }} \right) \frac{1}{|x|^\beta}\, \dxe \right]\\[5pt]
& \, \displaystyle   + C_{14}r_{j+1}^{\pl} r_j ^{-\pl-\beta}\left\{\eta ^{-\pl} \left[ \epsilon \omega(r_j)\right]^{\pl -1}+r_j^{(1-s^{+}) \pl } \right\}\\[7pt]
  \leq & \, \displaystyle  C_{15}(s^-, s^+ , p^-, \pl , \beta ,n) r_{j+1}^{-\beta} \left\{ \left[ \epsilon \omega(r_j)\right]^{\pl -1} + [ \epsilon \omega(r_j)]^{p^--1} r_{j+1} ^{(1-s^+)\pl}+ \left( \frac{r_{j+1}}{r_j}\right)^{\pl}r_j^{(1-s^+)\pl}\right\}.
  \end{aligned}
$$

Hence, employing also \eqref{138}, if $x_0 = 0$, we have:
\begin{equation}\label{137}
\begin{aligned}
J_3 \leq & \, \displaystyle  C_{15}2^{(i+4)(n+ s^+\pl ) +1} r_{j+1}^{-\beta} (2r_{j+1})^{\beta} \left\{ \left[ \epsilon \omega(r_j)\right]^{\pl -1} + \left( \frac{r_{j+1}}{r_j}\right)^{\pl}r_j^{(1-s^+)\pl}\right\}\int_{B^i} w_i \psi_i^{\pl } \, \dmvbd\\[5pt]
 \leq & \, \displaystyle  C_{15} 2^{(i+4)(n+ s^+\pl ) +1} \left\{ \left[ \epsilon \omega(r_j)\right]^{\pl} + \left( \frac{r_{j+1}}{r_j}\right)^{\pl}r_j^{(1-s^+)\pl} \left[ \epsilon \omega(r_j)\right]\right\} \mvbd\left(B^i \cap\left\{u_j \leq k_i\right\}\right).
\end{aligned}
\end{equation}
For the case $0 \notin \overline{B_{3R}(x_0)}$, using the estimates in \eqref{105} and \eqref{119}, we obtain a result similar to \eqref{137}.

\medskip

{\bf We now proceed to complete the induction argument, and thus conclude the proof of this lemma.}

\medskip

Recalling    from \eqref{124} and \eqref{138} that
$$
\left[\epsilon \omega (r_j)\right]^{1-\pl }\left(\frac{r_{j+1}}{r_j}\right)^{\pl } r_j^{\left(1-s^{+}\right) \pl }\leq 1 \quad \text { and } \quad \left[\epsilon \omega (r_j)\right]^{p^- -\pl }r_{j+1}^{\left(1-s^{+}\right) p^{-}}\leq 1,
$$
then by \eqref{134}, \eqref{135}, and \eqref{137}, we conclude that 
\begin{equation}\label{139}
J=r^{\pl}_{j+1} \fint_{B^{i}}|\nabla w_{i}|^{\pl} \psi_{i}^{\pl} \, \dmvbd \leq C_{16} \left[\epsilon \omega(r_{j}) \right]^{\pl} 2^{i(n+\pl)} \frac{\mvbd \left( A^{i}\right)}{\mvbd\left( B^{i}\right)} ,
\end{equation}
where $C_{16} = C_{16} (  {\sup _{y \in \Omega} |y|}  , R, s^- ,s^{+}, p^{-}, \pl , \beta, \Lambda , n)>0$.

Therefore, from \eqref{130}, \eqref{132}, and \eqref{139}, we get
\begin{equation}
\left(k_{i}-k_{i+1}\right)^{\pl} \left(\frac{\mvbd \left( A^{i+1}\right)}{\mvbd \left(  B^{i+1}\right)}\right)^{\frac{1}{\kappa}} \leq C_{17}  2^{i(n+\pl)} \left[\epsilon \omega(r_{j}) \right]^{\pl} \frac{\mvbd \left( A^{i}\right)}{\mvbd\left( B^{i}\right)} 
\end{equation}
where $C_{17} = C_{17} (  {\sup _{y \in \Omega} |y|}  , R, s^- , s^{+}, p^{-}, \pl , \beta, \Lambda , n)>0$.

Let
$$
Y_{i}=\frac{\mvbd \left( A^{i}\right)}{\mvbd\left( B^{i}\right)} , \quad i=0,1,2, \ldots
$$

Noting that $k_{i}-k_{i+1}=2^{-i-1} \epsilon \omega(r_{j})$ , we get
$$
Y_{i+1} \leq 2^{\pl}C_{17} 2^{i(2 \pl+n) \kappa} Y_{i}^{\kappa}
$$

 By \eqref{123} in {\bf Step 1}, we have
$$
Y_{0} \leq \frac{\hat{C} \left(  {\sup _{y \in \Omega} |y|}  , R, s^{-}, s^{+}, p^{-}, \pl , \beta , \Lambda , n \right)}{\ln \frac{1}{\eta}}.
$$

 Let
$$
c_{0}=2^{\pl}C_{17}, \quad b=2^{(2 \pl+n) \kappa}, \quad \gamma=\kappa - 1 \quad \text { and } \quad \eta_{1}=c_{0}^{-\frac{1}{\gamma}} b^{-\frac{1}{\gamma^{2}}}
$$

By choosing $\eta=\frac{1}{2} \min \{\frac{1}{4}, e^{-\frac{\hat{C}}{\eta_{1}}}\}$, we ensure $Y_{0} \leq \eta_{1}$. Thus,  by Lemma \ref{84}, we deduce that $\lim _{i \rightarrow \infty} Y_{i}=0$, and therefore, $u_{j} \geq \epsilon \omega(r_{j})$ in $B_{j+1}$. Using the definition of $u_{j}$ from \eqref{140}, we obtain
\begin{equation}\label{141}
\underset{B_{j+1}}{\osc} \ u \leq(1-\epsilon) \omega(r_{j})=(1-\epsilon) \eta^{-\alpha} \omega(r_{j+1}) \leq \omega(r_{j+1}), 
\end{equation}
where we have chosen $\alpha \in (0, \frac{1-s^{+}}{\pl-1} ]$  small enough such that
$$
\eta^{\alpha} \geq 1-\epsilon=1-\eta^{\frac{\pl}{\pl-1}-\alpha}.
$$

Thus, inequality  \eqref{141} proves the induction estimate \eqref{98} for $i=j+1$. Hence, the result follows.\end{proof}


\vspace{1.5cm}

\noindent{\bf {Acknowledgment:}}  {I would like to express my sincere gratitude to the referee for the helpful comments and valuable suggestions, which have significantly improved the quality of this article.}

\noindent {\bf Author Contributions:} All authors contributed equally to the writing, review, and editing of this work. All authors have read and approved the final version of the manuscript.

\noindent {\bf Funding:} This work was supported by CNPq - Conselho Nacional de Desenvolvimento Científico e Tecnológico, Grant 153232/2024-2.

\noindent {\bf Data Availability:} No data were used in the research described in this article.

\noindent {\bf Declarations}

\noindent {\bf Conflict of Interest:} The authors declare that there is no conflict of interest.

\noindent {\bf Declaration of Generative AI and AI-Assisted Technologies in the Writing Process:} During the preparation of this work, the author(s) used ChatGPT and Grok-3 to assist with grammar and spelling correction. After using these tools, the author(s) reviewed and edited the content as needed and take full responsibility for the final version of the manuscript.


\bibliographystyle{abbrv}
    \bibliography{ref}

@book{dibenedetto1993degenerate,
  title     = {Degenerate {Parabolic} {Equations}},
  author    = {DiBenedetto, Emmanuele},
  series    = {Universitext},
  publisher = {Springer New York},
  year      = {1993},
  edition   = {1},
  pages     = {XVI, 388}
}

@article{di2016localfract,
  title={Local behavior of fractional {$p-$}minimizers},
  author={Di Castro, Agnese and Kuusi, Tuomo and Palatucci, Giampiero},
  journal={Ann. Inst. H. Poincaré C Anal. Non Linéaire},
  volume={33},
  number={5},
  pages={1279--1299},
  year={2016},
  publisher={Elsevier}
}

@article{chaker2023local,
  title={Local regularity for nonlocal equations with variable exponents},
  author={Chaker, Jamil and Kim, Minhyun},
  journal={Math. Nachr.},
  volume={296},
  number={9},
  pages={4463--4489},
  year={2023},
  publisher={Wiley Online Library}
}

@book{grafakos2014fourier,
  author = {Grafakos, Loukas},
  title = {Classical {Fourier} {Analysis}},
  edition = {3},
  series = {Graduate Texts in Mathematics},
  publisher = {Springer},
  address = {New York, NY},
  year = {2014}
}

@book{stein1993harmonic,
  title={Harmonic analysis: real-variable methods, orthogonality, and oscillatory integrals},
  author={Stein, Elias M and Murphy, Timothy S},
  volume={3},
  year={1993},
  publisher={Princeton University Press}
}

@book{torchinsky1986harmonic,
  author = {Torchinsky, Alberto},
  title = {Real-{Variable} {Methods} in {Harmonic} {Analysis}},
  year = {1986},
  publisher = {Academic Press, Inc.},
  series = {Pure and Applied Mathematics}
}

@article{abdellaoui2016fractional,
  title={On the fractional {$p-$}{Laplacian} equations with weight and general datum},
  author={Abdellaoui, Boumediene and Attar, Ahmed and Bentifour, Rachid},
  journal={Adv. Nonlinear Anal.},
  volume={8},
  number={1},
  pages={144--174},
  year={2016},
  publisher={De Gruyter}
}

@article{ferrari2012radial,
  title={Radial fractional {Laplace} operators and {Hessian} inequalities},
  author={Ferrari, Fausto and Verbitsky, Igor E},
  journal={J. Differential Equations},
  volume={253},
  number={1},
  pages={244--272},
  year={2012},
  publisher={Elsevier}
}

@book{magnus1966formulas,
  title={Formulas and {Theorems} for the {Special} {Functions} of {Mathematical} {Physics}},
  author={Magnus, Wilhelm and Oberhettinger, Fritz and Soni, Raj Pal},
  series={Grundlehren der mathematischen Wissenschaften},
  publisher={Springer Berlin, Heidelberg},
  year={1966},
  edition={1}
}

@article{chani1985weightedpeano,
  author = {Chanillo, Sagun and Wheeden, Richard L},
  title = {Weighted {Poincar\'e} and {Sobolev} {Inequalities} and {Estimates} for {Weighted} {Peano} {Maximal} {Functions}},
  journal = {Amer. J. Math.},
  volume = {107},
  number = {5},
  year = {1985},
  pages = {1191--1226},
  publisher = {JSTOR}
}

@article{ok2023local,
  title={Local {H{\"o}lder} regularity for nonlocal equations with variable powers},
  author={Ok, Jihoon},
  journal={Calc. Var. Partial Differential Equations},
  volume={62},
  number={1},
  pages={32},
  year={2023},
  publisher={Springer}
}

@book{gradshteyn2007table,
  author = {Gradshteyn, I S and Ryzhik, I M},
  editor = {Jeffrey, Alan and Zwillinger, Daniel},
  title = {Table of {Integrals}, {Series}, and {Products}},
  edition = {7},
  publisher = {Academic Press},
  address = {Burlington, MA},
  year = {2007}
}

@book{abramowitz1965formulas,
  editor = {Abramowitz, Milton and Stegun, Irene A},
  title = {Handbook of {Mathematical} {Functions} with {Formulas}, {Graphs}, and {Mathematical} {Tables}},
  publisher = {Dover Publications, Inc.},
  address = {New York, NY},
  year = {1965}
}

@article{diening2008muckenhoupt,
  title={Muckenhoupt weights in variable exponent spaces},
  author={Diening, Lars and H{\"a}st{\"o}, Peter},
  journal={preprint},
  year={2008}
}

@book{dienhar2011lebesgue,
 title = {Lebesgue and {Sobolev} {Spaces} with {Variable} {Exponents}},
  author = {Diening, Lars and Harjulehto, Petteri and H{\"a}st{\"o}, Peter and Ruzicka, Michael},
  series = {Lecture Notes in Mathematics},
  year = {2011},
  publisher = {Springer-Verlag},
  address = {Berlin, Heidelberg},
  pages = {IX, 509},
  edition = {1}
  }

@article{fanzhao2001spaces,
  title={On the spaces {$L^{p (x)}(\Omega)$} and {$W^{m, p (x)}(\Omega)$}},
  author={Fan, Xianling and Zhao, Dun},
  journal={J. Math. Anal. Appl.},
  volume={263},
  number={2},
  pages={424--446},
  year={2001},
  publisher={Elsevier}
}

@article{kovavcik1991spaces,
  title={On spaces {$L^{p(x)}$} and {$W^{1, p(x)}$}},
  author={Kov{\'a}{\v{c}}ik, Ondrej and R{\'a}kosn{\'\i}k, Ji{\v{r}}{\'\i}},
  journal={Czechoslovak Math. J.},
  volume={41},
  number={4},
  pages={592--618},
  year={1991},
  publisher={Institute of Mathematics, Academy of Sciences of the Czech Republic}
}

@article{cheng2020variablefrac,
  title={Variable-order fractional {Sobolev} spaces and nonlinear elliptic equations with variable exponents},
  author={Cheng, Yi and Ge, Bin and Agarwal, Ravi P},
  journal={J. Math. Phys.},
  volume={61},
  number={7},
  year={2020},
  publisher={AIP Publishing}
}

@article{bahrouni2018fract,
 author = {Bahrouni, Anouar and R{\u{a}}dulescu, Vicen{\c{t}}iu D},
 title = {On a new fractional {Sobolev} space and applications to nonlocal variational problems with variable exponent},
 fjournal = {Discrete and Continuous Dynamical Systems. Series S},
 journal = {Discrete Contin. Dyn. Syst., Ser. S},
 issn = {1937-1632},
 volume = {11},
 number = {3},
 pages = {379--389},
 year = {2018},
 language = {English},
 doi = {10.3934/dcdss.2018021},
 zbMATH = {6810426},
 Zbl = {1374.35158}
}

@article{bahrouni2018compfrac,
 author = {Bahrouni, Anouar},
 title = {Comparison and sub-supersolution principles for the fractional {{\(p(x)\)}}-{Laplacian}},
 fjournal = {Journal of Mathematical Analysis and Applications},
 journal = {J. Math. Anal. Appl.},
 issn = {0022-247X},
 volume = {458},
 number = {2},
 pages = {1363--1372},
 year = {2018},
 language = {English},
 doi = {10.1016/j.jmaa.2017.10.025},
 zbMATH = {6813457},
 Zbl = {1378.35053}
}

@article{ho2019fract,
 author = {Ho, Ky and Kim, Yun-Ho},
 title = {A-priori bounds and multiplicity of solutions for nonlinear elliptic problems involving the fractional {{\(p(\cdot)\)}}-{Laplacian}},
 fjournal = {Nonlinear Analysis. Theory, Methods \& Applications. Series A: Theory and Methods},
 journal = {Nonlinear Anal., Theory Methods Appl., Ser. A, Theory Methods},
 volume = {188},
 pages = {179--201},
 year = {2019},
 language = {English},
 doi = {10.1016/j.na.2019.06.001},
 zbMATH = {7134940},
 Zbl = {1425.35041}
}

@article{nezza2012guidefract,
 author = {Di Nezza, Eleonora and Palatucci, Giampiero and Valdinoci, Enrico},
 title = {Hitchhiker's guide to the fractional {Sobolev} spaces},
 fjournal = {Bulletin des Sciences Math{\'e}matiques},
 journal = {Bull. Sci. Math.},
 issn = {0007-4497},
 volume = {136},
 number = {5},
 pages = {521--573},
 year = {2012},
 language = {English},
 doi = {10.1016/j.bulsci.2011.12.004},
 zbMATH = {6064171},
 Zbl = {1252.46023}
}

@article{zbMATH06983323,
 author = {Palatucci, Giampiero},
 title = {The {Dirichlet} problem for the {{\(p\)}}-fractional {Laplace} equation},
 fjournal = {Nonlinear Analysis. Theory, Methods \& Applications. Series A: Theory and Methods},
 journal = {Nonlinear Anal., Theory Methods Appl., Ser. A, Theory Methods},
 volume = {177},
 pages = {699--732},
 year = {2018},
 language = {English},
 doi = {10.1016/j.na.2018.05.004},
 zbMATH = {6983323},
 Zbl = {1404.35212}
}

@book{zbMATH06559661,
 author = {Bucur, Claudia and Valdinoci, Enrico},
 title = {Nonlocal diffusion and applications},
 fseries = {Lecture Notes of the Unione Matematica Italiana},
 series = {Lect. Notes Unione Mat. Ital.},
 issn = {1862-9113},
 volume = {20},
 isbn = {978-3-319-28738-6; 978-3-319-28739-3},
 year = {2016},
 publisher = {Cham: Springer; Bologna: UMI},
 language = {English},
 doi = {10.1007/978-3-319-28739-3},
 zbMATH = {6559661},
 Zbl = {1377.35002}
}

@book{zbMATH06533015,
 author = {Molica Bisci, Giovanni and Radulescu, Vicentiu D. and Servadei, Raffaella},
 title = {Variational methods for nonlocal fractional problems},
 fseries = {Encyclopedia of Mathematics and Its Applications},
 series = {Encycl. Math. Appl.},
 issn = {0953-4806},
 volume = {162},
 isbn = {978-1-107-11194-3; 978-1-316-28239-7},
 year = {2016},
 publisher = {Cambridge: Cambridge University Press},
 language = {English},
 doi = {10.1017/CBO9781316282397},
 zbMATH = {6533015},
 Zbl = {1356.49003}
}

@article{zbMATH05204436,
 author = {Caffarelli, Luis and Silvestre, Luis},
 title = {An extension problem related to the fractional {Laplacian}},
 fjournal = {Communications in Partial Differential Equations},
 journal = {Commun. Partial Differ. Equations},
 issn = {0360-5302},
 volume = {32},
 number = {8},
 pages = {1245--1260},
 year = {2007},
 language = {English},
 doi = {10.1080/03605300600987306},
 zbMATH = {5204436},
 Zbl = {1143.26002}
}

@article{zbMATH02111870,
 author = {Bass, Richard F. and Kassmann, Moritz},
 title = {Harnack inequalities for non-local operators of variable order},
 fjournal = {Transactions of the American Mathematical Society},
 journal = {Trans. Am. Math. Soc.},
 issn = {0002-9947},
 volume = {357},
 number = {2},
 pages = {837--850},
 year = {2005},
 language = {English},
 doi = {10.1090/S0002-9947-04-03549-4},
 zbMATH = {2111870},
 Zbl = {1052.60060}
}

@article{zbMATH02232666,
 author = {Bass, Richard F. and Kassmann, Moritz},
 title = {H{\"o}lder continuity of harmonic functions with respect to operators of variable order},
 fjournal = {Communications in Partial Differential Equations},
 journal = {Commun. Partial Differ. Equations},
 issn = {0360-5302},
 volume = {30},
 number = {8},
 pages = {1249--1259},
 year = {2005},
 language = {English},
 doi = {10.1080/03605300500257677},
 zbMATH = {2232666},
 Zbl = {1087.45004}
}

@article{zbMATH05913780,
 author = {Caffarelli, Luis and Chan, Chi Hin and Vasseur, Alexis},
 title = {Regularity theory for parabolic nonlinear integral operators},
 fjournal = {Journal of the American Mathematical Society},
 journal = {J. Am. Math. Soc.},
 issn = {0894-0347},
 volume = {24},
 number = {3},
 pages = {849--869},
 year = {2011},
 language = {English},
 doi = {10.1090/S0894-0347-2011-00698-X},
 zbMATH = {5913780},
 Zbl = {1223.35098}
}

@article{zbMATH05551486,
 author = {Caffarelli, Luis and Silvestre, Luis},
 title = {Regularity theory for fully nonlinear integro-differential equations},
 fjournal = {Communications on Pure and Applied Mathematics},
 journal = {Commun. Pure Appl. Math.},
 issn = {0010-3640},
 volume = {62},
 number = {5},
 pages = {597--638},
 year = {2009},
 language = {English},
 doi = {10.1002/cpa.20274},
 zbMATH = {5551486},
 Zbl = {1170.45006}
}

@article{zbMATH05952965,
 author = {Caffarelli, Luis and Silvestre, Luis},
 title = {Regularity results for nonlocal equations by approximation},
 fjournal = {Archive for Rational Mechanics and Analysis},
 journal = {Arch. Ration. Mech. Anal.},
 issn = {0003-9527},
 volume = {200},
 number = {1},
 pages = {59--88},
 year = {2011},
 language = {English},
 doi = {10.1007/s00205-010-0336-4},
 zbMATH = {5952965},
 Zbl = {1231.35284}
}

@article{zbMATH05223048,
 author = {Kassmann, Moritz},
 title = {The theory of {De} {Giorgi} for non-local operators},
 fjournal = {Comptes Rendus. Math{\'e}matique. Acad{\'e}mie des Sciences, Paris},
 journal = {C. R., Math., Acad. Sci. Paris},
 issn = {1631-073X},
 volume = {345},
 number = {11},
 pages = {621--624},
 year = {2007},
 language = {English},
 doi = {10.1016/j.crma.2007.10.007},
 url = {comptes-rendus.academie-sciences.fr/mathematique/articles/10.1016/j.crma.2007.10.007/},
 zbMATH = {5223048},
 Zbl = {1129.45006}
}

@article{zbMATH05373494,
 author = {Kassmann, Moritz},
 title = {A priori estimates for integro-differential operators with measurable kernels},
 fjournal = {Calculus of Variations and Partial Differential Equations},
 journal = {Calc. Var. Partial Differ. Equ.},
 issn = {0944-2669},
 volume = {34},
 number = {1},
 pages = {1--21},
 year = {2009},
 language = {English},
 doi = {10.1007/s00526-008-0173-6},
 zbMATH = {5373494},
 Zbl = {1158.35019}
}

@article{zbMATH06435282,
 author = {Kuusi, Tuomo and Mingione, Giuseppe and Sire, Yannick},
 title = {Nonlocal self-improving properties},
 fjournal = {Analysis \& PDE},
 journal = {Anal. PDE},
 issn = {2157-5045},
 volume = {8},
 number = {1},
 pages = {57--114},
 year = {2015},
 language = {English},
 doi = {10.2140/apde.2015.8.57},
 url = {semanticscholar.org/paper/dfdb9a8e49a8fa6367388d6da9de4e1120bc8515},
 zbMATH = {6435282},
 Zbl = {1317.35284}
}

@article{zbMATH06568053,
 author = {Ros-Oton, Xavier and Serra, Joaquim},
 title = {Regularity theory for general stable operators},
 fjournal = {Journal of Differential Equations},
 journal = {J. Differ. Equations},
 issn = {0022-0396},
 volume = {260},
 number = {12},
 pages = {8675--8715},
 year = {2016},
 language = {English},
 doi = {10.1016/j.jde.2016.02.033},
 zbMATH = {6568053},
 Zbl = {1346.35220}
}

@article{zbMATH05049012,
 author = {Silvestre, Luis},
 title = {H{\"o}lder estimates for solutions of integro-differential equations like the fractional {Laplace}},
 fjournal = {Indiana University Mathematics Journal},
 journal = {Indiana Univ. Math. J.},
 issn = {0022-2518},
 volume = {55},
 number = {3},
 pages = {1155--1174},
 year = {2006},
 language = {English},
 doi = {10.1512/iumj.2006.55.2706},
 zbMATH = {5049012},
 Zbl = {1101.45004}
}

@article{zbMATH06330975,
 author = {Di Castro, Agnese and Kuusi, Tuomo and Palatucci, Giampiero},
 title = {Nonlocal {Harnack} inequalities},
 fjournal = {Journal of Functional Analysis},
 journal = {J. Funct. Anal.},
 issn = {0022-1236},
 volume = {267},
 number = {6},
 pages = {1807--1836},
 year = {2014},
 language = {English},
 doi = {10.1016/j.jfa.2014.05.023},
 zbMATH = {6330975},
 Zbl = {1302.35082}
}

@article{zbMATH07511756,
 author = {Chaker, Jamil and Kim, Minhyun},
 title = {Regularity estimates for fractional orthotropic {{\(p\)}}-{Laplacians} of mixed order},
 fjournal = {Advances in Nonlinear Analysis},
 journal = {Adv. Nonlinear Anal.},
 issn = {2191-9496},
 volume = {11},
 pages = {1307--1331},
 year = {2022},
 language = {English},
 doi = {10.1515/anona-2022-0243},
 zbMATH = {7511756},
 Zbl = {1487.35149}
}

@article{zbMATH06710259,
 author = {Cozzi, Matteo},
 title = {Regularity results and {Harnack} inequalities for minimizers and solutions of nonlocal problems: a unified approach via fractional {De} {Giorgi} classes},
 fjournal = {Journal of Functional Analysis},
 journal = {J. Funct. Anal.},
 issn = {0022-1236},
 volume = {272},
 number = {11},
 pages = {4762--4837},
 year = {2017},
 language = {English},
 doi = {10.1016/j.jfa.2017.02.016},
 zbMATH = {6710259},
 Zbl = {1366.49040}
}

@article{zbMATH07796271,
 author = {De Filippis, Cristiana and Mingione, Giuseppe},
 title = {Gradient regularity in mixed local and nonlocal problems},
 fjournal = {Mathematische Annalen},
 journal = {Math. Ann.},
 issn = {0025-5831},
 volume = {388},
 number = {1},
 pages = {261--328},
 year = {2024},
 language = {English},
 doi = {10.1007/s00208-022-02512-7},
 zbMATH = {7796271},
 Zbl = {1532.49035}
}

@article{zbMATH07309248,
 author = {Ding, Mengyao and Zhang, Chao and Zhou, Shulin},
 title = {Local boundedness and {H{\"o}lder} continuity for the parabolic fractional {{\(p\)}}-{Laplace} equations},
 fjournal = {Calculus of Variations and Partial Differential Equations},
 journal = {Calc. Var. Partial Differ. Equ.},
 issn = {0944-2669},
 volume = {60},
 number = {1},
 pages = {45},
 note = {Id/No 38},
 year = {2021},
 language = {English},
 doi = {10.1007/s00526-020-01870-x},
 zbMATH = {7309248},
 Zbl = {1459.35053}
}

@article{zbMATH07576867,
 author = {Garain, Prashanta and Kinnunen, Juha},
 title = {On the regularity theory for mixed local and nonlocal quasilinear elliptic equations},
 fjournal = {Transactions of the American Mathematical Society},
 journal = {Trans. Am. Math. Soc.},
 issn = {0002-9947},
 volume = {375},
 number = {8},
 pages = {5393--5423},
 year = {2022},
 language = {English},
 doi = {10.1090/tran/8621},
 zbMATH = {7576867},
 Zbl = {1496.35118}
}

@article{zbMATH07134362,
 author = {Korvenp{\"a}{\"a}, Janne and Kuusi, Tuomo and Lindgren, Erik},
 title = {Equivalence of solutions to fractional {{\(p\)}}-{Laplace} type equations},
 fjournal = {Journal de Math{\'e}matiques Pures et Appliqu{\'e}es. Neuvi{\`e}me S{\'e}rie},
 journal = {J. Math. Pures Appl. (9)},
 issn = {0021-7824},
 volume = {132},
 pages = {1--26},
 year = {2019},
 language = {English},
 doi = {10.1016/j.matpur.2017.10.004},
 zbMATH = {7134362},
 Zbl = {1426.35220}
}

@article{zbMATH06604339,
 author = {Korvenp{\"a}{\"a}, Janne and Kuusi, Tuomo and Palatucci, Giampiero},
 title = {The obstacle problem for nonlinear integro-differential operators},
 fjournal = {Calculus of Variations and Partial Differential Equations},
 journal = {Calc. Var. Partial Differ. Equ.},
 issn = {0944-2669},
 volume = {55},
 number = {3},
 pages = {29},
 note = {Id/No 63},
 year = {2016},
 language = {English},
 doi = {10.1007/s00526-016-0999-2},
 zbMATH = {6604339},
 Zbl = {1346.35214}
}

@article{zbMATH06610795,
 author = {Korvenp{\"a}{\"a}, Janne and Kuusi, Tuomo and Palatucci, Giampiero},
 title = {H{\"o}lder continuity up to the boundary for a class of fractional obstacle problems},
 fjournal = {Atti della Accademia Nazionale dei Lincei. Classe di Scienze Fisiche, Matematiche e Naturali. Serie IX. Rendiconti Lincei. Matematica e Applicazioni},
 journal = {Atti Accad. Naz. Lincei, Cl. Sci. Fis. Mat. Nat., IX. Ser., Rend. Lincei, Mat. Appl.},
 issn = {1120-6330},
 volume = {27},
 number = {3},
 pages = {355--367},
 year = {2016},
 language = {English},
 doi = {10.4171/RLM/739},
 zbMATH = {6610795},
 Zbl = {1344.35159}
}

@article{zbMATH06637333,
 author = {Korvenp{\"a}{\"a}, Janne and Kuusi, Tuomo and Palatucci, Giampiero},
 title = {A note on fractional supersolutions},
 fjournal = {Electronic Journal of Differential Equations (EJDE)},
 journal = {Electron. J. Differ. Equ.},
 issn = {1072-6691},
 volume = {2016},
 pages = {9},
 note = {Id/No 263},
 year = {2016},
 language = {English},
 zbMATH = {6637333},
 Zbl = {1515.35319}
}

@article{zbMATH06798251,
 author = {Korvenp{\"a}{\"a}, Janne and Kuusi, Tuomo and Palatucci, Giampiero},
 title = {Fractional superharmonic functions and the {Perron} method for nonlinear integro-differential equations},
 fjournal = {Mathematische Annalen},
 journal = {Math. Ann.},
 issn = {0025-5831},
 volume = {369},
 number = {3-4},
 pages = {1443--1489},
 year = {2017},
 language = {English},
 doi = {10.1007/s00208-016-1495-x},
 zbMATH = {6798251},
 Zbl = {1386.35028}
}

@article{zbMATH06442824,
 author = {Kuusi, Tuomo and Mingione, Giuseppe and Sire, Yannick},
 title = {Nonlocal equations with measure data},
 fjournal = {Communications in Mathematical Physics},
 journal = {Commun. Math. Phys.},
 issn = {0010-3616},
 volume = {337},
 number = {3},
 pages = {1317--1368},
 year = {2015},
 language = {English},
 doi = {10.1007/s00220-015-2356-2},
 zbMATH = {6442824},
 Zbl = {1323.45007}
}

@article{zbMATH06699542,
 author = {Lindgren, Erik},
 title = {H{\"o}lder estimates for viscosity solutions of equations of fractional {{\(p\)}}-{Laplace} type},
 fjournal = {NoDEA. Nonlinear Differential Equations and Applications},
 journal = {NoDEA, Nonlinear Differ. Equ. Appl.},
 issn = {1021-9722},
 volume = {23},
 number = {5},
 pages = {18},
 note = {Id/No 55},
 year = {2016},
 language = {English},
 doi = {10.1007/s00030-016-0406-x},
 zbMATH = {6699542},
 Zbl = {1380.35097}
}

@article{zbMATH07309168,
 author = {Nowak, Simon},
 title = {Higher {H{\"o}lder} regularity for nonlocal equations with irregular kernel},
 fjournal = {Calculus of Variations and Partial Differential Equations},
 journal = {Calc. Var. Partial Differ. Equ.},
 issn = {0944-2669},
 volume = {60},
 number = {1},
 pages = {37},
 note = {Id/No 24},
 year = {2021},
 language = {English},
 doi = {10.1007/s00526-020-01915-1},
 zbMATH = {7309168},
 Zbl = {1509.35087}
}

@article{zbMATH07665927,
 author = {Nowak, Simon},
 title = {Regularity theory for nonlocal equations with {VMO} coefficients},
 fjournal = {Annales de l'Institut Henri Poincar{\'e} C Analyse Non Lin{\'e}aire},
 journal = {Ann. Inst. Henri Poincar{\'e} C Anal. Non Lin{\'e}aire},
 issn = {0294-1449},
 volume = {40},
 number = {1},
 pages = {61--132},
 year = {2023},
 language = {English},
 doi = {10.4171/AIHPC/37},
 zbMATH = {7665927},
 Zbl = {1510.35358}
}

@article{zbMATH06535137,
 author = {Bae, Jongchun},
 title = {Regularity for fully nonlinear equations driven by spatial-inhomogeneous nonlocal operators},
 fjournal = {Potential Analysis},
 journal = {Potential Anal.},
 issn = {0926-2601},
 volume = {43},
 number = {4},
 pages = {611--624},
 year = {2015},
 language = {English},
 doi = {10.1007/s11118-015-9488-z},
 zbMATH = {6535137},
 Zbl = {1385.47018}
}

@misc{arXiv:1505.05498,
 author = {Bae, Jongchun and Kassmann, Moritz},
 title = {Schauder estimates in generalized {H{\"o}lder} spaces},
 year = {2015},
 howpublished = {Preprint, {arXiv}:1505.05498 [math.{AP}] (2015)},
 url = {https://arxiv.org/abs/1505.05498},
 arXiv = {arXiv:1505.05498}
}

@article{zbMATH07058421,
 author = {Kim, Minhyun and Kim, Panki and Lee, Jaehun and Lee, Ki-Ahm},
 title = {Boundary regularity for nonlocal operators with kernels of variable orders},
 fjournal = {Journal of Functional Analysis},
 journal = {J. Funct. Anal.},
 issn = {0022-1236},
 volume = {277},
 number = {1},
 pages = {279--332},
 year = {2019},
 language = {English},
 doi = {10.1016/j.jfa.2018.11.011},
 zbMATH = {7058421},
 Zbl = {1481.35102}
}

@article{zbMATH07046606,
 author = {De Filippis, Cristiana and Palatucci, Giampiero},
 title = {H{\"o}lder regularity for nonlocal double phase equations},
 fjournal = {Journal of Differential Equations},
 journal = {J. Differ. Equations},
 issn = {0022-0396},
 volume = {267},
 number = {1},
 pages = {547--586},
 year = {2019},
 language = {English},
 doi = {10.1016/j.jde.2019.01.017},
 zbMATH = {7046606},
 Zbl = {1412.35041}
}

@article{zbMATH01586360,
 author = {Acerbi, Emilio and Mingione, Giuseppe},
 title = {Regularity results for a class of functionals with non-standard growth},
 fjournal = {Archive for Rational Mechanics and Analysis},
 journal = {Arch. Ration. Mech. Anal.},
 issn = {0003-9527},
 volume = {156},
 number = {2},
 pages = {121--140},
 year = {2001},
 language = {English},
 doi = {10.1007/s002050100117},
 zbMATH = {1586360},
 Zbl = {0984.49020}
}

@article{zbMATH02191022,
 author = {Acerbi, Emilio and Mingione, Giuseppe},
 title = {Gradient estimates for the {{\(p(x)\)}}- {Laplacian} system},
 fjournal = {Journal f{\"u}r die Reine und Angewandte Mathematik},
 journal = {J. Reine Angew. Math.},
 issn = {0075-4102},
 volume = {584},
 pages = {117--148},
 year = {2005},
 language = {English},
 doi = {10.1515/crll.2005.2005.584.117},
 zbMATH = {2191022},
 Zbl = {1093.76003}
}

@article{zbMATH06603041,
 author = {Byun, Sun-Sig and Ok, Jihoon},
 title = {On {{\(W^{1,q(\cdot)}\)}}-estimates for elliptic equations of {{\(p(x)\)}}-{Laplacian} type},
 fjournal = {Journal de Math{\'e}matiques Pures et Appliqu{\'e}es. Neuvi{\`e}me S{\'e}rie},
 journal = {J. Math. Pures Appl. (9)},
 issn = {0021-7824},
 volume = {106},
 number = {3},
 pages = {512--545},
 year = {2016},
 language = {English},
 doi = {10.1016/j.matpur.2016.03.002},
 zbMATH = {6603041},
 Zbl = {1344.35058}
}

@article{zbMATH06591396,
 author = {Byun, Sun-Sig and Ok, Jihoon and Ryu, Seungjin},
 title = {Global gradient estimates for elliptic equations of {{\(p(x)\)}}-{Laplacian} type with {BMO} nonlinearity},
 fjournal = {Journal f{\"u}r die Reine und Angewandte Mathematik},
 journal = {J. Reine Angew. Math.},
 issn = {0075-4102},
 volume = {715},
 pages = {1--38},
 year = {2016},
 language = {English},
 doi = {10.1515/crelle-2014-0004},
 zbMATH = {6591396},
 Zbl = {1347.35095}
}

@article{zbMATH01270231,
 author = {Coscia, Alessandra and Mingione, Giuseppe},
 title = {H{\"o}lder continuity of the gradient of {{\(p(x)\)}}-harmonic mappings},
 fjournal = {Comptes Rendus de l'Acad{\'e}mie des Sciences. S{\'e}rie I. Math{\'e}matique},
 journal = {C. R. Acad. Sci., Paris, S{\'e}r. I, Math.},
 issn = {0764-4442},
 volume = {328},
 number = {4},
 pages = {363--368},
 year = {1999},
 language = {English},
 doi = {10.1016/S0764-4442(99)80226-2},
 zbMATH = {1270231},
 Zbl = {0920.49020}
}

@article{zbMATH05146777,
 author = {Fan, Xianling},
 title = {Global {{\(C^{1,\alpha}\)}} regularity for variable exponent elliptic equations in divergence form},
 fjournal = {Journal of Differential Equations},
 journal = {J. Differ. Equations},
 issn = {0022-0396},
 volume = {235},
 number = {2},
 pages = {397--417},
 year = {2007},
 language = {English},
 doi = {10.1016/j.jde.2007.01.008},
 zbMATH = {5146777},
 Zbl = {1143.35040}
}

@article{zbMATH01279360,
 author = {Fan, Xianling and Zhao, Dun},
 title = {A class of {De} {Giorgi} type and {H{\"o}lder} continuity},
 fjournal = {Nonlinear Analysis. Theory, Methods \& Applications},
 journal = {Nonlinear Anal., Theory Methods Appl.},
 issn = {0362-546X},
 volume = {36},
 number = {3},
 pages = {295--318},
 year = {1999},
 language = {English},
 doi = {10.1016/S0362-546X(97)00628-7},
 zbMATH = {1279360},
 Zbl = {0927.46022}
}

@article{zbMATH05636583,
 author = {Foondun, Mohammud},
 title = {Heat kernel estimates and {Harnack} inequalities for some {Dirichlet} forms with non-local part},
 fjournal = {Electronic Journal of Probability},
 journal = {Electron. J. Probab.},
 issn = {1083-6489},
 volume = {14},
 pages = {314--340},
 year = {2009},
 language = {English},
 doi = {10.1214/EJP.v14-604},
 url = {https://eudml.org/doc/230452},
 zbMATH = {5636583},
 Zbl = {1190.60069}
}

@article{zbMATH05544580,
 author = {Barlow, Martin T. and Bass, Richard F. and Chen, Zhen-Qing and Kassmann, Moritz},
 title = {Non-local {Dirichlet} forms and symmetric jump processes},
 fjournal = {Transactions of the American Mathematical Society},
 journal = {Trans. Am. Math. Soc.},
 issn = {0002-9947},
 volume = {361},
 number = {4},
 pages = {1963--1999},
 year = {2009},
 language = {English},
 doi = {10.1090/S0002-9947-08-04544-3},
 zbMATH = {5544580},
 Zbl = {1166.60045}
}

@article{zbMATH05769037,
 author = {Chen, Zhen-Qing and Kumagai, Takashi},
 title = {A priori {H{\"o}lder} estimate, parabolic {Harnack} principle and heat kernel estimates for diffusions with jumps},
 fjournal = {Revista Matem{\'a}tica Iberoamericana},
 journal = {Rev. Mat. Iberoam.},
 issn = {0213-2230},
 volume = {26},
 number = {2},
 pages = {551--589},
 year = {2010},
 language = {English},
 doi = {10.4171/RMI/609},
 zbMATH = {5769037},
 Zbl = {1200.60065}
}

@article{zbMATH06191309,
 author = {Chen, Zhen-Qing and Kim, Panki and Song, Renming and Vondra{\v{c}}ek, Zoran},
 title = {Boundary {Harnack} principle for {{\({{\Delta}} +{{\Delta}} ^{{{\alpha}} /2}\)}}},
 fjournal = {Transactions of the American Mathematical Society},
 journal = {Trans. Am. Math. Soc.},
 issn = {0002-9947},
 volume = {364},
 number = {8},
 pages = {4169--4205},
 year = {2012},
 language = {English},
 doi = {10.1090/S0002-9947-2012-05542-5},
 zbMATH = {6191309},
 Zbl = {1271.31006}
}

@article{zbMATH06871997,
 author = {Athreya, Siva and Ramachandran, Koushik},
 title = {Harnack inequality for non-local {Schr{\"o}dinger} operators},
 fjournal = {Potential Analysis},
 journal = {Potential Anal.},
 issn = {0926-2601},
 volume = {48},
 number = {4},
 pages = {515--551},
 year = {2018},
 language = {English},
 doi = {10.1007/s11118-017-9646-6},
 zbMATH = {6871997},
 Zbl = {1429.60062}
}

@article{zbMATH05938378,
 author = {Chen, Zhen-Qing and Kim, Panki and Song, Renming},
 title = {Heat kernel estimates for {{\({{\Delta}} +{{\Delta}} ^{{{\alpha}} /2}\)}} in {{\({C}^{1, 1}\)}} open sets},
 fjournal = {Journal of the London Mathematical Society. Second Series},
 journal = {J. Lond. Math. Soc., II. Ser.},
 issn = {0024-6107},
 volume = {84},
 number = {1},
 pages = {58--80},
 year = {2011},
 language = {English},
 doi = {10.1112/jlms/jdq102},
 zbMATH = {5938378},
 Zbl = {1229.60089}
}

@article{zbMATH06122084,
 author = {Chen, Zhen-Qing and Kim, Panki and Song, Renming and Vondra{\v{c}}ek, Zoran},
 title = {Sharp {Green} function estimates for {{\({{\Delta}} + {{\delta}} ^{{{\alpha}} /2}\)}} in {{\(C^{1,1}\)}} open sets and their applications},
 fjournal = {Illinois Journal of Mathematics},
 journal = {Ill. J. Math.},
 issn = {0019-2082},
 volume = {54},
 number = {3},
 pages = {981--1024},
 year = {2010},
 language = {English},
 zbMATH = {6122084},
 Zbl = {1257.31002}
}

@article{zbMATH07502820,
 author = {Biagi, Stefano and Dipierro, Serena and Valdinoci, Enrico and Vecchi, Eugenio},
 title = {Mixed local and nonlocal elliptic operators: regularity and maximum principles},
 fjournal = {Communications in Partial Differential Equations},
 journal = {Commun. Partial Differ. Equations},
 issn = {0360-5302},
 volume = {47},
 number = {3},
 pages = {585--629},
 year = {2022},
 language = {English},
 doi = {10.1080/03605302.2021.1998908},
 url = {hdl.handle.net/11585/879388},
 zbMATH = {7502820},
 Zbl = {1486.35412}
}

@article{zbMATH07400634,
 author = {Biagi, Stefano and Vecchi, Eugenio and Dipierro, Serena and Valdinoci, Enrico},
 title = {Semilinear elliptic equations involving mixed local and nonlocal operators},
 fjournal = {Proceedings of the Royal Society of Edinburgh. Section A. Mathematics},
 journal = {Proc. R. Soc. Edinb., Sect. A, Math.},
 issn = {0308-2105},
 volume = {151},
 number = {5},
 pages = {1611--1641},
 year = {2021},
 language = {English},
 doi = {10.1017/prm.2020.75},
 zbMATH = {7400634},
 Zbl = {1473.35622}
}

@article{zbMATH07544282,
 author = {Dipierro, Serena and Proietti, Lippi Edoardo and Valdinoci, Enrico},
 title = {Linear theory for a mixed operator with {Neumann} conditions},
 fjournal = {Asymptotic Analysis},
 journal = {Asymptotic Anal.},
 issn = {0921-7134},
 volume = {128},
 number = {4},
 pages = {571--594},
 year = {2022},
 language = {English},
 doi = {10.3233/ASY-211718},
 zbMATH = {7544282},
 Zbl = {1506.35248}
}

@article{zbMATH07752594,
 author = {Dipierro, Serena and Lippi, Edoardo Proietti and Valdinoci, Enrico},
 title = {(Non)local logistic equations with {Neumann} conditions},
 fjournal = {Annales de l'Institut Henri Poincar{\'e} C. Analyse Non Lin{\'e}aire},
 journal = {Ann. Inst. Henri Poincar{\'e} C Anal. Non Lin{\'e}aire},
 issn = {0294-1449},
 volume = {40},
 number = {5},
 pages = {1093--1166},
 year = {2023},
 language = {English},
 doi = {10.4171/AIHPC/57},
 zbMATH = {7752594},
 Zbl = {1527.35436}
}

@article{zbMATH07514705,
 author = {Dipierro, Serena and Ros-Oton, Xavier and Serra, Joaquim and Valdinoci, Enrico},
 title = {Non-symmetric stable operators: regularity theory and integration by parts},
 fjournal = {Advances in Mathematics},
 journal = {Adv. Math.},
 issn = {0001-8708},
 volume = {401},
 pages = {100},
 note = {Id/No 108321},
 year = {2022},
 language = {English},
 doi = {10.1016/j.aim.2022.108321},
 zbMATH = {7514705},
 Zbl = {1490.45011}
}

@article{zbMATH06766506,
 author = {Abdellaoui, B. and Bentifour, R.},
 title = {Caffarelli-Kohn-Nirenberg type inequalities of fractional order with applications},
 fjournal = {Journal of Functional Analysis},
 journal = {J. Funct. Anal.},
 issn = {0022-1236},
 volume = {272},
 number = {10},
 pages = {3998--4029},
 year = {2017},
 language = {English},
 doi = {10.1016/j.jfa.2017.02.007},
 zbMATH = {6766506},
 Zbl = {1373.49007}
}

@article{zbMATH06566683,
 author = {Abdellaoui, Boumediene and Medina, Mar{\'{\i}}a and Peral, Ireneo and Primo, Ana},
 title = {The effect of the {Hardy} potential in some {Calder{\'o}n}-{Zygmund} properties for the fractional {Laplacian}},
 fjournal = {Journal of Differential Equations},
 journal = {J. Differ. Equations},
 issn = {0022-0396},
 volume = {260},
 number = {11},
 pages = {8160--8206},
 year = {2016},
 language = {English},
 doi = {10.1016/j.jde.2016.02.016},
 zbMATH = {6566683},
 Zbl = {1386.35422}
}

@article{zbMATH06293381,
 author = {Abdellaoui, Boumediene and Peral, Ireneo and Primo, Ana},
 title = {A remark on the fractional {Hardy} inequality with a remainder term},
 fjournal = {Comptes Rendus. Math{\'e}matique. Acad{\'e}mie des Sciences, Paris},
 journal = {C. R., Math., Acad. Sci. Paris},
 issn = {1631-073X},
 volume = {352},
 number = {4},
 pages = {299--303},
 year = {2014},
 language = {English},
 doi = {10.1016/j.crma.2014.02.003},
 url = {comptes-rendus.academie-sciences.fr/mathematique/articles/10.1016/j.crma.2014.02.003/},
 zbMATH = {6293381},
 Zbl = {1295.26018}
}

@article{zbMATH05817574,
 author = {Frank, Rupert L. and Lieb, Elliott H. and Seiringer, Robert},
 title = {Hardy-Lieb-Thirring inequalities for fractional {Schr{\"o}dinger} operators},
 fjournal = {Journal of the American Mathematical Society},
 journal = {J. Am. Math. Soc.},
 issn = {0894-0347},
 volume = {21},
 number = {4},
 pages = {925--950},
 year = {2008},
 language = {English},
 doi = {10.1090/S0894-0347-07-00582-6},
 zbMATH = {5817574},
 Zbl = {1202.35146}
}

@article{zbMATH07837279,
 author = {Bahrouni, Anouar and Missaoui, Hlel and Ounaies, Hichem},
 title = {On the fractional {Musielak}-{Sobolev} spaces in {{\(\mathbb{R}^d\)}}: embedding results \& applications},
 fjournal = {Journal of Mathematical Analysis and Applications},
 journal = {J. Math. Anal. Appl.},
 issn = {0022-247X},
 volume = {537},
 number = {1},
 pages = {33},
 note = {Id/No 128284},
 year = {2024},
 language = {English},
 doi = {10.1016/j.jmaa.2024.128284},
 zbMATH = {7837279},
 Zbl = {1540.35192}
}

@article{zbMATH08016974,
 author = {Antonini, Carlo Alberto and Cozzi, Matteo},
 title = {Global gradient regularity and a {Hopf} lemma for quasilinear operators of mixed local-nonlocal type},
 fjournal = {Journal of Differential Equations},
 journal = {J. Differ. Equations},
 issn = {0022-0396},
 volume = {425},
 pages = {342--382},
 year = {2025},
 language = {English},
 doi = {10.1016/j.jde.2025.01.030},
 zbMATH = {8016974}
}

@article{zbMATH07794358,
 author = {Byun, Sun-Sig and Kumar, Deepak and Lee, Ho-Sik},
 title = {Global gradient estimates for the mixed local and nonlocal problems with measurable nonlinearities},
 fjournal = {Calculus of Variations and Partial Differential Equations},
 journal = {Calc. Var. Partial Differ. Equ.},
 issn = {0944-2669},
 volume = {63},
 number = {2},
 pages = {48},
 note = {Id/No 27},
 year = {2024},
 language = {English},
 doi = {10.1007/s00526-023-02631-2},
 zbMATH = {7794358},
 Zbl = {1531.35101}
}

@article{zbMATH07460689,
 author = {Dipierro, Serena and Valdinoci, Enrico},
 title = {Description of an ecological niche for a mixed local/nonlocal dispersal: an evolution equation and a new {Neumann} condition arising from the superposition of {Brownian} and {L{\'e}vy} processes},
 fjournal = {Physica A},
 journal = {Physica A},
 issn = {0378-4371},
 volume = {575},
 pages = {20},
 note = {Id/No 126052},
 year = {2021},
 language = {English},
 doi = {10.1016/j.physa.2021.126052},
 zbMATH = {7460689},
 Zbl = {1528.60037}
}

@book{turesson2000potentweig,
 author = {Turesson, Bengt Ove},
 title = {Nonlinear potential theory and weighted {Sobolev} spaces},
 fseries = {Lecture Notes in Mathematics},
 series = {Lect. Notes Math.},
 issn = {0075-8434},
 volume = {1736},
 isbn = {3-540-67588-4},
 year = {2000},
 publisher = {Berlin: Springer},
 language = {English},
 doi = {10.1007/BFb0103908},
 zbMATH = {1479421},
 Zbl = {0949.31006}
}

@book{heinonen2006nonlinpot,
  author = {Heinonen, Juha and Kilpel{\"a}inen, Tero and Martio, Olli},
  title = {Nonlinear {Potential} {Theory} of {Degenerate} {Elliptic} {Equations}},
  publisher = {Dover Publications, Inc.},
  address = {Mineola, New York},
  year = {2006}
}

@book{zbMATH07647941,
 author = {Leoni, Giovanni},
 title = {A first course in fractional {Sobolev} spaces},
 fseries = {Graduate Studies in Mathematics},
 series = {Grad. Stud. Math.},
 issn = {1065-7338},
 volume = {229},
 isbn = {978-1-4704-6898-9; 978-1-4704-7253-5; 978-1-4704-7252-8},
 year = {2023},
 publisher = {Providence, RI: American Mathematical Society (AMS)},
 language = {English},
 doi = {10.1090/gsm/229},
 zbMATH = {7647941},
 Zbl = {1517.46001}
}

@book{zbMATH07575925,
 author = {Edmunds, D. E. and Evans, W. D.},
 title = {Fractional {Sobolev} spaces and inequalities},
 fseries = {Cambridge Tracts in Mathematics},
 series = {Camb. Tracts Math.},
 issn = {0950-6284},
 volume = {230},
 isbn = {978-1-00-925463-2; 978-1-00-925462-5},
 year = {2023},
 publisher = {Cambridge: Cambridge University Press},
 language = {English},
 doi = {10.1017/9781009254625},
 zbMATH = {7575925},
 Zbl = {1505.46003}
}

@article{zbMATH01308968,
 author = {Samko, Stefan G.},
 title = {Convolution type operators in {{\(L^{p(x)}\)}}},
 fjournal = {Integral Transforms and Special Functions},
 journal = {Integral Transforms Spec. Funct.},
 issn = {1065-2469},
 volume = {7},
 number = {1-2},
 pages = {123--144},
 year = {1998},
 language = {English},
 doi = {10.1080/10652469808819191},
 zbMATH = {1308968},
 Zbl = {0934.46032}
}

@article{zbMATH05159612,
 author = {Guan, Qing-Yang},
 title = {Integration by parts formula for regional fractional {Laplacian}},
 fjournal = {Communications in Mathematical Physics},
 journal = {Commun. Math. Phys.},
 issn = {0010-3616},
 volume = {266},
 number = {2},
 pages = {289--329},
 year = {2006},
 language = {English},
 doi = {10.1007/s00220-006-0054-9},
 zbMATH = {5159612},
 Zbl = {1121.60051}
}

@article{zbMATH06582990,
 author = {Warma, Mahamadi},
 title = {The fractional {Neumann} and {Robin} type boundary conditions for the regional fractional {{\(p\)}}-{Laplacian}},
 fjournal = {NoDEA. Nonlinear Differential Equations and Applications},
 journal = {NoDEA, Nonlinear Differ. Equ. Appl.},
 issn = {1021-9722},
 volume = {23},
 number = {1},
 pages = {46},
 note = {Id/No 1},
 year = {2016},
 language = {English},
 doi = {10.1007/s00030-016-0354-5},
 zbMATH = {6582990},
 Zbl = {1338.35484}
}

@article{zbMATH06406292,
 author = {Warma, Mahamadi},
 title = {The fractional relative capacity and the fractional {Laplacian} with {Neumann} and {Robin} boundary conditions on open sets},
 fjournal = {Potential Analysis},
 journal = {Potential Anal.},
 issn = {0926-2601},
 volume = {42},
 number = {2},
 pages = {499--547},
 year = {2015},
 language = {English},
 doi = {10.1007/s11118-014-9443-4},
 zbMATH = {6406292},
 Zbl = {1307.31022}
}

@article{zbMATH02038439,
 author = {Mingione, Giuseppe},
 title = {Bounds for the singular set of solutions to non linear elliptic systems},
 fjournal = {Calculus of Variations and Partial Differential Equations},
 journal = {Calc. Var. Partial Differ. Equ.},
 issn = {0944-2669},
 volume = {18},
 number = {4},
 pages = {373--400},
 year = {2003},
 language = {English},
 doi = {10.1007/s00526-003-0209-x},
 zbMATH = {2038439},
 Zbl = {1045.35024}
}

@book{zbMATH01061233,
 author = {Mal{\'y}, Jan and Ziemer, William P.},
 title = {Fine regularity of solutions of elliptic partial differential equations},
 fseries = {Mathematical Surveys and Monographs},
 series = {Math. Surv. Monogr.},
 issn = {0076-5376},
 volume = {51},
 isbn = {0-8218-0335-2},
 year = {1997},
 publisher = {Providence, RI: American Mathematical Society},
 language = {English},
 zbMATH = {1061233},
 Zbl = {0882.35001}
}

\end{document}